\documentclass[12pt,reqno,oneside]{amsart}

\usepackage{amsmath,amsthm,amssymb}
\usepackage{graphicx,xcolor}

\usepackage{amssymb}
\usepackage{amsmath,amsfonts,amsthm}
\usepackage{amstext,latexsym,bm}
\usepackage{graphicx}
\usepackage{amsmath}

\usepackage{amsfonts}
\usepackage{amsthm,amscd}
\usepackage{graphicx}
\usepackage{amsmath}
\usepackage{amssymb}
\usepackage{amstext}

\numberwithin{equation}{section}

\theoremstyle{plain}
\newtheorem{maintheorem}{Theorem}

\newtheorem{theorem}{Theorem }[section]
\newtheorem{proposition}[theorem]{Proposition}
\newtheorem{lemma}[theorem]{Lemma}
\newtheorem{corollary}[theorem]{Corollary}

\theoremstyle{definition} \theoremstyle{remark}
\newtheorem{remark}[theorem]{Remark}
\newtheorem{example}[theorem]{Example}
\newtheorem{definition}[theorem]{Definition}

\newcommand{\vep}{\varepsilon}

\newcommand{\cP}{\mathcal{P}}
\newcommand{\cL}{\mathcal{L}}

\newcommand{\cG}{\mathcal{G}}

\newcommand{\cF}{\mathcal{F}}

\newcommand{\cU}{\mathcal{U}}

\renewcommand{\le}{\leqslant}
\renewcommand{\leq}{\leqslant}
\renewcommand{\ge}{\geqslant}
\renewcommand{\geq}{\geqslant}

\usepackage{mathrsfs}
\usepackage[colorlinks=true,linkcolor=blue, urlcolor=red, citecolor=blue, backref]{hyperref}	

\title{A Dynamical Approach to Non-Extensive Thermodynamics}

\author{Artur O. Lopes and Paulo Varandas}

\begin{document}

\maketitle

\begin{abstract}
    We develop a non-extensive thermodynamic formalism for the one-sided shift on a
finite alphabet, inspired by Tsallis' generalization of Boltzmann entropy in statistical physics.
We introduce notions of $q$-entropy, $q$-pressure and $q$-transfer operators
which extend the classical thermodynamic formalism when $q=1$. We prove a Bowen-type relation linking the $q$-pressure with a $(2-q)$-Ruelle
transfer operator and show that $q$-equilibrium states correspond to classical
equilibrium states for a related potential. We establish existence and
uniqueness of $q$-equilibrium states for Lipschitz potentials, prove
differentiability of the $q$-pressure, and obtain variational principles for
both the $q$-pressure and a related asymptotic pressure.
Finally, we study cohomological equations associated with $(2-q)$-transfer
operators and prove differentiable dependence of their solutions on the
potential, yielding an alternative construction of eigenfunctions for classical
Ruelle operators.
We also propose an approach to non-extensive thermodynamics using non-additive formalisms.
\end{abstract}

\section{Introduction} \label{int}

\subsection{Non-extensive entropy} \label{int}

Entropy of invariant measures is a fundamental concept used in dynamical systems to quantify the rate of information production as a system evolves over time. Within the thermodynamic formalism, this notion of metric entropy plays a central role, allowing to relate the topological complexity of the dynamical system with the largest possible complexity offered by the invariant measures  by means of the classical variational principle for topological pressure. In particular, if $\sigma:\Omega \to \Omega$ denotes the usual shift acting on the  space  $\Omega = \{1,2,...,d\}^\mathbb{N}$ and $A: \Omega \to \mathbb R$ is a continuous potential then 
\begin{equation}
    \label{eq:varprincip}
	P(A) =\sup \Big\{ h(\mu) +\int A \,d \mu \,:\, \mu \in \mathcal{M}_{\text{inv}}(\sigma)\Big\}
\end{equation}
where $P(A)$ denotes the topological pressure of $A$ and $h(\mu)$ denotes the Kolmogorov-Shannon entropy of the $\sigma$-invariant probability measure $\mu$, computed through dynamically defined partitions which are weighted according to the Boltzman entropy function $H$ which, to each probability vector $p=(p_1,p_2, \dots, p_d)$ associates the value
\begin{equation}
    \label{eq:Shannoninf}
    H (p) =\,\sum_{i=1}^d  - p_i \log (p_i) 
\end{equation}
(see e.g. \cite{Walters} for the definitions and proof).
In a non-dynamical framework, the previous expression coincides with Shannon information, sometimes referred to as \emph{static entropy}. 

\smallskip
Several scientific papers consider concepts of {\it entropy} which differ from the Kolmogorov-Shannon entropy ~\eqref{eq:Shannoninf}, whose emphasis is to provide a bias on rare events (see e.g. \cite{ABH,BCMV,Cat,GLM,Sa,UmaTsa,
TaVe} and the discussion therein).
In fact, in Physics' literature it is somewhat common to consider a parameterized concepts of entropy, where for certain parameters 
entropy becomes \emph{non-additive}  even if dealing with  independent systems (cf. \cite[p.75 equation (6)]{Abe2}).  In order to elaborate further on that, let us recall some definitions considered in the classical literature on the non-extensive entropy theory (in a non-dynamically framework).
Given $q>0$, the \emph{$q$-entropy} of the probability vector $p=(p_1,p_2,...,p_n)$, introduced by Havrda and Charvat  \cite{HaCha} and Tsallis \cite{Tsa1}, is defined as
\begin{equation} \label{T11}
  H_q (p) =\frac{1}{1-q}  \,(\sum_{i=1}^d  p_i^{q}-1) =\frac{1}{1-q}  \sum_{i=1}^d p_i   ( p_i^{q-1}-1)=\sum_{i=1}^d p_i \log_q \Big( \frac{1}{p_i}\Big)\geqslant 0,
\end{equation}
and, for each $q\neq 1$, the function
\begin{equation}
\label{T1} 
\mathbb R_+\ni u \mapsto \log_q (u) =\frac{1}{1-q} ( u^{1-q}-1)
\end{equation}
is called the \emph{$q$-log function}.
The case $q=1$ which corresponds to Kolmogorov-Shannon entropy.
 It is clear that if $q\neq 1$ and 
 and one wants to maximize $\frac{1}{1-q}  \,(\sum_i  p_i^{q}-1)$  among different  probability vectors $p$ then there exists a {\it bias} which is not present whenever $q=1$. Indeed,
if $q<1$ then $p_i^q > p_i$ and the $q$-entropy will enhance the relative importance of rare events, and 
if $q>1$ one will get the opposite. 
We refer the reader to  \cite{UmaTsa}
for a discussion. Moreover,
for a fixed probability vector $p$, the limit  of $H_q(p),$ as $q\to 1$ is the classical  
Boltzman entropy  $H(p)$.
In what follows we will  emphasize some of the properties of the $q$-entropy function and then establish a comparison with the classical thermodynamic formalism.

\smallskip
Throughout this paper it will be a standard assumption that $q>0$, in which case it makes sense to maximize entropy (see Figure~\ref{fig1}). 
\begin{figure}[h!]
	\centering
	\includegraphics[scale=0.31]{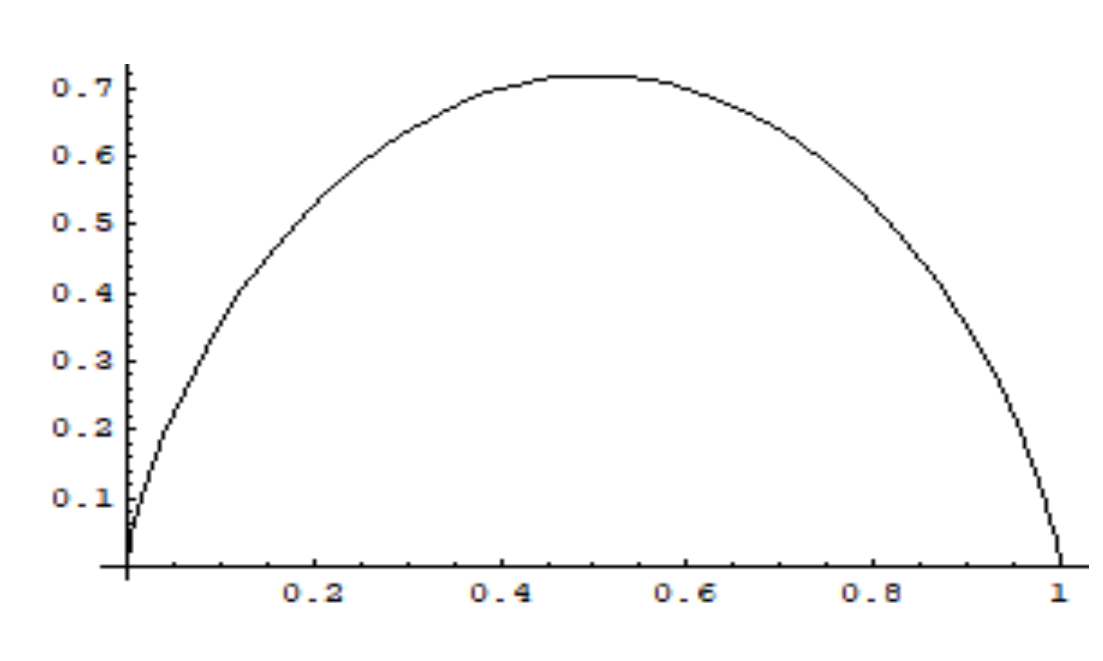}
	\qquad
		\includegraphics[scale=0.31]{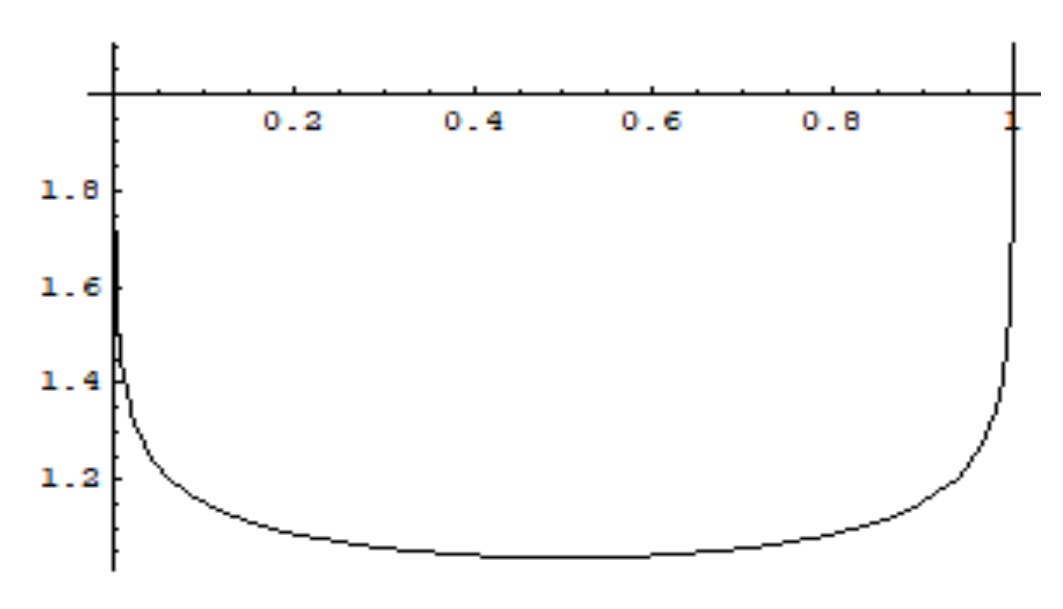}
	\caption{Graph of the function $p_1 \to H_q (p_1,1-p_1)$ when $q=0.9$ (left) and $q=-0.1$ (right) }
	\label{fig1}
\end{figure}
In fact, 
for $q>0$, the $q$-entropy function $p\mapsto H_q(p)$ defined by ~\eqref{T11} is concave as a
function of the probability vector $p=(p_1,p_2,..,p_n)$ because $\frac{d^2 }{d p_j^2}
\frac{1}{1-q} p_j^q =\frac{1}{1-q} q (q-1) p_j^{q-2}<0$ and the finite sum of concave functions is concave.  Moreover,
there is a crucial difference between $q$-log functions for $q=1$ and $q\neq 1$, with striking impact on the corresponding notions of extensive and non-extensive entropies, respectively.
Indeed, for every $q\neq 1$ and $a,b>0$,
 $$ \log_q (a b)
 \;=\; \log_q (a) + \log_q (b) + (1-q) \, \log_q (a) \, \log_q (b)
   \;\neq\;  \log_q (a) + \log_q (b)
 $$
the joint independent probability vector $rs$ obtained from probability vectors $r,s$ satisfies \emph{the non-extensive} relation
 \begin{equation} \label{had}H_q(rs) = H_q(r) + H_q(s) + (1-q)\, H_q(r)\, H_q(s)
\end{equation}
and one recovers additivity above if and only if $q=1$ (cf. \cite{Abe}), which corresponds to the classical (extensive) Boltzman entropy.
We refer the reader to Section~\ref{apen}, to \cite{Yam} or \cite[Appendix, page 84]{Oka} for for more details on $q$-exponential, $q$-logarithmic functions and $q$-entropies.

\smallskip
In this general 
non-dynamical framework,
given constants $q>0$, $\beta \in \mathbb{R}$ and a potential $A: \{1,2,...,n\} \to \mathbb{R}$,
one can define a notion of \emph{$q$-pressure} by the variational relation
\begin{equation} \label{fipre} P_q (\beta \,A) = \sup_p  \Big\{ H_q(p)  +\beta\, \,\sum_{j=1}^d p_j a_j   \Big\},
\end{equation}
where $a_j = A(j)$ and the supremum is taken over all probability vectors $p=(p_1,p_2,...,p_n)$,
in correspondence to what Umarov and Tsallis \cite{UmaTsa} refer to as the first choice of variational problem. A second  alternative, where the pressure is defined 
replacing $\sum_{j=1}^d p_j a_j$ by $\sum_{j=1}^d p_j^q a_j $
will not be considered here.

The latter finds a dual in the Statistical Mechanics literature, where in ~\eqref{fipre} the $q$-pressure is minus the Helmholtz free energy and
the values $a_j$ 
correspond to the values of minus the Hamiltonian.
However, in contrast to the classical pressure function, the pressure function
$\mathbb R \ni \beta \mapsto P_q (\beta \,A)$ defined by ~\eqref{fipre} is not globally convex nor concave and several problems can occur when considering  large ranges of $\beta$
(cf. Remark \ref{rtu}).

\subsection{Non-extensive thermodynamic formalism for the shift} \label{int22}

\medskip
In this paper we  aim to develop a the dynamical non-extensive thermodynamic formalism in parallel with the extensive (classical) setting,
looking for possible matches and discrepancies with some of the well known results in the classical thermodynamic formalism developed in the last decades
(see e.g. \cite{Bala,BecS, Bow, CL3, CS, DWY, FV, GKLM, GL, Keller, KW,LMMS, LR, LM1, PP,Rue,RW,Sar} just to mention a few).
Our work should be considered as
an initial attempt to explore and propose alternative answers to some questions that in our view are fundamental in the theory.
For that reason, we investigate a dynamical version of this class of problems for the one-sided shift
on a finite alphabet, where the extensive thermodynamic formalism (which corresponds to the classical  Thermodynamic Formalism, using the Kolmogorov-Shannon entropy) is extremely well understood.
Although both theories have similar motivation, the non-extensive thermodynamic formalism presents challenges and several conceptual differences
in respect to the classical (extensive) thermodynamic formalism and non-additive thermodynamic formalism which we will now discuss in detail.

\medskip
Consider the shift $\sigma:\Omega \to \Omega$ acting on the symbolic  space  $\Omega = \{1,2,...,d\}^\mathbb{N}$.
We denote by $\mathcal{M}_{\text{inv}}(\sigma)$ the set of $\sigma$-invariant probability measures and by
 $\mathcal{G}$ the set of classical (or extensive) equilibrium states associated to Lipschitz continuous potentials.
These measures are all ergodic, are singular with respect to each other and 
parameterized by the  associated Jacobian function $J_\mu$
(cf. \eqref{JJ} and \cite{GKLM} for more details).
Moreover, probability measures in $\mathcal{G}$ have nice ergodic properties, as they are mixing,
have exponential decay of correlations for H\"older continuous observables and satisfy large deviation principles
(see e.g. \cite{Bow,L3}).
Furthermore, by Rohklin's formula, the Kolmogorov-Shannon entropy $h(\mu)$ associated to $\mu \in \mathcal{G}$ is given by
\begin{equation}\label{eq:Rohklin}
h(\mu)=- \int \log J_\mu\, d \mu.
\end{equation}
The extensive pressure $P(A)$ of
a Lipschitz
continuous potential $A:\Omega \to \mathbb{R}$  satisfies the classical variational principles
\begin{equation} \label{Pia4}
	P(A) =\sup \Big\{ h(\mu) +\int A \,d \mu \,:\, \mu \in \mathcal{M}_{\text{inv}}(\sigma)\Big\}
	= \sup \Big\{ h(\mu) + \int A \, d \mu   \,:\, \mu \in \mathcal{G}\Big\}
\end{equation}
(the second equality in \eqref{Pia4} is due to the upper-semicontinuity of the entropy map and that any invariant probability can be weak$^*$ approximated, and in entropy, by a probability measure in $\mathcal{G}$, cf. \cite{L3}).
It is also well known that there exists a unique
equilibrium state $\mu_A \in \cG$ 
which maximizes the previous expression.

\smallskip
Given a Lipschitz continuous potential $A: \Omega \to \mathbb R$, Ruelle's theorem relates special features of the equilibrium state $\mu_A$ with leading eigenvalue and eigenfunction for the \emph{Ruelle transfer operator}  $\mathcal{L}_{A} : C^0(\Omega,\mathbb R) \to C^0(\Omega,\mathbb R)$, which is the bounded linear operator given by
\begin{equation}\label{AA}
\mathcal{L}_{A}(f) = \sum_{a=1}^d e^{A(ax)}\, f(ax)
\end{equation}
for every continuous function $f:\Omega \to \mathbb{R}$, where
$ax \in \Omega$ denotes the sequence in $\sigma^{-1}(x)$ starting with the symbol $a$.
Namely, the equilibrium state $\mu_A$ is such that $\mu_A= h_A \nu_A$ where $\cL_A h_A =\lambda_A h_A$, $\cL_A^* \nu_A =\lambda_A \nu_A $
and $\lambda_A=e^{P(A)}>0$ is the simple leading eigenvalue of $\mathcal L_A$
(see e.g. \cite{Bala}).
We proceed to extend the above concepts to the non-extensive framework.

\bigskip
\subsubsection*{Non-extensive thermodynamic quantities}
Let us now define the concepts of 
$q$-entropy of an invariant measure in $\mathcal G$,  $q$-pressure function and $q$-equilibrium states, in a way that one can recover the classical framework as the limit of such quantities as $q\to 1$.
Inspired by Rokhlins formula  we define, for each $q>0$,  $q \neq 1$,  the \emph{$q$-entropy} of a probability measure $\mu\in \cG$ as
\begin{equation}
\label{HQ}
H_q(\mu) =  \int \log_q \Big(\frac{1}{J}\Big) d \mu= \frac{1}{1-q}  \int \big(J^{q-1} -1\big)  d \mu.
\end{equation}
 We extend the concept of $q-$entropy for probability measures $\mu \in \mathcal{M}_{\text{inv}}(\sigma)$ and show that the $q$-entropy function is concave and upper semi-continuous (see Definition \ref{uod}, and Lemmas \ref{le:Hconcave} and \ref{uup}).
In the extensive framework Boltzman entropy is concave and  the Kolmogorov-Shannon entropy is an affine function on the convex set of  $\sigma$-invariant probability measures (see   Theorem 8.1 page 183 in \cite{Walters}).
 However,  the dynamical $q$-entropy is concave when $0<q \leq 1$ (cf. Example \ref{maio}).

In view of the second equality in ~\eqref{Pia4}, given a continuous potential $A: \Omega \to \mathbb R$ we defined the dynamical \emph{$q$-pressure function} of  $A$ by the variational relation
\begin{equation}  \label{PP234}
P_q (A)=  \sup\,\Big\{  H_q (\mu) + \int  A\,d \mu  \,:\, \mu \in \mathcal{G}\Big\},
\end{equation}
and we will say that  
$\mu_q$ is a \emph{$q$-equilibrium state} with respect to the potential $A$ is an invariant Gibbs measure attaining the supremum in ~\eqref{PP234}.
By definition,  the previous supremum is taken over the space of extensive Gibbs equilibrium states $\cG$, hence if non-extensive equilibrium states exist then these are equilibrium states for the classical thermodynamic formalism. 

\smallskip
In order to develop a spectral approach for the non-extensive thermodynamic formalism one needs to consider suitable transfer operators.  Given $q>0$, with $q\neq 1$, the inverse of $\log_q$ is the \emph{$q$-exp} function
defined by
\begin{equation} \label{T1e} u \mapsto  e_q^u=\exp_q (u) =  (1  + (1-q) u) ^{\frac{1}{ 1-q}}.
\end{equation}
This suggests to consider the family of transfer operators $\mathcal{L}_{A,q}: C^0(\Omega,\mathbb R) \to C^0(\Omega,\mathbb R)$
as
\begin{equation}\label{T2e}
\mathcal{L}_{A,q}(f)
= \sum_{a=1}^d e_{q}^{A(ax)}\, f(ax)
= \sum_{a=1}^d [1  + (1-q) A(ax)] ^{\frac{1}{ 1-q}}\, f(ax), \quad f\in C^0(\Omega,\mathbb R), 
\end{equation}
whenever the latter is well defined.

\smallskip
At this point there are major technical and conceptual differences in respect to the extensive transfer operators, at both conceptual and technical viewpoints. 
From the technical viewpoint,
the $q$-exp function behaves in a quite intricate way: (i) if $q>1$ then $\exp_q (u)$ is positive  if $u<\frac{1}{q-1}$; (ii)
if $0<q<1$, the value  $\exp_q (u)$ is positive  whenever $u>\frac{1}{q-1}$ and is complex otherwise (up to $q=1/2$, where is always a non-negative real number). Moreover, for $q>0$, the $q$-exp function is convex. The main conceptual differences
in respect to the extensive thermodynamic formalism are described in the next subsection.

\subsection{Non-extensive operators: duality and non-additive thermodynamic formalism}

Let us now describe the dynamical framework for non-extensive thermodynamic formalism, developed in the paper. The first key observation is that there exists
a 
non-standard relation on the $q$-parameter interval $I=(0,2)$ at $q=1$: the $q$-equilibrium states and the $q$-pressure function 
relates to the $(2-q)$-Ruelle operator.
More precisely, such a relation
$$
q-\text{pressure function}\;
P_q(\cdot)
\qquad
\leftrightsquigarrow
\qquad
 \; \text{transfer operator}\; \mathcal{L}_{A,2-q}(\cdot)
$$
is formalized in Theorem~\ref{mthm1}, which offers a Bowen-type formula, where we prove that solutions of a log-functional equation involving the $(2-q)$-transfer operators
are related to a zero of a Bowen-type equation involving the $q$-pressure function.
In this way, among the non-extensive $q$-equilibrium states associated to a certain Lipschitz continuous potential $A$ there exist extensive (classical) equilibrium states for a related Lipschitz continuous
(cf. Theorem~\ref{mthm1} for the precise formulation).

\smallskip
 The statement of Theorem~\ref{mthm1} is far from establishing a dictionary between extensive and non-extensive thermodynamic formalism.
 A fact that reinforces the latter is that
 the $q$-pressure function  $\beta \to P_q (\beta A)$ is neither convex nor concave on $\beta$ for large ranges of $\beta$ (see Remark \ref{rtu}), which creates technical problems for the use of the classical Legendre transform.
In this way, setting the duality of MaxEnt method and pressure 
in the non-extensive dynamical 
via the classical Legendre transform formalism seems not to  find parallel in the non-extensive framework (compare the
derivatives of the pressure functions in \cite[Proposition 4.10]{PP},  Example \ref{supex} and  \eqref{loo7} in case  $q=1/2$).
Furthermore, the lack of convexity of the $q$-pressure function seems to contribute to
a much richer structure on the space of $q$-equilibrium states and there are examples where
the log-functional equation has non-unique solutions (see Section~\ref{exe2}).

\smallskip
Another striking difference to the extensive thermodynamic formalism can be observed at the level of $q$-transfer operators, for each non-negative $q\ne 1$.
Indeed, while the usual exponential function
$
\exp: (\mathbb R,+) \to (\mathbb R_+,\times)
$
is a group homomorphism
and the leading eigenfunction of the classical transfer operators
can be obtained by normalized iterates $\cL_{A,1}^n(1)$, one has that $e_q^{a +b}\neq e_q^a \, e_q^b$ for every $q>0$, $q\neq 1$, and every
non-zero $a,b \in \mathbb{R}$. In this way, the weights appearing in the iterates $\cL_{A,q}^n(1)$ are much more intricate than Birkhoff sums, which are classical to Boltzman-equilibrium statistics.
In order to overcome this fact 
we introduce a sequence $(\mathfrak{L}_n)_{n\geqslant 1}$ of transfer operators \eqref{AAqinterat} adapted to the shift and a family $\Phi=(\varphi_n)_{n\geqslant 1}$ of Lipschitz continuous potentials $\varphi_n: \Omega \to \mathbb R$ given by ~\eqref{vai1}, which falls in the realm of non-additive thermodynamic formalism.
The sequence $\Phi=(\varphi_n)_{n\geqslant 1}$ has extremely mild additivity, i.e. it is just asymptotically sub-additive. Nevertheless one can show that
the $q$-asymptotic pressure $\mathfrak{P}^q(A)$, defined to represent the exponential growth rate of the norms of transfer operators $\mathfrak{L}_n$,
satisfies a variational principle involving the Kolmogorov-Shannon metric entropy and that equilibrium states always exist
(see Theorem~\ref{Mpressure} and Lemma~\ref{le:phing} for the precise statements).

\smallskip
In Theorem~\ref{Mddeste} we prove that the solutions of the functional equation concerning the
$(2-q)$-Ruelle transfer operators vary differentiably with the potential on a neighborhood of the original potential $A$, provided that this is normalizable.
This result can also be used to provide  an alternative argument for the existence of eigenfunctions for transfer operators, using the implicit function theorem.

\smallskip
Finally, there are several natural notions of equilibrium states appearing motivated by the non-extensive thermodynamic formalism, and
it is of huge interest to understand their interplay. A general relation between two of these notions, namely $q$-equilibrium states and $q$-asymptotic
equilibrium states, still seems out of reach due to the much different nature of the non-extensive objects (cf. Remark~\ref{rmkcompeq}). Nevertheless,
as Theorem~\ref{mthm1} establishes a bridge between non-extensive $q$-equilibrium states of a Lipschitz continuous potential $A$ and extensive equilibrium states for a modified Lipschitz continuous potential, it is natural to look for the dependence of these objects on the potential $A$, and to
understand its possible applications.

\subsection{Organization of the paper} \label{int321}

For the readers' convenience let us give a brief des\-cription on the organization of this paper.

\smallskip
The main results of this paper
are stated in Section~\ref{sec:statements}:
Theorem~\ref{mthm1} offers a 
duality between the $q$-pressure function
and the $(2-q)$-Ruelle operator,   
Theorem~\ref{Mpressure} establishes a variational principle for the non-extensive pressure function, while Theorem~\ref{Mddeste} describes solutions of a certain cohomological equation.

\smallskip
Our main focus in the present paper is the  study of the concept of $q$-entropy for probability measures $\mu \in \mathcal{G}$.
In Section \ref{d-q-ent} we describe properties of $q$-entropy functions $H_q(\cdot)$ for Gibbs and Bernoulli measures.
Later, in Section \ref{Dyn} we consider a dynamical point of view for  $q$-pressure under the non-extensive framework.
First we analyze the  case where the probability $p$  is of the form $p=(p_1,p_2,...,p_n)$ and we exhibit the maximal $q$-pressure probability for a given potential (cf. Section \ref{pres1}).  This corresponds to the case in which the dynamical system does not intervene and it is in consonance with most results in non-extensive Statistical Mechanics  
(see e.g. \cite[p. L71]{Curado}), as our definitions are associated to the so called  first choice of MaxEnt method as described in  \cite{Curado, Tsa1, Tsa3}.   Our main  objective in  Section \ref{pres1} 
is to allow the reader to understand the nature of the results that will be extended in later sections, and that will contemplate the dynamical viewpoint.
In Subsection \ref{FIR} considers the non-extensive point of view for the  $q$-pressure in  a dynamical setting.

\smallskip
In Section~\ref{prelim} we prove Theorem~\ref{mthm1}.
In order to do so, in Section \ref{qR} we  consider methods related to the problem of solving the $(2-q)$-Ruelle operator equation.   In Subsections \ref{1/2} and \ref{qq}, for $0<q<1$, we exhibit potentials $A$ for which we can find solutions for the  $(2-q)$-Ruelle operator equation \eqref{atchew1}, as solutions for the classical Ruelle operator equation with respect to some modified potentials.

\smallskip
In Section \ref{SubP} we relate the $q$-asymptotic pressure function with results from classical non-additive thermodynamic formalism. First, in
Subsection~\ref{SubP1} we describe the asymptotic sub-additive property of the family of potentials $(\varphi_n)_{n\geqslant 1}$, in the case $0<q<1$.
This will be crucial in the proof of Theorem~\ref{Mpressure}, which appears in Subsection~\ref{SubP2}.

\smallskip
On Section \ref{apen1} we present a proof of the existence of solutions for $q$-Ruelle Theorem equation using a  version of the implicit function theorem (see \cite{Lang,An}), and prove Theorem~\ref{Mddeste}.

\smallskip
In Section~\ref{exe2} is devoted to specific classes of examples on which the non-extensive objects can be computed.
In the case of locally constant potentials  $A:\{1,2\}^\mathbb{N} \to \mathbb{R}$ which depend only on the first two coordinates, that is $A(x)=A(x_1,x_2)$, we will show that the dynamical $q$-equilibrium state is a Markov probability.
Using a different technique,  Example~\ref{explimeq}
we present explicit solutions for the pressure problem in a non-extensive setting by exploring a relation between these and the $q$-Ruelle operator acting on Lipschitz continuous potentials, hence obtaining  a non-extensive version of the Ruelle's Theorem.

The remainder of the paper consists of four appendices.
Appendix A in Section \ref{dpreq} describes an explicit expression for the derivative of the pressure in the case $q=1/2$. Appendix B in   Section \ref{sec:shannon} describes the point of view of dynamical partitions for the non-extensive case (somehow related to \cite{MeVe}). In   Appendix C Section \ref{renyi} we briefly relate  the results of the non-extensive entropy described in our text with Renyi entropy.
Finally, the Appendix D (Section~\ref{apen}) contains a self-contained description of several general and useful properties for the $q$-log and $q$-exp functions; some of them  are used throughout the paper.

\color{black}

\section{Main results} \label{sec:statements}

\subsection{Setting} \label{ddyn}

\smallskip
Let $\sigma: \Omega \to \Omega$ denote the one-sided shift on $\Omega=\{1,2,...,d\}^\mathbb{N}$, endowed with the usual distance (which we denote by $\mbox{dist}$), which makes it diameter one.
Let $C^0(\Omega,\mathbb R)$ be the Banach space of all continuous functions on $\Omega$ endowed with the $C^0$-topology
and let $\mbox{Lip}(\Omega) \subset C^0(\Omega,\mathbb R)$ be the subspace of Lipschitz continuous functions endowed with the
norm 
\begin{equation}
    \label{eq:Lipnorm}
    \|f\|_{\text{Lip}} = \|f\|_{C_0} + |f|_{\text{Lip}}
    \quad\text{where}\quad
    |f|_{\text{Lip}} =\sup_{x\neq y} \frac{|f(x)-f(y)|}{\mbox{dist}(x,y)}
\end{equation}
Given a Lipschitz continuous function  $A:\Omega \to \mathbb{R}$
the \emph{Ruelle transfer operator} associated to the potential $A$ is the linear operator $\mathcal{L}_{A}$ given by \eqref{AA}.
Its dual, denoted by $\mathcal{L}_A^* $, acts on the space  $\mathcal M_1(\Omega)$ of probability measures on $\Omega$
by the duality relation
$$
\int f \,d (\mathcal{L}_A^* (\mu_1)) =\int \mathcal{L} (f) \, d \mu_1 \qquad \forall f\in C^0(\Omega,\mathbb R).
$$
It is clear from the definition that $\mathcal{L}_{A}$ preserves the space $\mbox{Lip}(\Omega)$ of Lipschitz continuous pontentials on $\Omega$.
In case there exists
 $c\in \mathbb{R}$ and a continuous function $\varphi:\Omega \to \mathbb{R}_+$ so that the \emph{Ruelle operator equation}
 \begin{equation} \label{atchew111}
 \mathcal{L}_{A} \Big( \lambda^{-1} \frac{\varphi}{(\varphi \circ \sigma)\,}\Big) (x)= \sum_{a=1}^d e^{A(a x) + \log(\varphi (a x)) -  \log(\varphi)(x) - c}=1\quad
 \text{for every}\; x \in \Omega
\end{equation}
holds then we say that $A$ is a \emph{normalized potential}, that $\varphi$ is the \emph{eigenfunction} and that the constant $\lambda=e^c>0$ is the
the \emph{eigenvalue} associated to $\mathcal{L}_{A}.$
For short we will say that the pair $(\varphi,c)$ solves the (extensive) Ruelle operator equation.

Ruelle's theorem ensures that
for any Lipschitz continuous potential  $A:\Omega \to \mathbb{R}$ there exist
 $c\in \mathbb{R} $ and continuous positive  Lipschitz continuous function $\varphi:\Omega \to \mathbb{R}$ satisfying
 \eqref{atchew111}, and that the topological pressure $P(A)$
 of the shift $\sigma$ with respect to the potential $A$ coincides with $\log \lambda$ (see e.g. \cite{Bala,PP}).

\smallskip
A \emph{Jacobian} $J:\Omega 
\to \mathbb{R}$ is a positive Lipschitz continuous function such that
\begin{equation}\label{JJ}
\mathcal{L}_{\log J} (1)(x)=\sum_{a=1}^d J(ax)=1 \quad \text{for every $x\in \Omega$.}
\end{equation}
For each Jacobian $J$ there exists a unique probability measure $\mu=\mu_{\log J}$, fully supported, such that
$\mathcal{L}_{\log J}^*(\mu)=\mu$, to which we will refer as the (extensive) \emph{equilibrium state} of the potential $\log J$
(these are often called \emph{Gibbs measures} cf. \cite{Bow}).  The space
\begin{equation}\label{eqG}
\mathcal G = \Big\{  \mu_{\log J} \colon J \; \text{is a Jacobian} \Big\}
\end{equation}
of all Gibbs measures associated to the extensive thermodynamic formalism is an infinite dimensional manifold (cf. \cite{LR1}).
As there exists a bijective relation between the space of Jacobians and the associated equilibrium states,
the elements in $\mathcal{G}$ can be either parameterized by their elements $\mu$ or by their Jacobians $J$   (we refer the reader to \cite{GKLM,LR1}
for more details).
Our main goal is to determine whether there exist $q$-equilibrium states - recall these are probability measures on $\mathcal{G}$ attaining the supremum in ~\eqref{PP234} -  and to build possible bridges between the classical (extensive) equilibrium states and $q$-equilibrium states.

\subsection{Statements}
The starting point for our study of the non-extensive thermodynamic formalism is the following duality between the $q$-pressure function
and the $(2-q)$-Ruelle operator equation. In this way, the next theorem provides not only sufficient conditions for the solution of a non-extensive Bowen type equation, as it ensures that the $q$-equilibrium state for a given Lispchitz potential $A$ is a  classical equilibrium for a pressure problem for another potential  within the extensive thermodynamic formalism.

\begin{maintheorem}\label{mthm1}
Let $q>0$ and $A:\Omega \to \mathbb{R}$ be  a Lipschitz continuous potential. If there exists a constant
 $c\in \mathbb{R} $ and a continuous function $\varphi:\Omega \to \mathbb{R}$ so that
 \begin{equation} \label{atchew1}
 \sum_{a=1}^d e_{2-q}^{A(a x) + \varphi (a x) - \varphi(x) - c}=1 \quad \text{for every $x \in \Omega$}
\end{equation}
and that all summands above are strictly positive then
$P_q(A) =c$.
 Moreover, the following properties hold:
 \begin{enumerate}
     \item
     there exists a $q$-equilibrium state associated to $A$  that coincides with  the  equilibrium state
     for the 
     potential  $\log J$, where
$J(x)= e_{2-q}^{A( x) + \varphi ( x) - \varphi(\sigma(x)) - c};$
     \item
\begin{equation}  \label{eeste3}
P_q \Big( -\log_q \Big(\frac{1}{J}\Big) \Big) =0,
\end{equation}
where $J$ is the classical Jacobian described by 
\begin{equation}  \label{eeste1}-\log_q \Big(\frac{1}{J}\Big)=  \frac{-1 + (1 + (A+ \varphi  - (\varphi\circ \sigma) - c  )\, (q-1))^{-1} }{q-1 }.
\end{equation}
 \end{enumerate}
\end{maintheorem}

\begin{remark}
Theorem~\ref{mthm1} makes explicit a symmetry between the parameter $q$ of the
$q$-pressure and the parameter $\tilde q=2-q$ appearing in the $\tilde q$-Ruelle transfer operator equation ~\eqref{atchew1}. In the extensive framework
(corresponding to $q=1$) we recover that the pressure function can be obtained through the leading eigenvalue of the classical transfer operator.
We will refer to a solution $\varphi$ of \eqref{atchew1} as the \emph{$(2-q)$-non-extensive log-eigenfunction} (eigenfunction for short) and to the
constant $c$ as the \emph{$(2-q)$-non-extensive log-eigenvalue} (eigenvalue for short).
The existence of a continuous function $\varphi$ and $c\in \mathbb{R}$ solving \eqref{atchew1} is a non-extensive version of the Ruelle operator theorem  \cite[Theorem 2.2]{PP}.
    \end{remark}

\begin{remark}\label{rmk:nonuniquesol}
The pair of solutions $\varphi,c$  in Theorem~\ref{mthm1} \emph{need not to be unique}
(see  Example \ref{jana}).
 Moreover, in Example~\ref{explimeq} we exhibit a Lipchitz continuous potential $A$ where there exist a pair of solution $\varphi,c$ for \eqref{atchew1}, but one of the summands    can take the value zero for some $x\in \Omega$.
In this way, in the non-extensive setting one should not expect a full extension of the classical Ruelle Theorem.
\end{remark}

\smallskip
It is worth mentioning that whenever one writes an expression like \eqref{atchew1} we are implicitly assuming that all values are well defined and, in most cases,  we will assume that all summands are strictly positive for every $x$. Let us discuss a bit
further the class of potentials that one considers. Given $1< q \leqslant 2$ consider the open subset (in the Lipschitz topology)
$$
\cF_{q} =\Big\{ A\in \mbox{Lip}(\Omega)\colon A(x) >\frac{1}{1-q}, \; \forall x\in \Omega\Big\}.
$$
This open set varies continuously with $q$, in the Hausdorff distance.
At a first sight Theorem~\ref{mthm1} could suggest to
consider, for each $A\in \cF_{q}$, the $q$-transfer operator
\begin{equation*}\label{AAq}
\mathcal{L}_{A,q}(f) = \sum_{a=1}^d e_{q}^{A(ax)}\, f(ax), \qquad f\in C^0(\Omega,\mathbb R) 
\end{equation*}
However, in virtue of the properties of $q$-exponentials (cf.  Section~\ref{apen}) their iterates 
$$
\mathcal L_{A,q}^n (f)(x) = \sum_{\sigma^n(y)=x} \, \Big[ \prod_{j=0}^{n-1} e_q^{ A(\sigma^j(y))} \Big] \, f(y)
$$
seem not suitable to study the non-extensive thermodynamic formalism.
For that reason we  will consider the family of transfer operators
$(\mathfrak{L}_n)_{n \in \mathbb{N}}$, defined by
\begin{equation}\label{AAqn}
\mathfrak{L}_n (f) (x)= \sum_{\sigma^n(y)=x}  e_q^{ A(y) + A (\sigma(y)) +... + A(\sigma^{n-1}(y) ) } \, f(y)
\end{equation}
for every $f\in C^0(\Omega,\mathbb R) $ 
and $x\in \Omega$, which can be written as 
\begin{equation}\label{AAqinterat}
\mathfrak{L}_n (f) (x) = \sum_{\sigma^n(y)=x}  e^{\varphi_n(y)} f(y)
\end{equation}
where
\begin{equation} \label{vai1}
\varphi_{n}(y)= \frac{1}{1-q} \, \log  \left(\, 1 + (1-q) \,
\sum_{j=0}^{n-1} A(\sigma^j(y))\,\right), \; n \in \mathbb{N}.
\end{equation}
The family of potentials $(\varphi_n)_{n\ge 1}$ is asymptotically sub-additive (cf. Lemma~\ref{le:phing}). However, this family of potentials is not almost additive and does not seem to satisfy any of the sufficient conditions  that allow to study the non-additive thermodynamic formalism developed in previous works
(see \cite{Barr0,CFH,Falconer,VZ} and references therein), which does not allow us to use of spectral theory
to study non-extensive thermodynamic formalism.
Nevertheless we prove the following variational principle.

\begin{maintheorem} \label{Mpressure}  Given $0<q<1$, a Lipchitz continuous potential $A: \Omega \to \mathbb{R}$ the limit
\begin{equation} \label{ratim}
\mathfrak{P}^q(A) = \lim_{n \to \infty} \frac{1}{n} \log \mathfrak{L}_n (1) (x_0)
\end{equation}
exists and is independent ot the point $x_0$.
Moreover, 
 \begin{equation} \label{eefe1}
\mathfrak{P}^q (A)\,=\,\max\,\Big\{ h(\nu)+ \lim_{n \to \infty} \frac{1}{n} \int \varphi_n\, d \nu \colon {\nu \in \mathcal{M}_{\text{inv}}(\sigma)}\Big\},
 \end{equation}
 where $ h(\nu)$ is the extensive Kolmogorov-Shannon entropy of $\nu$.
\end{maintheorem}

The proof of Theorem~\ref{Mpressure} will be given in Section \ref{SubP}.
The limit $\mathfrak{P}^q(A)$ in ~\eqref{ratim} will be called the {\it $q$-asymptotic pressure} of the potential
$A: \Omega \to \mathbb{R}$
and the probability measures attaining the maximum will be referred to as
\emph{$q$-asymptotic equilibrium states}.

\begin{remark}\label{rmkcompeq}
The variational principle in ~\eqref{eefe1} should be compared with the variational definition of the $q$-pressure $P_q(A)$ in ~\eqref{PP234}. The exponential growth rate of the norm of the operators $(\mathfrak{L}_n)_{n\in \mathbb N}$ is related to sum of extensive entropies with the average of
a sub-additive family of potentials, whereas the $q$-pressure considers non-extensive entropies but considers the usual integral over the potential.
\end{remark}

\begin{remark}
A priori there is no relation between $q$-asymptotic equilibrium states and the (non-extensive) $q$-equilibrium states,
as these involve different measure theoretical entropies.
The above expression connects the non-extensive setting with the sub-additive setting (recall ~\eqref{T11}).
\end{remark}

\medskip

\begin{maintheorem} \label{Mddeste}  Let $0<q\leqslant 1$, let $\tilde{A}:\Omega \to \mathbb{R}$ be a normalized
Lipschitz continuous potential and let $\nu_{\tilde{A}}$ be the probability measure such that $\mathcal{L}_{\tilde A}^* \nu_{\tilde A}=\nu_{\tilde A}$.
There exists an
open neighborhood $\mathcal U\subset \mbox{Lip}(\Omega)$ of $\tilde A$ and a differentiable map
$$
\mathcal U \ni A \mapsto (\varphi_A,c_A) \in \mbox{Lip}(\Omega) \times \mathbb R
$$
such that
\begin{equation} \label{atc21}
\sum_{a=1}^d e_q^{A(a x) + \varphi_A (a x) - \varphi_A(x) - c_ A}=1\qquad \text{for every $x\in \Omega$}
\end{equation}
 and $ \int \varphi_A d \nu_{\tilde{A}}=0$.
 
 Moreover, $\lim_{A \to \tilde A} (\varphi_A,c_A) =(0,0)\in \mbox{Lip}(\Omega)\times \mathbb R$.
 
 \end{maintheorem}

The previous result will be proved in Section~\ref{apen1}, stated as Theorem~\ref{Mddeste0}.
One can ask about natural condition under which for a Lipschitz continuous potential
there exist $\varphi$ and $c$ such that \eqref{atc21} holds.
 We conjecture that given $q>0$, in general there exists a positive answer to the question for potentials $A$ on some open set of potentials.  In the special case that $\Omega=\{1,2\}^\mathbb{N}$, and $A$ depends on two coordinates,  we obtain explicit solutions for a nontrivial class of  examples (see Example \ref{explimeq} in Section \ref{exe2} and Subsection \ref{1/2}).

\begin{remark}
We point out that when considering potentials which are merely continuous 
in the extensive thermodynamic formalism
then phase transitions,  slow decay of correlations and existence of non-ergodic equilibrium states may appear
(see e.g. \cite{CL1,CL3,FL,Lep1,Lo1,LopR,RW}  and references therein). In particular there exist continuous potentials $A:\Omega \to \mathbb{R}$ for which there is no eigenfunction solution for the Ruelle operator equation \eqref{atchew111} (see \cite{GKLM}).
In this paper we will always consider Lipschitz continuous potentials and postpone
the study of the non-extensive thermodynamic formalism for less regular potentials to a future work.
\end{remark}

\begin{remark}
The extensive equilibrium  state $\hat{\mu}$ for the two-sided shift and a  Lipschitz continuous potential   $\hat{A}:\{1,2,...,d\}^\mathbb{Z} \to \mathbb{R}$ 
can be analyzed by showing that it is cohomologous to a Lipschitz continuous potential $A:\{1,2,...,d\}^\mathbb{N} \to \mathbb{R}$ which depends only of positive coordinates (see Proposition 1.2 in \cite{PP} or Appendix 7.1 in \cite{Lop}).
\end{remark}

\section{$q$-entropies} \label{d-q-ent}

In this subsection we will study some basic properties of $q$-entropy functions, namely that this is a concave and upper-semicontinuous function on the probability measures. Throughout the section let $\mathcal{C}^+$ denote the space of continuous positive functions on $\{1,2,..,d\}^{\mathbb N}$. 

\subsection{Gibbs measures}
Given $0<q<1$, recall the {$q$-entropy of a probability measure $\mu \in \mathcal{G}$}   is
\begin{equation*} \label{H1q}
H_q(\mu) = \int \frac{1}{1-q} ( J(x)^{q-1}-1)      d \mu (x)=  \int \log_q (\frac{1}{J(x)}) d \mu(x) 
\end{equation*}
where $J=J_\mu$ is the Jacobian of $\mu \in \mathcal{G}$, hence it is always non-negative (see Section~\ref{apen} for the properties of the
$q$-exponential function).
Moreover, as $- \log x > \frac{1}{1- q}(x^{q-1} -1)$ for every $0<x<1$ and $0<q<1$, one has that 
\begin{equation} \label{hH} h(\mu) \geqslant H_q(\mu) \quad \text{for every $\mu\in \mathcal G$ and $0<q<1$.}
\end{equation}

\begin{lemma}
\label{rew1}
The function $\cG\ni \mu \to H_q(\mu)$ is differentiable.
\end{lemma}

\begin{proof}
  By \cite{LR1}, the space $\mathcal{G}$ is an analytic Banach manifold. Therefore, for a fixed $q>0$,  using that $\log_q$ is differentiable on its domain
and that the function $\mu \to J_\mu$ is differentiable
(see e.g. \cite{GKLM}). Hence  we conclude that   $\cG\ni \mu \to H_q(\mu)$ is differentiable as well.
\end{proof}

We proceed to prove a variational characterization for the $q$-entropy of Gibbs measures.

\begin{lemma}\label{lemavp}
Fix an H\"older continuous normalized  potential $B=\log J:\{1,2,..,d\}^\mathbb{N}\to \mathbb{R}$ and let $\mu_{\log J}$ be the  equilibrium state with Jacobian $J$.
Then, for $0<q<1,$
\begin{equation} \label{oiu34} H_q(\mu)=\inf_{u\in \mathcal{ C}^+}\bigg\{\int  \log_q \bigg(\frac{ \sum_{a=1}^d  u(a\, x)}{u(x)}\bigg) d\mu(x) \bigg\}.
\end{equation}

\end{lemma}

\noindent
\begin{proof}
By definition in ~\eqref{HQ} one knows that  $H_q(\mu) =  \int \log_q( J^{-1}) d \mu.$
Taking $\tilde u(x)=e^{\log J (x)} \in \mathcal C^+$, one can write the right-hand side above as 
$$ 
\log_q \bigg(\frac{ \sum_{a=1}^d \tilde u(a\, x)}{\tilde u(x)}\bigg)=
\log_q \bigg(\frac{\sum_{a=1}^d J (a\,x) }{e^{\log J (x)}}\bigg)=\log_q (J^{-1}), 
$$
and so, by integration,
\begin{eqnarray*}
\int  \log_q \bigg(\frac{\sum_{a=1}^d \tilde{u} (a\,x)}{\tilde u(x)}\bigg)d\mu(x) = H_q(\mu).
\end{eqnarray*}
Now, given a general $\overline{u}\in \mathcal{ C}^+$ it can always be written as $\overline{u}=u J,$ where $u$ is positive. We claim that 
$$
\int  \log_q \bigg(\frac{\sum_{a=1}^d \overline{u} (a\,x)}{\overline{u}(x)}\bigg)d\mu(x) \geq H_q(\mu).$$
Using that $\log(y) \leq \log_q(y)$ for each $0<q<1$, the concavity of  $x \to \log (x)$ and Jensen's inequality we deduce that
\begin{align*}
 \int  \log_q \bigg(\frac{\sum_{a=1}^d \overline{u} (a\,x)}{\overline{u}(x)}\bigg)d\mu(x) 
& =\int  \log_q \bigg(\frac{\sum_{a=1}^d u (a\,x) J(a x)}{u(x) \,J(x)}\bigg)d\mu(x) \\
& \ge
\int  \log \bigg(\sum_{a=1}^d u (a\,x)\, J(a x)\, u(x)^{-1} \,J(x)^{-1}\bigg)d\mu(x)  \\
& = \int  \log \bigg(\,\sum_{a=1}^d [\,u (a\,x)\, \, u(x)^{-1} \,J(x)^{-1}\,]\, J(a x)\bigg)d\mu(x) \\
& \ge \int   \sum_{a=1}^d     \log(\,u (a\,x)\, \, u(x)^{-1} \,J(x)^{-1}\,)\, J(a x)\,d\mu(x).
\end{align*}
This, together with the $\sigma$-invariance of $\mu$ and  \eqref{hH} yields that the right-hand side term above can be written as 
\begin{align*}
   \int        \log(\,u (x)\, \, u(\sigma(x))^{-1} \,J(\sigma( \,x))^{-1}\,)\, \,d\mu(x) 
   & = \int        \log (J(\sigma( \,x))^{-1}\,)\, \,d\mu(x) \\
   & = \int        \log (J(x)^{-1}\,)\, \,d\mu(x) 
    = h(\mu) \geqslant H_q(\mu).
\end{align*}
This proves the lemma.
 \end{proof}
 \medskip

\subsection{Invariant measures}
 In this subsection we extend the concept of $q$-entropy for arbitrary probability measures in $ \mathcal{M}_{\text{inv}}(\sigma).$
 In fact, Lemma~\ref{lemavp} suggests the following definition.

 \begin{definition} \label{uod} Given $0<q \le 1$ and a $\sigma$-invariant probability measure $\mu \in \mathcal{M}_{\text{inv}}(\sigma)$, the \emph{$q-$entropy} of $\mu$ is defined as
\begin{equation} \label{oiu54} H_q(\mu)=\inf_{u\in \mathcal{ C}^+}\bigg\{\int  \log_q \bigg(\frac{ \sum_{a=1}^d  u(a\, x)}{u(x)}\bigg) d\mu(x) \bigg\}.
\end{equation}
\end{definition}

We note that, in opposition to the case of Gibbs measures, for an arbitrary invariant measure the infimum  in  \eqref{oiu54} may not be necessarily attained by some function in $\mathcal{C}^+$.

\smallskip 
We proceed to study the concavity of the $q$-entropy map $\mu  \mapsto H_q(\mu)$. Concavity will not follows from a naive approach using Gibbs measures as, while for each
$\lambda \in (0,1)$ and $\mu_1,\mu_2\in \mathcal{G}$, the invariant probability measure
$\lambda \mu_1 + (1-\lambda)  \mu_2$ does not belong to $\mathcal{G},$ even though it can be weak$^*$  accumulated by Gibbs measures (see \cite{L3} or \cite[Corollary 7.14]{GKLM}). 
In this way we will use the variational formulation for $q$-entropy to prove that it is actually a concave function.

\begin{lemma}\label{le:Hconcave} Fix $0<q<1$. The $q$-entropy map $\mathcal{M}_{\text{inv}}(\sigma) \ni \mu \mapsto
H_q(\mu) $ is concave.
\end{lemma}

\begin{proof} Fix $0<q<1$. 
Given $\lambda \in [0,1]$ and $\mu_1,\mu_2\in  \mathcal{M}_{\text{inv}}(\sigma)$, consider the probability $\mu_\lambda = \lambda \mu_1 + (1-\lambda )\mu_2$. 
By definition, given $\vep>0$ there exists  $u\in \mathcal{ C}^+$ such that
\begin{align*}
  H_q(\mu_\lambda)
  & \ge \int  \log_q \bigg(\frac{ \sum_{a=1}^d  u(a\, x)}{u(x)}\bigg) d \mu_\lambda \, -\vep  \\
  & = \int  \log_q \bigg(\frac{ \sum_{a=1}^d  u(a\, x)}{u(x)}\bigg) d  [\,\lambda \mu_1 + (1-\lambda)  \mu_2\,]\,- \vep
  \\
  & =
  \lambda \,\int  \log_q \bigg(\frac{ \sum_{a=1}^d  u(a\, x)}{u(x)}\bigg) d  \, \mu_1 \\
  & +  (1-\lambda) \int  \log_q \bigg(\frac{ \sum_{a=1}^d  u(a\, x)}{u(x)}\bigg) d  \mu_2\, (x)\,
  - \vep.  
\end{align*}
Now, as $q$-entropy of each $\mu_i$ is defined by an infimum over $\mathcal C^+$ of the expression in ~\eqref{oiu54} one concludes that 
$H_q(\mu_\lambda) \ge \lambda H_q(\mu_1) + (1-\lambda)  H_q(\mu_2) -\vep.$
As $\vep>0$ is arbitrary, this proves the concavity of the $q$-entropy map, as desired.
\end{proof}

\begin{remark}
 The Kolmogorov-Shannon entropy map $\mathcal{M}_{\text{inv}}(\sigma)\ni \mu\mapsto h(\mu)$ is well known to be affine.   
Simulations suggest that the $q-$entropy is not affine when $0<q<1.$
\end{remark}

\smallskip

\begin{lemma}\label{uup} Fix $0<q<1$. The $q$-entropy $H_q (\mu)$  is an upper semi-continuous  function of $\mu$ on $\mathcal{M}_{\text{inv}}(\sigma)$.
\end{lemma}

\begin{proof} Fix $\mu \in \mathcal{M}_{\text{inv}}(\sigma)$ and take a sequence $(\mu_n)_{n\ge 1}$ in $\mathcal{M}_{\text{inv}}(\sigma)$ converging to $\mu$ in the weak$^*$ topology.  Assume, by contradiction, there exists $\vep>0$, such that
$$ \limsup_{n \to \infty} H_q(\mu_n) > H_q(\mu) + \vep.$$
Up to extracting a subsequence of the original sequence we will assume that the limit in the left-hand side does exist. 
Then there exists $u \in \mathcal{C}^+$   and $N \ge 1 $, such that
\begin{equation} \label{tyty} H_q(\mu_n) > H_q(\mu) + \vep\geq \int  \log_q \bigg(\frac{ \sum_{a=1}^d  u(a\, x)}{u(x)}\bigg) d  \mu\, (x)
\end{equation}
for every $n \ge N$.
By weak$^*$ convergence, as $u$ is continuous and strictly positive one gets
$$ \lim_{n \to \infty} \int  \log_q \bigg(\frac{ \sum_{a=1}^d  u(a\, x)}{u(x)}\bigg) d  \mu_n \, (x)=\int  \log_q \bigg(\frac{ \sum_{a=1}^d  u(a\, x)}{u(x)}\bigg) d  \mu \, (x) $$
which, together with \eqref{tyty}, shows that
$$ H_q(\mu_n)>\int  \log_q \bigg(\frac{ \sum_{a=1}^d  u(a\, x)}{u(x)}\bigg) d  \mu_n\, (x) ,$$
which is a contradiction with \eqref{oiu54}.
\end{proof}

The next result, inspired by \cite[Theorem 9.12]{Walters} offers a dual variational principle for the $q$-entropy, showing that is coincides with the Legendre-Fenchel transform of the $q$-pressure function. More precisely, define
\begin{equation}
    \label{eq:defPq}
    \hat{P}_q(g)=  \sup\,\Big\{  H_q (\mu) + \int  g\,d \mu \,:\, \mu \in \mathcal{M}_{\text{inv}}(\sigma)\Big\}.
\end{equation}
The next result shows that the  topological pressure $\hat{P}_q$ determines the $q$-entropy $H_q$ map.

\begin{proposition} \label{kul}  Suppose  $\mu \in \mathcal{M}_{\text{inv}}(\sigma)$. Then,
$$H_q(\mu) =  \inf_{g \in C(\Omega,\mathbb{R})} \Big\{ \hat{P}_q(g) - \int g\, d\mu \Big\}.
$$
\end{proposition}

\begin{proof} 
Fix $\mu_0\in\mathcal{M}_{\text{inv}}(\sigma)$ and $g  \in C(\Omega, \mathbb{R})$. By ~\eqref{eq:defPq},
 $\hat{P}_q(g) - \int g \, d\mu_0\geq H_q(\mu_0) .$ 
By arbitrariness of $g$ this proves that
\begin{equation} \label{nnf1}
\inf_{g \in C(\Omega,\mathbb{R})} \Big\{\hat{P}_q(g) - \int g \, d\mu_0\Big\}\geq H_q(\mu_0).
\end{equation}

\smallskip
We proceed to prove the converse. 
Take $b>H_q(\mu_0).$
Take
$ C = \{(\mu,t)\in \mathcal{M}_{\text{inv}}(\sigma) \times \mathbb{R} \,:\,0 \leq  \, t \leq H_q(\mu) \}. $ 
Note that for any  $\mu\in\mathcal{M}_{\text{inv}}(\sigma) $ there exists $t\ge 0$ such that $(\mu,t)\in C.$
As the entropy map $H_q$ is concave (see Lemma \ref{le:Hconcave}), we get that $C$ is a compact convex set: if $(\mu,t), (\nu,s)\in C$ and $\lambda \in [0,1]$ then 
$ 
\lambda t + (1-\lambda) s
    \le \lambda H_q(\mu) + (1-\lambda) H_q(\nu) \le H_q(\lambda \mu + (1-\lambda) \nu),
$ 
which proves that $(\lambda \mu + (1-\lambda)\nu,\lambda t + (1-\lambda)s)\in C$.
From the upper semicontinuity  of $H_q$ (recall Lemma \ref{uup}), one concludes that $(\mu_0,b) \notin \overline{C}$.
As $\mathcal M_{inv}(\sigma)\subset \mathcal M(X)\simeq C(X)^*$, by the geometric Hahn-Banach separation theorem (cf. \cite[p. 417]{DF}) there exists a linear functional $L: C(X)^*\times \mathbb R \to \mathbb R$ and $\alpha\in \mathbb R$ so that 
$ 
L(\mu,t)  \le \alpha <L(\mu_0,b) 
$ 
for every $(\mu,t) \in \overline{C}$.
By Riesz's representation theorem, there exists a continuous function
$v:\Omega  \to \mathbb{R}$  such that
$$ \int v  \; d \mu  +\alpha t < \int v  \; d \mu_0  +\alpha b,$$
for all $(\mu,t) \in \overline{C}$.
Taking $(\mu_0,H_q(\mu_0))\in \overline{C}$ in the previous expression we get that $\alpha   H_q(\mu_0)< \alpha b$, which shows that $\alpha>0.$ Therefore, 
\begin{equation*} \label{jet}  H_q(\mu)+ \frac{1}{\alpha}\int v\,  d \mu < b + \frac{1}{\alpha}\int v\, d \mu_0
\quad\text{for any $\mu\in \mathcal{M}_{\text{inv}}(\sigma)$}
\end{equation*}
 and so, taking the supremum over all invariant measures and recalling ~\eqref{eq:defPq},
$$
\hat{P}_q\Big( \frac{v}{\alpha} \Big)
\leq b +  \frac{1}{\alpha}\int v\,  d \mu_0.
$$
Thus 
$$
b \geq \hat{P}_q \Big( \frac{v}{\alpha} 
\Big) -  \frac{1}{\alpha}\int v d \mu_0
\geq \inf \Big\{ \hat{P}_q (g) - \int g d \mu_0\, :\, g \in C(\Omega, \mathbb{R}) \Big\}.
$$
Since $b>H_q(\mu_0)$ is arbitrary we conclude that 
$$  H_q(\mu_0) \geq \inf \Big\{ \hat{P}_q (g) - \int g d \mu_0\, :\, g \in C(\Omega, \mathbb{R}) \Big\}.$$
This finishes the proof of the proposition.
\end{proof}

The next result shows that, even though for each $0<q< 1$ the $q$-entropy maps lacks the symmetry of the extensive entropy map,
all of these attain its maximal value at the Bernoulli measure with equal weights. More precisely:

\begin{lemma} \label{tos} Fix $0<q<1$. The function $\cG\ni\mu\mapsto H_q(\mu)$ has a unique maximum, attained at the Bernoullli measure $\mu_0$ with equal weights $1/d$, and
 $H_q(\mu_0)
 =\log_q (d)$.
\end{lemma}
\begin{proof}
 Note first that, by Lagrange multipliers, it is not hard to check that the supremum of the function
  $p=(p_1,p_2,...,p_d) \to \Phi(p)=\sum_{j=1}^d p_j^{q}-1$ 
 is the value $d^{1-q}-1,$
 and that it is only
attained at $p_0=(1/d,1/d,...,1/d)$.
Hence, if $J$ denotes the Jacobian of a probability measure $\mu\in \mathcal G$ then, by ~\eqref{HQ}, $\sigma$-invariance of $\mu$ and $\sum_{a=1}^d J(ax)=1$, 
\begin{align*}
    H_q(\mu) 
    & = \int \frac{1}{1-q} ( J(x)^{q-1}-1)      d \mu (x)
    =
 \frac{1}{1-q}  \int \sum_{a=1}^d  J(a x)( J(a x)^{q-1}-1)      d \mu (x) \\
 & =\frac{1}{1-q}  \int    (\sum_{a=1}^d J(a x)^{q}-1)    d \mu (x).
\end{align*}
Now, as $(J(ix))_{1\le i \le d}$ is a probability vector one has that 
$\Phi((J(ix))_{1\le i \le d}) \le d^{1-q}-1$ for every $x\in \Omega$.
In consequence, 
$$H_q(\mu) =\frac{1}{1-q}  \int    (\sum_{a=1}^d J(a x)^{q}-1)    d \mu (x) 
\leqslant   \frac{1}{1-q}     (d^{1-q}-1) = H_q(\mu_0) = \log_q(d),
$$
which proves the lemma.
\end{proof}

\smallskip

\subsection{Dynamical relative $q$-entropy} \label{sec:qrel}

Let us finish this section by introducing the concept of relative $q$-entropy for Gibbs measures.
Suppose $\mu_i\in \mathcal G$ has Lipschitz continuous Jacobian $J_i$, $i=1,2$. The \emph{relative $q$-entropy} (also called
\emph{$q$-KL divergence}) is the value

\begin{equation} \label{HH2}
H_q(\mu_1,\mu_2) = \int \log_q \Big(\frac{1}{ J_2}\Big)d \mu_1 -  \int \log_q \Big(\frac{1}{J_1}\Big)d \mu_1.
\end{equation}
\color{black}
This function $H_q(\mu_1,\mu_2) $ is analytic on the pair $(\mu_1,\mu_2)$ in the Banach manifold of  Lipschitz equilibrium states
$\mathcal{G}$ (cf. \cite{LR1}). Moreover, $H_q(\mu_1,\mu_2) $ is non-negative: using \eqref{rt15} we obtain
\begin{align}
H_q(\mu_1,\mu_2) & = \int \log_q (\frac{1}{ J_2(x)})d \mu_1 (x) -  \int \log_q (\frac{1}{\tilde J(x)})d \mu_1 (x) \nonumber \\
	& =  \int \sum_{a=1}^d J_1(a x)   \Big[\log_q (\frac{1}{J_2(a x)}) -  \log_q (\frac{1}{J_1(ax)}) \Big]\,d \mu_1 (x)  \geqslant 0.
    \label{oret}
\end{align}

\section{The $q$-pressure function} \label{Dyn}

\subsection{The non-dynamical $q$-pressure function} \label{pres1}
In this subsection we consider the $q$-pressure in a setting where there is no underlying dynamical system,
 the initial setting considered in \cite{Curado1,Tsa1}.
 Althouth the results in this subsection are not strictly necessary for reading the rest of the paper, these offer a motivation and insights for the theory to be developed in the next sections.

\smallskip
Given $0<q<1$ and a continuous potential $A:\{1,2,...,d\} \to \mathbb{R}$, the $q$-pressure of 
$\beta A$, defined in ~\eqref{fipre} and involving 
a non-extensive entropy and an extensive integral, is defined as 
\begin{equation}\label{outP} 
P_q (\beta \,A) = \sup_p  \Big\{ H_q(p)  +\beta\, \,\sum_{j=1}^d p_j a_j   \Big\}
 = \sup_{p}\, \Big\{\,\, \frac{1}{1-q}  \,(\sum_{i=1}^d p_i^{q}-1)
 	+ \beta  \sum_{j=1}^d a_j \, p_j  \Big\}
\end{equation}
where the supremum
is  taken over all probability vectors $p$ on $\{1,2, \dots, d\}$,
and a probability vector $p$ is a \emph{$q$-equilibrium for $\beta\, A$} if $p$ attains the supremum above. The classical pressure of $\beta A$ is similar to the previous expression with $H_q(p)$  replaced by  $\sum_{j=1}^d - p_i\log p_j$. 
It is easy to check (using \eqref{hH21}) that $P(A) \geqslant P_q(A)$ for every  for $0<q<1$ and that $P(A) \leqslant P_q(A)$
for every $q>1$.

\begin{lemma} \label{pr1}
Given $q>0$,  $\beta \in \mathbb{R}$ and  a potential $A=(a_1,a_2,...,a_d)$,  the $q$-equilibrium state $p=(p_1,p_2,...,p_d)$ for $\beta A$ is unique
and given by
\begin{equation}\label{outtt}
p_j= \frac{e_{2-q}^{\beta a_j}}{\sum_{i=1}^d\; e_{2-q}^{\beta a_i}},
\qquad \text{for every $1\le j \le d$.}
\end{equation}
\end{lemma}

{
\begin{proof} In order to determine the probability that attains the maximal value of \eqref{outP} subject to the constraint $\sum_{j=1}^dp_j=1$ we use Lagrange multipliers for the function
$$((p_1,p_2, \dots, p_d),\lambda) \mapsto L((p_1,p_2, \dots, p_d),\lambda):= \beta  \sum_{j=1}^d a_j \, p_j +    \frac{1}{1-q}  \,(\sum_{j=1}^d p_j^{q}-1) - \lambda \Big[\sum_{j=1}^d  \, p_j-1\Big]$$
where $\lambda$ is a constant.
Given $1\le j \le d$ the condition $\frac{dL}{dp_j}(p,\lambda)=0$ can be written as
\begin{equation}  \label{louc1}  \beta a_j + \frac{q}{1-q} p_j^{q-1} - \lambda=0
\end{equation}
or, equivalently, 
$$  p_j = \Big(\frac{1-q}{q}\Big)^{\frac{1}{q-1}} \; \big(\, \lambda - \beta a_j\,\big)^{\frac{1}{q-1}}=
\Big(\frac{1-q}{q}
\Big)^{\frac{1}{q-1}}\; \lambda^{\frac{1}{q-1}} \; \Big(\, 1 + \frac{\beta a_j}{\lambda(1-q)}(q-1)\,\Big)^{\frac{1}{q-1}}.
$$
subject to the constraint that $\sum_{j=1}^dp_j=1$.
Taking $\lambda =\frac{1}{1-q}$ above we obtain 
$$  \tilde p_j =
\Big(\frac{1}{q}
\Big)^{\frac{1}{q-1}}\;  \Big(\, 1 + {\beta a_j}(q-1)\,\Big)^{\frac{1}{q-1}}
=
\Big(\frac{1}{q}
\Big)^{\frac{1}{q-1}}\; e_{2-q}^{\beta a_j}.
$$
Taking the normalization we conclude that 
$$
p_j=\frac{\tilde p_j}{\sum_{i=1}^d \tilde p_i}
= \frac{e_{2-q}^{\beta a_j}}{\sum_{i=1}^d\; e_{2-q}^{\beta a_i}}
$$
for each $1\le j \le d$. This proves the lemma.
\end{proof}

\begin{remark} \label{rtu}
 Given $0<q<1$ and $A=(a_1,a_2), B=(b_1,b_2)$, after a tedious computation one can show that
\begin{equation}
\frac{d}{d \beta} P_q  (A +\beta\, B)|_{\beta=0}\,\neq \int B d p_A=b_1\, \frac{ e_{2-q}^{\,a_1}}{\sum_{r=1}^2 e_{2-q}^{\, a_r}  }\,  +\, b_2  \, \frac{ e_{2-q}^{\,a_2}}{\sum_{r=1}^2 e_{2-q}^{\, a_r}  },
\end{equation}
where $p_A$ is the $q$-equilibrium state for $A=(a_1,a_2).$ In this way, in the non-extensive case,  the derivative of the pressure does not behave exactly like in the classical extensive case.
In the special case that  $n=2$, $q=1/2$ and $A=(a_1,a_2)=(3,7)$,
the graph  of the function
$$\beta \to P_q(\beta \,A)$$
(see Figure \ref{q-fig1})
indicates that  this function is neither concave, nor convex, nor monotonous in the interval $(-0.5,1.5)$. This is due to the fact that the entropy $p \to H_q(p)$ may be convex or concave depending on $q$ (see Figure \ref{fig1}) and this somehow  influence what is observed in the pressure when the value of $\beta$ changes.
\begin{figure}[htb] \label{q-fig1}
\centering
\includegraphics[width=12cm]{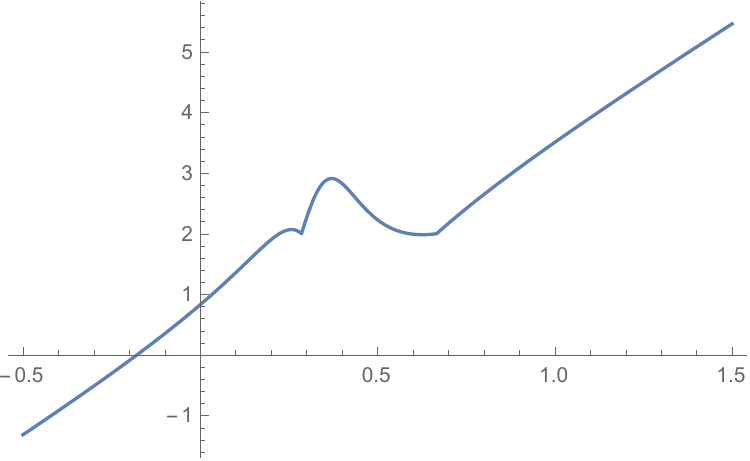}
\caption{The graph of the function $\beta \to P_q(\beta \,A)$, when $\beta\in (-0.5,1.5)$, for $q=1/2$, and $A=(a_1,a_2)=(3,7)$, obtained in Mathematica.}
\end{figure}
Similarly, taking $\beta=1$ and fixed potential $A=(a_1,a_2)=(3,7)$ the graph obtained in Mathematica for
 $q \to P_q(A)$ 
indicates that this function  is neither concave, nor convex, nor  monotonous.
\end{remark}

Finally, in what follows we illustrate how Lemma~\ref{pr1} can be used to compute the non-dynamical $q$-pressure function.

}

\begin{example}  \label{firdu}  If $q=1/3$, $n=2$, $\beta =1.2, a_1=0.5, a_2=0.8$,    $q$-pressure for $ \beta A$ is equal to $1.6895..$ and the $q$-equilibrium state $p$  is $p=(p_1,p_2)$ given by \begin{equation} \label {ryy} p_1=  0.3172...  =  \frac{ e_{2-q}^{\,\beta a_1}}{e_{2-q}^{\,\beta a_1 } + e_{2-q}^{\,\beta a_2 }  }\,\,\text{and}\,\, p_2=  0.6828...  =  \frac{ e_{2-q}^{\,\beta a_2}}{e_{2-q}^{\,\beta a_1 } + e_{2-q}^{\,\beta a_2 }  } .
\end{equation}

\end{example}

\begin{example} \label{loloi} In the case $q=1/2$ we get that the probability vector $(p_1,p_2)$ attaining the maximum in Lemma~\ref{pr1}
is given by
$$ p_1 = \frac{( a_2 b -b)^2}{8 - 4 a_1 b - 4 a_2 b + a_1^2 b^2 + a_2^2 b^2}$$
and
$$ p_2 = \frac{( a_1 b -b)^2}{8 - 4 a_1 b - 4 a_2 b + a_1^2 b^2 + a_2^2 b^2}.$$

\end{example}

\subsection{The dynamical $q$-pressure} \label{FIR}

Given $0<q<1$, and a continuous potential $A: \Omega \to \mathbb{R},$ the 
\emph{dynamical $q$-pressure of $A$ on $\mathcal G$}
(or \emph{$q$-pressure} for short, when no confusion is possible) is defined as
\begin{equation} \label{PP2}
P_q^{\mathcal{G}} (A) = \sup_{\mu \in \mathcal{G} }\, \Big\{  H_q (\mu) +\int A \,d \mu \Big\}
 =\sup_{\mu \in \mathcal{G} }\,\Big\{  \int \log_q \Big(\frac1{J}\Big) \,d \mu   + \int A \,d \mu \Big\},
\end{equation}
Any probability measure $\mu\in \mathcal{G}$ attaining the supremum  \eqref{PP2}  is called a \emph{$q$-equilibrium state} associated to $A$, for the $q$-pressure function in $\mathcal{G}$. A difficulty that arises in this context is that the set $\mathcal{G}$ is non-compact.
We define, analogously, for  a
continuous potential $A: \Omega \to \mathbb{R}$  and $0<q<1$,   the 
\emph{dynamical $q$-pressure of $A$ on
$\mathcal{M}_{\text{inv}}(\sigma)$} 
is defined by
\begin{equation} \label{PP3}
P_q^{\mathcal{M}_{\text{inv}}(\sigma)} (A) = \sup_{\mu \in \mathcal{M}_{\text{inv}}(\sigma)}\, \Big\{  H_q (\mu) +\int A \,d \mu \Big\}.
\end{equation}
Any probability measure $\mu$ attaining the supremum in  \eqref{PP3}  is referred to as a\emph{$q$-equilibrium state} associated to $A$, for the $q$-pressure on the compact set
$\mathcal{M}_{\text{inv}}(\sigma)$.

{ 
\begin{remark}
\label{rmk.existencea}
 Given $0<q\leq 1$, by boundedness and upper semi-continuity of the $q$-entropy function 
 (recall \eqref{hH} and Lemmas \ref{le:Hconcave} and \ref{uup})
 and compactness of $\mathcal{M}_{\text{inv}}(\sigma)$, there always exist  $q$-equilibrium states associated to the continuous potential $A$ on 
 $\mathcal{M}_{\text{inv}}(\sigma)$.
Moreover, even though
 $\mathcal{G}$ is a non-compact analytic manifold,
the existence of a $q$-equilibrium state for $A$ on $\mathcal G$ that attains the supremum in \eqref{PP2} follows by boundedness, upper semi-continuity and concavity of the $q$-entropy function. 
\end{remark}
}

\begin{remark}
The space $\cG$ coincides with the space of extensive equilibrium states for Lipschitz continuous potentials.
For that reason, using Rohklin formula \eqref{eq:Rohklin}, the classical pressure $P(A)$ of a Lipschitz continuous potential $A$ satisfies
$$
P(A)=\sup_{\mu \in \mathcal{G} }\,\Big\{  \int \log J \, d \mu   + \int A \,d \mu \Big\},
$$
where $J$ denotes the Jacobian associated to $\mu$.
Moreover, from \eqref{hH} we get that 
$$
P(A) \geqslant P_q(A) \qquad \text{for every $0<q<1$ and $A\in C(\Omega,\mathbb R)$}
$$
\end{remark}

\medskip 
We now proceed to study the differentiability of the pressure function. Let us first recall some necessary notions.
We say that $B: \Omega \to \mathbb{R}$ is \emph{cohomologous to $A: \Omega \to \mathbb{R}$}  if there exists  a continuous function $f: \Omega \to \mathbb{R}$, and $c\in \mathbb{R}$ such that
\begin{equation} \label{rwt} A= B + f - (f \circ \sigma) -c
\end{equation}
As $\mathcal{G}\subset \mathcal M_{inv}(\sigma)$, if ~\eqref{rwt} holds then $\int A \,d\mu =\int B\, d\mu - c $ for every $\mu\in \cG$ and, consequently, 
the $q$-equilibrium states associated with any two cohomologous potentials 
$A$ and $B$ are the same.

\smallskip
In \cite{BCMV}, Bi\'s et al  introduce an axiomatic definition of pressure function as any map $\Gamma: {\bf B} \to \mathbb R$
on a Banach space ${\bf B}\subset L^\infty(\Omega)$ that satisfies the following properties: for any $A,B \in {\bf B}$ and $c\in \mathbb R$,
\begin{itemize}
\item[(H1)] (monotonicity)
	if $A\leqslant B$ then $\Gamma(A) \leqslant \Gamma(B)$;
\item[(H2)] (translation invariance)
	$\Gamma(A+c)=\Gamma(A)+c$;
\item[(H3)] (convexity)
	$\Gamma(\alpha A+(1-\alpha)B) \leqslant \alpha\Gamma(A)+(1-\alpha)\Gamma(B)$ 
	for every $0\leqslant \alpha \leqslant 1$.
\end{itemize}
It is easy to check from its definition in ~\eqref{PP2} that $P_q: C^0(\Omega,\mathbb R) \to \mathbb R$ is a pressure function. Additionally, as the supremum in ~\eqref{PP2} is taken over probability measures that are $\sigma$-invariant then the pressure function $P_q$ also satisfies:

\begin{itemize}
\item[(H4)] (coboundary invariance)
	$P_q(A+f\circ \sigma - f) = P_q(A)$ for every $A,f \in C^0(\Omega,\mathbb R)$.
\end{itemize}
In view of \cite{BCMV} we obtain the following immediate consequence:

{ 
\begin{corollary}\label{cor:BCMV}
In the locus of convexity,
the $q$-pressure function $P_q: \mbox{Lip}(\Omega) \to \mathbb R$ is Gateaux differentiable if and only if it has a unique $q$-equilibrium state on $\mathcal M_{inv}(\sigma)$ for every Lipschitz continuous potential.
\end{corollary}
}

\color{black}

The next proposition allows to obtain non-extensive equilibrium states from the classical extensive framework.

{

\begin{proposition}\label{esteaq1} 
    Let $\mu\in \mathcal G$ be the extensive equilibrium state for the potential $\log J$. Then, the following holds:
    \begin{enumerate}
        \item $\mu$ is a $q$-equilibrium state for the potential $-\log_q(\frac1J)$;
        \item $\mu$ is the unique $q$-equilibrium state for the potential $-\log_q(\frac1J)$ which belongs to $\mathcal G$;
        \item $P_q(-\log_q \frac1J)=0$.
    \end{enumerate}.
\end{proposition}

\begin{proof}  
Fix $\mu\in \mathcal G$ as the extensive equilibrium state for the potential $\log J$. Hence
$$
0=P(\log J)=
H_q(\mu) + \int \log J \, d\mu
=  
\int \log_q\Big(\frac1J\Big) \,d\mu
- \int \log \Big(\frac{1}{J}\Big) d \mu.
$$
Now, for each $\tilde \mu\in \mathcal{G}$, if one denotes its Jacobian by $\tilde J$, it follows from \eqref{oret} that 
\begin{equation}
    \label{eq:conseqhq}
\int \log_q \Big(\frac{1}{  J(x)}\Big)d \tilde\mu (x) -  \int \log_q \Big(\frac{1}{\tilde{J}(x)}\Big) d \tilde \mu (x) \ge 0 
\end{equation}
and, in particular, recalling ~\eqref{HQ},
\begin{align*}
P_q\Big(-\log_q \Big(\frac{1}{{J}(x)}\Big) \Big)
    & = \sup_{\tilde \mu \in \mathcal G}
\Big\{ H_q(\tilde \mu) + \int -\log_q \Big(\frac{1}{  J(x)}\Big)d \tilde\mu (x) \Big\}  \\
    & = \sup_{\tilde \mu \in \mathcal G}
\Big\{ \int \log_q \Big(\frac{1}{\tilde{J}(x)}\Big) d \tilde \mu (x) - \int \log_q \Big(\frac{1}{  J(x)}\Big)d \tilde\mu (x) \Big\}
\ge 0.
\end{align*}
Moreover, the probability measure $\tilde \mu=\mu$ is the only measure in $\mathcal G$ for which the equality in ~\eqref{eq:conseqhq} is attained. This implies simultaneously that $P_q\Big(-\log_q \Big(\frac{1}{{J}(x)}\Big) \Big)=0$ and that $\mu$ is the unique $q$-equilibrium state in $\mathcal G$ for the potential $-\log_q \Big(\frac{1}{{J}(x)}\Big)$.
This proves the proposition.
 \end{proof}

\begin{remark}
We observe that while $-\log J=\log\Big(\frac1{J}\Big)$, the $\log_q$ terms appearing in the non-extensive thermodynamic formalism obey to a symmetry in parameters. More precisely, if $A=-\log_q(\frac1J)$ then (observing the relation between $\log_q$ and $e_q$ in \eqref{prop1}) 
\begin{equation} \label{jau}\log J=- \log  ( e_q ^{-A})= \log (e^{ A}_{2-q}).
\end{equation}

\end{remark} 
}

\section{Normalized potentials and eigenfunctions of transfer operators} \label{prelim}

Given the potential $ A $, we would like to determine the value of the $q$-dynamical pressure
$ P_q (A) $ in a procedure similar to that obtained when finding the eigenfunction of the Ruelle operator (as in \cite {PP}).
There are some technical and conceptual difficulties in
trying to implement an alike strategy,  due to the fact that $e_q^{x+y} \neq e_q^x e_q^y$.
For that reason we need to introduce different  and alternative concepts which are also natural to study.

\subsection{Normalizable potentials}

\begin{definition} \label{klr12}
Given $q>0$ and a   Lipschitz continuous potential $A$, we say that $A$ is \emph{$q$-normalized} if
$$
\sum_{a=1}^d e_q^{A(a x)} =1, \quad \text{for every $x\in \Omega$}.
$$
In case there exists a Lipschitz continuous function $\varphi_{{A}}$ and a constant $c_A$ 
 such that
\begin{equation} \label{atchr27} \sum_{a=1}^d e_q^{{A}(a x) +  \varphi_{{A}} (a x) - \varphi_{{A}}(x) - c_{{A}}}=1 \quad \text{for every $x\in \Omega$}
\end{equation}
then we say that  ${A}$ is \emph{$q$-normalizable}.
\end{definition}

It is clear from the definition that the potential $A=0$ is normalizable (there exists the trivial solution $\varphi=0$ and $c= \log_q(d)$).
Moreover, in the extensive context, corresponding to $q=1$, equation~\eqref{atchr27} can be written as 
\begin{equation} \label{atchr27b}
\cL_A( e^{\varphi_A})(x)= \sum_{a=1}^d e^{{A}(a x) +  \varphi_{{A}} (a x)}=e^{-c_A} \, e^{\varphi_{{A}}(x) }\quad \text{for every $x\in \Omega$},
\end{equation}
meaning that $e^{\varphi_A}$ is an eigenfunction for the transfer operator $\cL_A$ associated to the eigenvalue $e^{-c_A}$. In particular, in the extensive framework for every Lipschitz continuous potential $A$ the transfer operator $\mathcal L_A$ has a spectral gap on the space of Lipschitz continuous observables, hence every Lipschitz continuous potential is normalizable.

\smallskip
In the non-extensive the situation changes drastically. More precisely, if $q\neq1$ then 
$$e^{x + y + (1-q) x\,y}_{q} = e^x_q \, \,e^y_q
$$
for every $x,y$ in the domain of the $q$-exponential (cf. \eqref{prop2}) and consequently the 
solutions of \eqref{atchr27} become unrelated to eigenvalues of any $\tilde q$-transfer operator. This is a major obstruction to obtain to the construction of the
non-extensive thermodynamic formalism using the tools developed in the extensive framework.

\begin{remark}
 In Examples \ref{explimeq} and \ref{jana} we illustrate the fact  that the existence of a pair $(\varphi_A,c_A)$ satisfying ~\eqref{atchr27} in the non-extensive setting can be much more subtle than its extensive counterpart. Solutions to the normalization problem will be
produced via the implicit function theorem, and this will give also a new method for the solution of eigenfunctions for transfer operators in the classical extensive framework (we refer the reader to Sections~\ref{apen0} and \ref{apen1} for more details).
This method has the advantage of showing that for each $0<q<1$, the eigenfunction $\varphi_{{A}}$ and the constant $c_{{A}}$ vary
differentially with the potential $A$.   
\end{remark}

\smallskip
The next lemma ensures that the solutions of  ~\eqref{atchr27} are related to solutions of coboundary equations required for the analysis of the $q$-pressure problem, namely that under some normalization condition, every Lipschitz continuous potential is cohomologous to a potential of the form $-\log_q(\frac1J)$.

\begin{lemma} \label{analy} Take $0<q<1$ and let $A:\Omega \to \mathbb{R}$
be a Lipschitz continuous potential for which there exists a Lipschitz continuous function $\varphi_{{A}}$ and a constant $c_A$ 
 such that
\begin{equation*}\sum_{a=1}^d e_{2-q}^{{A}(a x) +  \varphi_{{A}} (a x) - \varphi_{{A}}(x) - c_{{A}}}=1 \quad \text{for every $x\in \Omega$,}
\end{equation*}
and that all summands above are strictly positive. Then, there exist a Lipschitz continuous positive Jacobian $J:\Omega \to \mathbb{R}$,  a Lipschitz continuous function   $\varphi:\Omega \to \mathbb{R}$  and $c\in \mathbb{R}$ so that
\begin{equation} \label{oia1}
 -\, \log_q \Big(\frac{1}{J}\Big)     = A( x) + \varphi ( x) - \varphi(\sigma(x)) - c.
\end{equation}
\end{lemma}

\begin{proof}
By assumption, the function $J(x)= e_{2-q}^{A( x) + \varphi ( x) - \varphi(\sigma(x)) - c}$ is a Lipschitz continuous Jacobian. 
Now, using that  $e_q^x =( e_{2-q} ^{-x})^{-1}$, the latter is equivalent to
$$ \frac{1}{J(x)}  =( e_{2-q}^{A( x) + \varphi ( x) - \varphi(\sigma(x)) - c})^{-1}= e_q^{-\,(A( x) + \varphi ( x) - \varphi(\sigma(x)) - c)}$$
and so ~\eqref{oia1} holds.
\end{proof}

\smallskip

\begin{remark}\label{kjl}
    The previous lemma ensures that if $A$ is $(2-q)$-normalizable with the pair $(\varphi_A,c_A)$ then $J(x)= e_{2-q}^{A( x) + \varphi_A ( x) - \varphi_A(\sigma(x)) - c_A}$ is a Jacobian. This, combined with \eqref{prop16}, ensures that
$$
e_{2-q}^{ -  \log_q (\frac{1}{ J})} = J = e_{2-q}^{A +  \varphi_{A}
 - (\varphi_{A} \circ \sigma) - c_{A}}
 $$
 and, consequently, 
$-\log_q (\frac{1}{ J})=A +  \varphi_{A}
 - (\varphi_{A} \circ \sigma) - c_{A}$.
 This proves that $-\log_q (\frac{1}{ J})$ is cohomologous to  $A - c_{A}$ and, by properties (H2) and (H4) on the pressure function,
 it follows that: (i)
$P_q (-\log_q (\frac{1}{ J}))=P_q(A) - c_{A}$,  and (ii) both potentials have the same $q$-equilibrium states.
\end{remark}

\subsection{Proof of Theorem~\ref{mthm1}} \label{tryu}

 Given a Lipschitz continuous potential  $A:\Omega \to \mathbb{R}$, denote by $\varphi$ and  $c$  the solutions
of equation \eqref{atchew1} and write  $J(x)= e_{2-q}^{A( x) + \varphi ( x) - \varphi(\sigma(x)) - c}$ for the corresponding Jacobian.
By Lemma \ref{analy} and Remark~\ref{kjl} there exists a Lipschitz continuous function $\varphi$ and $c\in \mathbb R$ so that $-\log_q\Big(\frac1J\Big) = A( x) + \varphi ( x) - \varphi(\sigma(x)) - c$ and  
$P\Big(-\log_q\Big(\frac1J\Big)\Big)=P(A)-c$
Moreover, from  Proposition~\ref{esteaq1},
the classical equilibrium state for the potential  $\, \log (J)$  is a $q$-equilibrium state for $-\log_q\Big(\frac1J\Big)$ (hence for $A$ as well).
The reasoning above shows that $c$ is unique.

\smallskip
In order to complete the proof of the theorem it remains to prove that $P_q(A)=c$. As before, we parameterize an arbitrary $\tilde \mu\in \mathcal G$ by its Jacobian $\tilde J$. By definition of $q$-pressure and invariance of probability measures in $\mathcal G$,
\begin{align*}
  P_q (A) 
  & = \sup_{{\tilde \mu} \in \mathcal{G} }\,\Big[ H_q(\tilde \mu) +  \int A \, d{\tilde \mu} (x) \Big] \\
  & = \sup_{{\tilde \mu} \in \mathcal{G} }\,\Big[ \int \log_q (\frac{1}{\tilde J(x)})d {\tilde \mu} (x) +  \int A(x) \,d {\tilde \mu} (x)  \Big] \\
  & =  
   \sup_{{\tilde \mu} \in \mathcal{G} }\,\Big[  \int -(A(x)  + \varphi ( x) - \varphi(\sigma(x)) - c\,) \,d {\tilde \mu} (x) +  \int A\, d {\tilde \mu} (x) \Big]\, = \, c.
\end{align*}
This finishes the proof of the theorem.
\hfill $\square$

\bigskip

The following is a direct consequence of the proof of Theorem~\ref{mthm1} and the classical thermodynamic formalism.

\begin{corollary} \label{mixi} 
Given $0<q<1$ and a $(2-q)$-normalizable Lipschitz continuous potential $A$, there exists a unique $q$-equilibrium state $\mu_{q,A}$ for $A$ in $\mathcal G$. Moreover, $\mu_{q,A}$
is exact, has exponential decay of correlations for Lipschitz observables and it varies differentiably with respect to $A$.
\end{corollary}

 \subsection{The $(2-q)$-Ruelle operator equation} \label{qR}

 \smallskip

 In this subsection we will relate solutions for the $q$-Ruelle operator with  solutions for the classical Ruelle operator.
Given $0<q<1$ and a Lipchitz continuous potential $A:\Omega \to \mathbb{R}$,  if  for some $\varphi_A$ and $c_A$ we get  for any $x\in \Omega$
 \begin{equation} \label{gre1}
 \sum_{a=1}^d e_{2-q}^{A(a x) + \varphi_A (a x) - \varphi_A(x) - c_A}=1,
\end{equation}
then from Theorem~\ref{mthm1} the value $c_A$ satisfies $P^{q}(A) =c_A$,

\begin{remark}
In case there exists a continuous function $\varphi_A$ and $c_A\in \mathbb{R}$  satisfying  the previous relation we say that $q$-Ruelle theorem equation can be solved. If this is the case,  from \eqref{eeste1} a $q$-equilibrium state can be given in terms of the solutions for the Ruelle equation which are known in the classical Thermodynamic Formalism.    
\end{remark}

\smallskip

We denote by $\mathcal{L}_B$ the classical Ruelle operator for the potential
$B:\Omega \to \mathbb{R}$.
Given an $\alpha$-H\"older continuous function $F:\Omega \to \mathbb{R}$, $0<\alpha \leqslant 1$, we denote by
$$ |F|_\alpha = \sup_{ x, y \in \Omega, \, x \neq y} \frac{|F(x) - F(y)|}{d(x,y)^\alpha} 
$$
its $\alpha$-H\"older constant.
It is not hard to check that if $f$ is Lipschitz continuous then $\log_q (f)$ and $e_q^f$ are Lipschitz continuous and that 
\begin{equation*} \label{est1L} |\log_q (f)|_1 \leqslant |f|_1 \sup (|f|^{-q})
\end{equation*}
and
\begin{equation*} \label{est2L} |e^f_q|_1 \leqslant \frac{1}{1-q} \, |f|_1 \, \sup (|1+ (1-q) f|^{\frac{q}{1-q}}).
\end{equation*}
We will need the next instrumental result.

\medskip 
We consider the following space of potentials.
Given $q>0$ ($q\ne 1$), let us denote by $\mathcal{H}_q$ the space of Lipschitz continuous functions
$A:\Omega \to \mathbb{R}$ such that 
\begin{equation} \label{utr} A(x) >\frac{1}{q-1}, \quad\text{for all $x\in \Omega$.}
\end{equation}
By compactness of $\Omega$, the latter condition ~\eqref{utr} ensures that $(1-q)A$ is bounded away from $-1$, hence it makes sense to define the following.

\begin{definition}
\label{def:changeqA}
Given $q>0$ with  $q \neq 1$ and a Lipschitz continuous potential $A \in\mathcal{H}_q$ consider the potential $A_q: \Omega \to \mathbb R$ given by 
\begin{equation} 
\label{i1} A_q= \frac{1}{1 -q}   \log (  1 + (1-q) A) = \log ( e_q^{ A} ).
\end{equation}
Note that the potential $A_q$ depends in a differentiable fashion from $A$.
\end{definition}

Given a  potential $A\in \mathcal{H}_q$,
the \emph{$q$-Ruelle operator} $\mathcal{L}_{A,q} : C^0(\Omega,\mathbb R) \to C^0(\Omega,\mathbb R)$ is defined by $ f \mapsto  \mathcal{L}_{A,q} (f) $, where 
\begin{equation} \label{Ru}  \mathcal{L}_{A,q} (f) (x) = \sum_{a} e_q ^{A (a x)}\, f(a x)=\sum_{a} (1  + (1-q) \,A(a x))^{\frac{1}{ 1-q}}\, f(a x) .
\end{equation}

\color{black}

\begin{lemma}
\label{kert}
Fix $0<q<1$ and a potential $A\in \mathcal H_q$. The following properties hold:
\begin{enumerate}
    \item the operators
$\mathcal{L}_{A_q}$ and $\mathcal{L}_{A,q}$ coincide;
    \item there exists a leading eigenvalue $ \lambda_{A,q}>0$ and a unique normalized positive eigenfunction  $\varphi_{A,q}$,  in the sense that
\begin{equation}
\label{iii} \mathcal{L}_{A,q} (\varphi_{A,q})=   \lambda_{A,q}\, \varphi_{A,q};
\end{equation}
\item $P({A_q})=\log \lambda_{A,q}$ is the classical pressure for the potential $A_q,$
\item $P(A)\ge P(A_q)$.
\end{enumerate}
\end{lemma}

\begin{proof}
    Note that, for any given continuous  function $f:\Omega \to \mathbb{R}$ and every $x\in \Omega$,
    \begin{align*}
   \mathcal{L}_{A_q} (f) (x) 
   & = \sum_{a=1}^d e^{A_q(a x)} f(a x) =  \sum_{a=1}^d e^{ \frac{1}{1 - q}   \log (  1 + (1-q) A(a\,x))} f(a x) \\
   & = \sum_{a=1}^d(1  + (1-q) A( a x))^{\frac{1}{ 1-q}}f(a x)= \mathcal{L}_{A,q} (f) (x).
   \label{i51}   
    \end{align*}
This proves item(1). Items (2) and (3) are a direct consequence of (1) together with the classical Ruelle's theorem for the transfer operator $\mathcal{L}_{A_q}$.
Finally, denoting by $\mu_{A,q}$  the classical equilibrium state for the potential $A_q$ one knows that 
  $$ P(A_q) = h( \mu_{A,q}) + \int A_q \, d  \mu_{A,q}.$$ 
In particular, the classical variational principle together with the fact that 
$$\frac{1}{1 -q}   \log (  1 + (1-q) x) <x$$
for each $0<q<1$ and $x>0$, ensures that 
\begin{align*}
    P(A) & \ge 
    h( \mu_{A,q}) + \int A \, d  \mu_{A,q} \\
    & \ge h( \mu_{A,q}) + 
    \int \frac{1}{1 -q}   \log (  1 + (1-q) A) \;  d \mu = P(A_q).
\end{align*}
This proves item (4) and completes the proof of the lemma.
\end{proof}

Some comments are in order. First, 
note that equation  \eqref{iii} is quite different from \eqref{atchew1}, which involves the $(2-q)$-exp map.
\medskip

\medskip

Note  from Lemma \ref{kert} that given a general $A$, we define $A_q$ (via \eqref{i1}), for which  there exists a continuous function $\varphi_{A,q}$ and $c_A\in \mathbb R$ so that 
\begin{equation} \label{liga1}  \sum_{a=1}^d e^{ A_q(a x) + \log \varphi_{A,q}(ax)  - \log \varphi_{A,q} (x) - \log \lambda_{A,q}}=1,\qquad
\text{for all $x\in \Omega$}.
\end{equation}
Observing that ~\eqref{i1} is equivalent to $e^{ A_q(x) }= e_q^{A(x)}$ 
one concludes that 
 \begin{equation} \label{i108} \sum_{a=1}^d e_q^{A(a x)} e^{ \log \varphi_{A,q}(a x)  - \log \varphi_{A,q} (x) - c_{A,q}} 
 = \sum_{a=1}^d e^{A_q(a x)} e^{ \log \varphi_{A,q}(a x)  - \log \varphi_{A,q} (x) - c_{A,q}}=1.
 \end{equation}

 Our strategy in the next subsections  is to explore relations of the form $e_{2-q}^x = e_q^y \, e^z$, for some $x,y,z$.
 \medskip

\subsubsection{The case $q=1/2$} \label{1/2}

In the case $q=1/2$ we are able to relate the extensive  Thermodynamic Formalism to the non-extensive  Thermodynamic Formalism. This in a sense that given a general  potential $A$ (and  the associated $A_q$ as in \eqref{i1}) we will be able to exhibit a potential $B$ such that one can express  a solution $\varphi_B$ for the $(2-q)$  equation 
\begin{equation} \label{gre1}
\sum_{a=1}^d e_{2-q}^{B(a x) + \varphi_B (a x) - \varphi_B(x) - c_B}=1.
\end{equation}
This will be achieved via  finding the eigenfunction of the  classical Ruelle operator for the potential $A_q$ (see \eqref{vol21}).

In this way one can produce examples getting a normalized equation for the non-extensive Ruelle operator for the parameter $2-q$.

In order to do that we claim that when $q=1/2$, given $r$ and $a$
$$e^r = e_q^a$$ is equivalent
to
$$ a = -2 \pm  2 e^{r/2}.$$

Below we will choose the option corresponding to $a = -2 +  2 e^{r/2}$ in our computations. We will not present
the details of all computations to derive the potential $B$; just the final expressions.

Consider the function $g$ given by
\begin{equation} \label{vol136} g (a1,a2,C, a)=  \frac{-4 + (2 +a) \, \sqrt{e^{ a1-a2-C } }\,[2 - (a1-a2-C) ]}{ (2+ a) \sqrt{e^{ a1-a2-C } }}.
\end{equation}

One can  show that
 \begin{equation} \label{vol1}e_{2-q}^{ g (a1,a2,C, a) +( a1-a2-C)}=1/ 4 (2 + a)^2\, e^{ a1 - a2 - C}= e_{q}^a \,  e^{ a1 - a2 - C}.
 \end{equation}
\smallskip

For the parameter  $(2-q)$, a potential  $B$ and $\varphi_B$  as in \eqref{gre1} will be derived   from $A$ (or, from the associated  $A_q$):
given the potential $A_q$ expressed via  \eqref{i1}, consider the  associated $\varphi_{A,q}$ and $\lambda_{A,q} = e^{c_{A,q}}$ obtained from Lemma \ref{kert}. Finally, take
 \begin{equation} \label{vol2} B=  \frac{-4 + (2 +A)  \sqrt{ \frac{\varphi_{A,q}}{ c_{A,q} (\varphi_{A,q}\circ \sigma)} }[2 -  \log (\varphi_{A,q})+\log (\varphi_{A,q}\circ \sigma)+c_{A,q})  ]}{ (2+ A)\,\sqrt{ \frac{\varphi_{A,q}}{ c_{A,q} \,\,\,(\varphi_{A,q}\circ \sigma)} } },
 \end{equation}
$\varphi_B=\log\varphi_{A,q}$ and $c_B=c_{A,q}.$
Then, given $x$ we get from \eqref{vol1}, \eqref{liga1} and \eqref{i108}
$$  \sum_{a=1}^d e_{2-q}^{A(a x) + (\varphi_A (a x) - \varphi_A(x) - c_A)}=  \sum_{a=1}^d e_q^{A(a x)} e^{ \log \varphi_{A,q}(a x)  - \log \varphi_{A,q} (x) - c_{A,q}}=$$
 \begin{equation} \label{vol21} \sum_{a=1}^d e^{A_q(a x)} e^{ \log \varphi_{A,q}(a x)  - \log \varphi_{A,q} (x) - c_{A,q}}=1.
 \end{equation}

In this way we show the existence of solutions for the $(2-q)$-Ruelle theorem equation \eqref{gre1} from classical results for the potential $A_q$.  
 Note that we can also  obtain $A$ from $A_q$ via $A= \frac{e^{(1-q) \,A_q}-1}{1-q}.$ 
 In this way we can begin our reasoning taking some $A_q$ of our choice. Explicit extensive solutions of eigenfunction and eigenvalue  for a family of potentials that depends on infinite coordinates on $\Omega$ are presented in \cite{CDLS}. Therefore, taking as $A_q$ a potential described in \cite{CDLS} one gets   solutions for the $q$-Ruelle operator equation for a family of nontrivial potentials $A$  on the non-extensive setting.

Note also that from  \eqref{vol2} we get that  the potential $A$ depends on a differentiable way from the potential $A$. Remember that the potential $A_q$ depends in a differentiable fashion $A$ and vice versa.

\medskip

\subsubsection{The general case $0<q<1$} \label{qq} The procedure is similar to the previous one but the solution is not so explicit and simple.
 Define $g$ by
 $$ g (a1,a2,C, a)= $$
 $$\frac{e^{-a1} (e^{a2 +c} (1+a-aq)^{\frac{1}{q-1}}(e^{a1-a2-c}(1+a-aq)^{\frac{1}{q-1}})^q }{q-1}  -$$
 $$ \frac{e^{a1} ( 1 + a2 + c + a1 (q-1) -a2 \,q - c q }{q-1 }.$$

 Expression \eqref{prop15} is very helpful on this section.

 One can show that
 \begin{equation} \label{vol9}e_{2-q}^{ g (a1,a2,C, a) +( a1-a2-C)}= e_{q}^a \,  e^{ a1 - a2 - C}.
 \end{equation}

 Given $A$, consider the potential $A_q$ given by \eqref{i1} and the  associated eigenfunction $\varphi_{A,q}$ and eigenvalue $\lambda_{A,q} = e^{c_{A,q}}$ obtained from Lemma \ref{kert}.

 Denote
 $$ A= g ( \log \varphi_{A,q}, \log (\varphi_{A,q} \circ \sigma),c_{A,q},A ), $$
 and $\varphi_A=\log\varphi_{A,q}$ and $c_A=c_{A,q}.$

 Then, given $x$ we get from \eqref{vol9} and \eqref{i108}, in the same way as  in  \eqref{vol21}
 \begin{equation} \label{vol213}  \sum_{a=1}^d e_{2-q}^{A(a x) + \varphi_A (a x) - \varphi_A(x) - c_A}=1.
 \end{equation}

\section{Asymptotically sub-additive $q$-potentials and asymptotic pressure} \label{SubP}

This section is devoted to the proof of Theorem~\ref{Mpressure}. We first describe the sequence of potentials appearing in the family of transfer operators $\mathfrak{L}_n$, for $n\in \mathbb N$.

\subsection{Sequences of sub-additive potentials and transfer operators}\label{SubP1}
Fix $0<q<1$ and a Lipchitz continuous potential $A: \Omega \to \mathbb{R}$.
For each $n\in \mathbb{N}$ and $x\in \Omega$ we write $S_n A(x)= \sum_{j=0}^{n-1} A(\sigma^j(x))$.
For each $n \in \mathbb{N}$, the linear operator $\mathfrak{L}_n$
defined by 
$$\mathfrak{L}_n (f) (x)= \sum_{\sigma^n(y)=x}  e_q^{ S_n A(y) } \, f(y)=
  \sum_{\sigma^n(y)=x} (1 + (1-q) \,S_n A(y))^{\frac{1}{1-q} }\, f(y)$$
\begin{equation} \label{kle2} =\sum_{\sigma^n(y)=x}\, e^{ \frac{1}{1-q} \, \log  (\, 1 + (1-q) \,S_n A(y)\,)}\, f(y).
\end{equation}
is positive and bounded.
From \eqref{kle2} it is natural to investigate the family of Lipschitz continuous functions
\begin{equation*} \varphi_{n}(y)= \frac{1}{1-q} \, \log  (\, 1 + (1-q) \,S_n A(y)\,),\,\, n \in \mathbb{N},
\end{equation*}
which we refer to as the \emph{family of $q$-potentials associated to $A$} (we omit the dependence of the sequence on $q$ and the potential
$A$ for notational simplicity).

\begin{lemma}\label{le:phin}
If the potential $A$ is non-negative then the sequence $(\varphi_n)_{n\in \mathbb N}$ is sub-additive, i.e.
$ 
\varphi_{m+n} \leqslant  \varphi_{m} \circ \sigma^n+\varphi_{n}
$ 
for every $m,n \in \mathbb N$.
\end{lemma}

\begin{proof}
Given $0<q<1$ and $a,b>0$ one has that
\begin{equation*} \label{kle3}
 \frac{1}{1-q} \, \log  (\, 1 + (1-q) \,(a +b) \,) \leqslant
  \frac{1}{1-q} \, \log  (\, 1 + (1-q) \,a \,)+
  \frac{1}{1-q} \, \log  (\, 1 + (1-q) \,b \,).
  \end{equation*}
The previous expression can be obtained by taking exponential on both sides of the inequality and using that $1+x+y \le (1+x)(1+y)$ for every $x,y>0$. In consequence,
  \begin{equation} \label{kle7}   \frac{(\, 1 + (1-q) \,(a +b) \,)^{ \frac{1}{1-q}}}{ (\, 1 + (1-q) \,a \,)^{\frac{1}{1-q}  }\,\,\,(\, 1 + (1-q) \,b\,)^{\frac{1}{1-q}  }} \leqslant 1.
  \end{equation}
From \eqref{kle3} we get that $\varphi_{n+m} \leqslant \varphi_{n} +  (\varphi_{m} \circ \sigma^n)$ for all $m,n\in \mathbb{N}$, as desired.
\end{proof}

\begin{lemma}\label{le:phing}
Assume that $A: \Omega \to\mathbb R$ is a continuous potential. Then the sequence $(\varphi_n)_{n\in \mathbb N}$ is asymptotically sub-additive, i.e.
there exists a sub-additive sequence $(\psi_n)_{n\in \mathbb N}$ such that
$$
\lim_{n\to\infty} \frac1n \|\varphi_{n}- \psi_n\|_\infty =0.
$$
In particular, given a $\sigma$-invariant and ergodic probability measure $\mu$,
$$
\lim_{n\to\infty} \frac1n \varphi_n(x) = \inf_{n\geqslant 1} \frac1n \int \varphi_n \, d\mu \qquad \text{for $\mu$-a.e. $x\in \Omega$.}
$$
and the map
$
\mathcal M(\sigma)\ni \mu \mapsto \lim_{n\to\infty} \frac1n \int \varphi_n \, d\mu
$
is upper-semicontinuous.
\end{lemma}

\medskip

\begin{proof}
In the case that $\inf_{x\in \Omega} A(x)\geqslant 0$, Lemma~\ref{le:phin} ensures that $(\varphi_n)_{n\in \mathbb N}$ is sub-additive and there is nothing to prove. Assume now that $\inf_{x\in \Omega} A(x)<0$ and choose the constant $c=-2\inf_{x\in \Omega} A(x)>0$. Then one can write
\begin{align}
\frac{1}{1-q} & \, \log  (\, 1 + (1-q) \,S_n A)(y)\,) \nonumber \\
		& = \frac{1}{1-q} \, \log  (\, 1 + (1-q) \,S_n (A+c)(y)\,)  \label{asymp1a} \\
		& +
		\frac{1}{1-q} \,  \log  \Big[ \frac{\, 1 + (1-q) \,S_n A)(y)\,}{\, 1 + (1-q) \,S_n (A+c)(y)\,} \Big].  \label{asymp2a}
\end{align}
The sequence $\psi_n(x)=\frac{1}{1-q} \, \log  (\, 1 + (1-q) \,S_n (A+c)(y)\,)$ appearing in \eqref{asymp1a} is sub-additive, by Lemma~\ref{le:phin}. Moreover, since $0<q<1$ and  
$\log y \le y$ for every $y\ge 1$
one can use 
$$
1\leqslant \frac{\, 1 + (1-q) \,S_n (A+c)(y)\,}{\, 1 + (1-q) \,S_n A(y)\,} \leqslant 1+ \frac{(1-q)cn}{1+(1-q) n \inf_{x\in \Omega} |A(x)|} \leqslant 4
$$
for every large $n\geqslant 1$,
to bound ~\eqref{asymp2a} in the following way
\begin{align*}
0\le \frac{1}{(1-q)} \,
	   \log  \Big( \frac{\, 1 + (1-q) \,S_n (A+c)(y)\,}{\, 1 + (1-q) \,S_n A)(y)\,} \Big) 
	 \leqslant \frac{4}{(1-q)}.
\end{align*}
This implies that $\frac1n\|\varphi_-\psi_n\|_\infty \le \frac{4}{(1-q)n} $ 
tends to zero as $n\to\infty$, which  proves the first assertion in the lemma.
The latter, together with Kingman's sub-additive ergodic theorem (see \cite[Theorem 10.1]{Walters})
implies that, for every $\sigma$-invariant and ergodic probability $\mu$ on $\Omega$,
\begin{equation} \label{vau2}\lim_{n \to \infty} \frac{1}{n}\,\psi_n(x) =\inf_{n\geqslant 1} \frac1n \int \psi_n \, d\mu, \quad \text{for $\mu$-a.e. $x\in \Omega$.}
\end{equation}
Furthermore, as the map $\mathcal M(\sigma)\ni \mu \mapsto \inf_{n\geqslant 1} \frac1n \int \psi_n \, d\mu$ is the infimum of continuous maps, then it is upper-semicontinuous.
Finally, the second and third claims in the lemma are direct consequences of the corresponding statements for the sub-additive sequence $(\psi_n)_{n\in \mathbb N}$ and the fact that $\lim_{n\to\infty} \frac1n \|\varphi_{n}- \psi_n\|_\infty =0.$
\end{proof}

\subsection{Proof of Theorem~\ref{Mpressure} }\label{SubP2}

Consider a Lipschitz continuous potential  $A:\Omega \to \mathbb{R}  $ and  $0<q<1$.
We will first show  that for any given $x_0 \in \Omega$, the limit
$$ \lim_{n \to \infty} \frac{1}{n} \log  \mathfrak{L}_n (1) (x_0)$$
exists and it is independent of $x_0$.

\begin{lemma}\label{limitex}
For each $x_0$ in $\Omega$, the sequence
$
(\frac1n \log \mathfrak{L}_n (1) (x_0))_{n\geqslant 1}
$ is convergent.
\end{lemma}

\begin{proof}
Fix $x_0$ in $\Omega$ and consider the sequence
\begin{equation} \label{poou}a_n = \log  \mathfrak{L}_n (1) (x_0)=  \log   \sum_{\sigma^n(y)=x_0} (1 + (1-q) \,S_n A(y))^{\frac{1}{1-q} }.
\end{equation}
Given $x\in \Omega$ and $n\in \mathbb{N}$, we denote by
$y_{j,n}^x$, $1\le j \le d^n$, the collection of 
points $y\in \Omega$ satisfying $\sigma^n (y)=x$.
 It is not hard to check that there exists a uniform constant $H>0$ so that 
\begin{equation} \label{kle78} \frac{\mathfrak{L}_m (1)(y_{j,n}^{x_0}) }{ \mathfrak{L}_m (1)(x_0)} \leqslant  H
 \end{equation}
 for every $x_0$,  $m\ge 1$ and $1\le j \le d^n$.
Using \eqref{kle7} and \eqref{kle78} we deduce that 
\begin{align*}
\mathfrak{L}_{m+n} (1)(x_0)
& = \sum_{j=1}^{d^{m+n}}\Big( 1 + (1-q) [S_m A(y_{j,m+n}^{x_0}) +S_n A(\sigma^{m}y_{j,m+n}^{x_0})]\Big)^{ \frac{1}{1-q}}\\
& =  \sum_{k=1}^{d^{m}}     \sum_{j=1}^{d^n}    \Big(\; 1 + (1-q)    [S_m A(y_{k,m-n}^{y_{j,n}^{x_0}})    +  S_n A(y_{j,n}^{x_0})] \; \Big)^{ \frac{1}{1-q}} \\
& \leqslant
 \sum_{k=1}^{d^{m}}    \sum_{j=1}^{d^n}  [   (1 + (1-q) (S_m A(y_{k,m-n}^{y_{j,n}^{x_0}}) )]^{ \frac{1}{1-q}}[  1+ (1-q) (S_n A(y_{j,n}^{x_0})) ]^{ \frac{1}{1-q}} \\
 & =  \mathfrak{L}_m (1) (y_{j,n}^{x_0}) \; \mathfrak{L}_n (1) (x_0) 
  \leqslant H \, \mathfrak{L}_m (1) (x_0)\,   \mathfrak{L}_n (1) (x_0)
  \end{align*}
 for every $m,n\ge 1$.
Then, the sequence $(a_n)_{n\in \mathbb N}$ given by \eqref{poou} satisfies the weakly sub-additive condition
$a_{m+n} \leqslant   a_m  + a_n + \log H$
for every $m, n\geqslant 1$ and, from \cite[Theorem 1.9.2 ]{Steele}, we get that
\begin{equation} \label{kly} \lim_{n \to \infty}\frac{ a_n}{n} = \inf_{n\ge 1} \frac{a_n}{n}.
\end{equation}
This proves that the sequence
$
(\frac1n \log \mathfrak{L}_n (1) (x_0))_{n\geqslant 1}
$ is convergent, as desired.
\end{proof}

The next lemma ensures that the previous limit does not depend on the initial point.

\begin{lemma} \label{kk23} For every $x_0,x_1\in \Omega$ the following holds:
$$
\lim_{n \to \infty} \frac{1}{n} \log  \mathfrak{L}_n (1) (x_0) = \lim_{n \to \infty} \frac{1}{n} \log  \mathfrak{L}_n (1) (x_1)
$$
\end{lemma}

\begin{proof}
It is well known that the potential $A$, being Lipschitz continuous, satisfies the following bounded distortion property: there exists $C>0$ such that
 \begin{equation} \label{quet} |S_n A(y) - S_n A (y^\prime)| \leqslant C \mbox{dist}(x,x^\prime).
 \end{equation}
for any $n\geqslant 1$, any points $x, x^\prime\in \Omega$ and any paired pre-images $y=y_{j,n} \in \sigma^{-n}(x)$ and $y^\prime= y^\prime_{j,n}
\in \sigma^{-n}(x')$ ($1\le j \le d^n$) 
in the same inverse branch for $\sigma^n$.
Moreover, as $\Omega$ has finite diameter, from \eqref{quet} there  exists $H>0$ such that
\begin{equation}\label{slavg1up}
\frac{( 1 + (1-q)S_n A(y))^{\frac1{1-q}} }{   (1 + (1-q)S_nA (y^\prime))^{\frac1{1-q}}}
	= \Big(\frac{ 1 + (1-q)S_n A(y)}{   1 + (1-q)S_n A(y^\prime)} \Big)^{\frac1{1-q}}<H
\end{equation}
for any $n\geqslant 1$ and any paired pre-images $y,y^\prime$.
Using that
\begin{align*} \label{quiet}
|\, (1 + (1-q)r)^{ \frac{1}{1-q}} - (1+ (1-q)s)^{ \frac{1}{1-q}} \,|
	& = (1 + (1-q)r)^{ \frac{1}{1-q}} \; \Big|\, 1 - \Big(\frac{1+ (1-q)s}{1+ (1-q)r}\Big)^{ \frac{1}{1-q}} \,\Big|
\end{align*}
for any $r,s>0$ and denoting by $y_i,y_i'$ the paired n$^{th}$ pre-images of $x, x'$, respectively, one obtains that:
given $n\geqslant 1$,  $x, x^\prime\in \Omega$ and the corresponding  paired pre-images $y,y^\prime$,
\begin{align*}
 |\mathfrak{L}_n (1) (x)  - \mathfrak{L}_n (1) (x^\prime) |
 	& = \Big| \sum_{\sigma^n(y)=x} ( 1 + (1-q) S_n A(y) )^{ \frac{1}{1-q}} - \sum_{\sigma^n(y')=x'} ( 1 + (1-q) S_n A(y') )^{ \frac{1}{1-q}} \Big|\\
	& \leqslant \Big| \sum_{i=1}^{d^n}   ( 1 + (1-q) S_n A(y_i) )^{ \frac{1}{1-q}}  \, |\, 1 - \Big(\frac{1+ (1-q)S_n A(y_i')}{1+ (1-q)S_n A(y_i)}\Big)^{ \frac{1}{1-q}} \,|   \Big|\\
	& \leqslant (1+H) \, \mathfrak{L}_n (1) (x).
\end{align*}
As $x,x'$ are arbitrary we conclude that
$$
\frac1{2+H} \leqslant \frac{\mathfrak{L}_n (1) (x^\prime) }{\mathfrak{L}_n (1) (x) } \leqslant 2+H
$$
for every $n\geqslant 1$ and every $x,x'\in \Omega$. The conclusion of the lemma follows from the last inequalities and the convergence
established in Lemma~\ref{limitex}.
\end{proof}

At this point we proved that for any $0<q<1$, the $q$-asymptotic pressure of  the Lipschitz continuous potential $A:\Omega \to \mathbb{R}$ can be
computed by the limit
\begin{equation*} \label{kjeu}\mathfrak{P}^q (A)= \lim_{n \to \infty} \frac{1}{n} \log \, \mathfrak{L}_n (1) (x_0) \end{equation*}
for an arbitrary point $x_0\in \Omega.$
In order to complete the proof of Theorem~\ref{Mpressure} it remains to prove that the $q$-asymptotic pressure satisfies the following variational principle and that the supremum can be attained.

\begin{lemma}
 \begin{equation} \label{eef}
\mathfrak{P}^q (A)\,=\,\sup_{\nu \in \mathcal{M}_{\text{inv}}(\sigma)}\,\Big\{ h(\nu)+ \lim_{n \to \infty} \frac{1}{n} \int \varphi_n\, d \nu\Big\},
 \end{equation}
 where $ h(\nu)$ is the Kolmogorov-Shannon entropy of $\nu$.
\end{lemma}

\begin{proof}
Fix $x_0\in \Omega$. For each value $n\in \mathbb{N}$ consider the partition $\mathcal{C}_n$ of
$\Omega$ formed by the collection of $2^n$ cylinders of size $n$ in $\Omega$.
We index the elements of the partition $\mathcal C_n$ by corresponding sets $I_{j}^n $, $1\le j \le d^n$.
For each value $n$, and $1\le j \le d^n$, we get that $\sigma^n (I_{j}^n )=\Omega$ and each $ I_{j}^n$ is
an injectivity domain for $\sigma^n$ and, for that reason, we may denote by
$y_{j,n}^x \in I_{n}^n$ the unique n$^{th}$ preimage of the point $x$ in $ I_{j}^n$.

Now, for each $1\le j \le d^n$, pick   $z_{j,n}\in I_{j}^n $ maximizing
the function $I_{j}^n \ni y \to \varphi_n(y)=\frac{1}{1-q} \, \log  (\, 1 + (1-q) \,S_n A(y)\,)$.
From \eqref{slavg1up} we get that
 \begin{equation}\label{slavg1112} \frac{(\, 1 + (1-q) \,S_n A(z_{j,n})\,)^{\frac{1}{1-q} } }{   (\, 1 + (1-q) \,S_n A(y_{j,n}^x)\,)^{\frac{1}{1-q} } }<H
\end{equation}
for any $n$, $x\in \Omega$ and $1\le j \le d^n$. An argument similar to the one used in the proof of Lemma~\ref{kk23} implies
that
$$
\frac1{2+H} \leqslant \frac{\mathfrak{L}_n (1) (x_0) }{\sum_{j=1}^{d^n} e^{\sup_{y\in I^n_j} \varphi_n(y)} } \leqslant 2+H
$$
for every $n\geqslant 1$. By Lemma~\ref{le:phing} there exists a sub-additive sequence
$\psi=(\psi_n)_{n\geqslant 1}$ so that $
\lim_{n\to\infty} \frac1n \|\varphi_{n}- \psi_n\|_\infty =0$.  Altogether we deduce that
 $$
 \mathfrak{P}^q (A) = \lim_{n\to\infty} \frac1n \log \mathfrak{L}_n (1) (x_0)
 		= \lim_{n\to\infty} \frac1n \log \sum_{j=1}^{d^n} \exp( \sup_{y\in I^n_j} \psi_n(y))
 $$
 coincides with the topological pressure of the sub-additive sequence of continuous potentials $(\psi_n)_{n\geqslant 1}$
 (see e.g.  (9) in \cite{Falconer}). 
Combining the variational principle for sub-additive sequences of potentials (see e.g. \cite[equation (14)]{Falconer}) and the second assertion in Lemma~\ref{le:phing}
 we obtain ~\eqref{eef}, as desired.
 \end{proof}

Finally, we note that the fact that 
\begin{equation}
    \label{eq:pfrak}
\mathfrak{P}^q (A)\,=\,\max_{\nu \in \mathcal{M}_{\text{inv}}(\sigma)}\,\Big\{ h(\nu)+ \lim_{n \to \infty} \frac{1}{n} \int \varphi_n\, d \nu\Big\},
\end{equation}
(hence $q$-asymptotic equilibrium states always exist) is a direct consequence of the upper-semicontinuity of the Kolmogorov-Shannon
entropy map $\mathcal{M}_{\text{inv}}(\sigma)\ni \nu \mapsto h(\nu)$  (see \cite{Walters}) and the
the upper-semicontinuity of the map
$
\mathcal{M}_{\text{inv}}(\sigma)\ni \nu \mapsto \lim_{n \to \infty} \frac{1}{n} \int \varphi_n\, d \nu
$
(recall Lemma~\ref{le:phing}). This finishes the proof of Theorem~\ref{Mpressure}. \hfill $\square$

\color{black}

 \subsection{On the space of $q$-asymptotic equilibrium states}

The 
family $\Phi=(\varphi_{n})_{n \in \mathbb{N}}$ 
of  Lipschitz continuous potentials is merely sub-additive in general. Nevertheless, as it satisfies the variational principle ~\eqref{eq:pfrak} it makes sense to ask whether the $q$-asymptotic equilibrium measures can be derived from the classical extensive thermodynamic formalism. We start by the following:

\begin{lemma}\label{cor:BCMV2aa}
The function $C^0(\Omega) \ni A \mapsto \mathfrak{P}^q (A)$ is a pressure function. 
\end{lemma}

\begin{proof} 
The monotonicity assumption (H1) is immediate from the definition of $\mathfrak{P}^q (A)$. The translation invariance and the convexity are immediate consequences of  ~\eqref{eq:pfrak}.
\end{proof}

In view of \cite[Lemma~8.3]{BCMV}, we define the following 
measurable, upper-semicontinuous and bounded potential
$\psi_A:\Omega \to\mathbb R$ defined by 
$$
\psi_A(x)=\inf_{n\geqslant 1} \frac1n \varphi_n(x)
= \inf_{n\geqslant 1} \Big[ \frac{1}{(1-q)n} \,\cdot\, \log  (\, 1 + (1-q) \,S_n A(x)\,) \Big]
$$
for every $x\in \Omega$.
We prove the following variational principle which involves the usual averages of the potential $\psi_A$. 

\begin{proposition}\label{cor:BCMV2}
There exists an upper-semicontinuous function $\mathfrak{h}: \mathcal M_{a}(\sigma) \to \mathbb R$ such that
$$
\mathfrak{P}^q (A)\,=\,\sup_{\nu \in \mathcal{M}_{a}(\sigma)}\,\Big\{ \mathfrak{h}(\nu)+ \int \psi_A\, d \nu\Big\},
$$
where $\mathcal M_a(\Omega)$ stands for the space of $\sigma$-invariant finitely additive measures.
In particular, there exists a $\sigma$-invariant finitely additive equilibrium state $\nu \in \mathcal{M}_{a}(\sigma)$ for $\sigma$ with respect to the
potential $\psi_A$.
\end{proposition}

\begin{proof}
For each $0<q<1$, the sub-additive family of  Lipschitz continuous potentials $\Phi=(\varphi_{n})_{n \in \mathbb{N}}$ satisfies
\begin{align*}
\inf_{n\geqslant 1} \frac{1}n \varphi_n(x)
	\geqslant  \frac{1}{1-q} \,\cdot\, [1-2(1-q) \inf_{x\in \Omega} A(x)].
\end{align*}
Therefore, $\mathfrak{P}^q (A)>-\infty$ for every bounded potential $A$. Then, it follows from the proof of \cite[Theorem~8.4]{BCMV}
that there exists an upper-semicontinuous entropy function $\mathfrak{h}: \mathcal M_{a}(\sigma) \to \mathbb R$, defined by
$
\mathfrak{h}(\mu)=\inf_{\psi\in L^\infty(\Omega)}  \Big( \mathfrak{P}^q (A) -\int \psi_A \, d\mu\Big)
$
and satisfying
 \begin{equation*} \label{eefa}
\mathfrak{P}^q (A)\,=\,\sup_{\nu \in \mathcal{M}_{a}(\sigma)}\,\Big\{ \mathfrak{h}(\nu)+ \int \psi_A\, d \nu \Big\}.
 \end{equation*}
 Finally, the second claim follows as a direct consequence of the first one together with the upper semicontinuity of the functions $\nu\mapsto \mathfrak h(\nu)$ and 
 $\nu\mapsto \int \psi_A \, d\nu$. This completes the proof of the proposition.
\end{proof}

\section{Solution of Bowen-type equations for the non-extensive pressure and non-extensive transfer operators} \label{apen1}

In this section we will study the space of normalizable potentials, which are related to the
existence of eigenfunctions.

\subsection{Existence of eigenfunctions for extensive transfer operators} \label{apen0}

We first prove the following warm-up theorem within the classical extensive framework (this corresponds to the special case $q=1$ in Theorem~\ref{Mddeste}).

\begin{theorem} \label{Mddeste0}  Let $\tilde{A}:\Omega \to \mathbb{R}$ be a normalized
Lipschitz continuous potential and $\nu_{\tilde{A}}$ be such that $\mathcal{L}_{\tilde A}^* \nu_{\tilde A}=\nu_{\tilde A}$.
There exists an
open neighborhood $\mathcal U\subset \mbox{Lip}(\Omega)$ of $\tilde A$ and a differentiable map
$$
\mathcal U \ni A \mapsto (\varphi_A,c_A) \in \mbox{Lip}(\Omega) \times \mathbb R
$$
such that:
\begin{itemize}
    \item[(a)] $ \int \varphi_A d \nu_{\tilde{A}}=0$, and 
    \item[(b)]  $\sum_{a=1}^d e^{A(a x) + \varphi_A (a x) - \varphi_A(x) - c_ A}=1$, for every $x\in \Omega$.
\end{itemize}

 Moreover, $\lim_{A \to \tilde A} (\varphi_A,c_A) =(0,0)\in \mbox{Lip}(\Omega)\times \mathbb R$.
 \end{theorem}

\subsubsection{Differentiability of transfer operators}

Let $\tilde{A}:\Omega \to \mathbb{R}$ be a fixed normalized
Lipschitz continuous potential and $\nu_{\tilde{A}}$ be a probability measure such that $\mathcal{L}_{\tilde A}^* \nu_{\tilde A}=\nu_{\tilde A}$.
Consider the analytic map
\begin{equation} \label{try}
\begin{array}{rccc}
F: & \mbox{Lip}(\Omega) \times \mbox{Lip}(\Omega)\times  \mathbb{R} & \to & \mbox{Lip}(\Omega) \times\mathbb{R} \\
  & (A,\varphi,c) & \mapsto & (\sum_{a=1}^d \, e^{A(a\cdot) + \varphi(a\cdot)  - \varphi(\cdot) - c}, \int  \varphi \,d \nu_{ {A}}).
  \end{array}
\end{equation}
and write $F_A(\cdot,\cdot)=F(A, \cdot, \cdot)$ for notational simplicity. Given $A\in \mbox{Lip}(\Omega)$, $\varphi\in \mbox{Lip}(\Omega)$,
$c\in \mathbb R$ and
$(H,h)\in \mbox{Lip}(\Omega)\times  \mathbb{R}$ one can use the Taylor expansion of the exponential map to write
\begin{align*}
F_A & (\varphi+H,c+h)   - F_A(\varphi,c)  \\
	 & =
	\Big(\sum_{a=1}^d \, [e^{A(a\cdot) + (\varphi+H)(a\cdot)  - (\varphi+H)(\cdot) - (c+h)} - e^{A(a\cdot) + \varphi(a\cdot)  - \varphi(\cdot) - c}], \int  H \,d \nu_{ {A}}\Big) \\
 & =
	\Big(\sum_{a=1}^d \,  e^{A(a\cdot) + \varphi(a\cdot)  - \varphi(\cdot) - c} \, \big(   e^{H(a\cdot)  - H(\cdot) - h} -1 \big ), \int  H \,d \nu_{ {A}}\Big) \\
 & =
	\Big(\sum_{a=1}^d \,  e^{A(a\cdot) + \varphi(a\cdot)  - \varphi(\cdot) - c} \, \big( H(a\cdot)  - H(\cdot) - h \big ), \int  H \,d \nu_{ {A}}\Big) + \mathcal O(\|H\|^2) + \mathcal O(h^2),
\end{align*}

(as usual the  terminology $\mathcal O(y)$ means that there exists $B>0$ so that the expression is bounded by $B |y|$).
The first term above is linear in $(H,h)$ while the second and third terms correspond to higher order terms. Hence, we conclude that the derivative 
$DF_A(\varphi,c): \mbox{Lip}(\Omega) \times\mathbb{R}\to \mbox{Lip}(\Omega) \times\mathbb{R}$ is given by
\begin{align}\label{eq:derF}
DF_A(\varphi,c)(H,h)=\Big(\sum_{a=1}^d \,  e^{A(a\cdot) + \varphi(a\cdot)  - \varphi(\cdot) - c} \, \big( H(a\cdot)  - H(\cdot) - h \big ), \int  H \,d \nu_{ {A}}\Big).
\end{align}

\subsubsection{Proof of Theorem~\ref{Mddeste0}}

Let $\tilde{A}:\Omega \to \mathbb{R}$ be a fixed normalized
Lipschitz continuous potential and $\nu_{\tilde{A}}$ be a probability measure such that $\mathcal{L}_{\tilde A}^* \nu_{\tilde A}=\nu_{\tilde A}$.
By the assumption, the equation
\begin{equation} \label{atch43b} \sum_{a=1}^d e^{\tilde{A}(a x) + \tilde{\varphi}(a x) - \tilde{\varphi}(x) - \tilde{c}}=1,
\end{equation}
has a solution for the function $\tilde \varphi\equiv 0$ and the constant $\tilde c=0$ (in this case $\nu_{\tilde{A}}$ is $\sigma$-invariant). Equivalently
\begin{equation} \label{atch43c}  F_{\tilde A} (0,0)=(1,0).
\end{equation}

Our purpose is to use the formulation in ~\eqref{atch43c} in order to use the implicit function theorem to show that there exists an open neighborhood of the potential $\tilde A$ formed by normalizable potentials. 

\begin{proposition}
    \label{le:derat0}
$DF_{\tilde A}(0,0): \mbox{Lip}(\Omega) \times\mathbb{R}\to \mbox{Lip}(\Omega) \times\mathbb{R}$ is an isomorphism.
\end{proposition}

\begin{proof}
    Using that $1$ is a simple leading eigenvalue and $\mathcal L_{\tilde A} 1 =1$, one can write
\begin{align*}
\mbox{Ker}(DF_{\tilde A}(0,0))
	& =\Big\{(H,h)\in  \mbox{Lip}(\Omega) \times\mathbb{R}\colon
		\sum_{a=1}^d \,  e^{\tilde A(ax)} \, ( H(ax)  - H(x)) =  h , \int  H \,d \nu_{ \tilde{A}}=0 \Big\}.
\end{align*}

In particular, if $(H,h)\in \mbox{Ker}(DF_{\tilde A}(0,0))$ then 
$$
H=\mathcal L_{\tilde A} H  - h  \qquad\text{and}\qquad
\int H \, d\nu_{\tilde A} = \int \mathcal L_{\tilde A} H \,d\nu_{\tilde A} - h.
$$
As $\mathcal L_{\tilde A}^*\nu_{\tilde A}=\nu_{\tilde A}$ the second equation implies that $h=0$. Consequently, all elements in the kernel of $DF_{\tilde A}(0,0)$
are of the form $(H,0)$ where 
$\mathcal L_{\tilde A} H=H$ (hence $H$ is constant) and $\int H\, d\nu_{\tilde A}=0$.
This implies that $DF_{\tilde A}(0,0)$ is injective.

\smallskip
We proceed to show that $DF_{\tilde A}(0,0)$ is surjective. Fix an arbitrary $(\psi, b) \in \mbox{Lip}(\Omega) \times\mathbb{R}$.
By definition, the equation $DF_{\tilde A}(0,0)(H,h)=(\psi, b)$ admits a solution $(H,h)$ if and only if
\begin{equation} \label{atch43bfe}
\begin{cases}
\mathcal L_{\tilde A} (H- H\circ \sigma)   = \psi + h \smallskip \\
\int  H \,d \nu_{ \tilde{A}}  = b
\end{cases}
\end{equation}
Integrating with respect to $\nu_{\tilde A}$ and noticing that
 we have it is $\sigma$-invariant (because $\tilde A$ is normalized)
the first equation implies that  $h=-\int \psi\, d\nu_{\tilde A}$. 

We will now make use of the spectral theory for the transfer operator $\mathcal L_{\tilde A}$.  Denote by
\begin{equation}
    \label{eq:cobod}
    \mathcal{C}=\{ \phi - \phi \circ \sigma +  z \colon  \phi\in \mbox{Lip}(\Omega),\,\,z\, \in \mathbb R\}
\end{equation}
the subspace formed by coboundaries and set $\mathcal{C}_0=\{ \phi - \phi \circ \sigma \colon \phi\in \mbox{Lip}(\Omega)\}$.
It is known that the linear map $\mathcal L_{\tilde A}\mid_{\mathcal C}: \mathcal{C} \to \mbox{Lip}(\Omega)$ is onto
(cf. \cite[Proposition 3.3 item 4]{GKLM}).

We claim that the linear operator
$\mathfrak{V}: \mathcal C_0 \to \{\phi \in \mbox{Lip}(\Omega) \colon \int \phi\, d\nu_{\tilde A}=0\}$ given by
$$
\mathfrak{V}(H)=\mathcal L_{\tilde A} (H- H\circ \sigma)
$$
 is a bijection. As the injectivity follows as in the proof of the injectivity of $DF_{\tilde A}(0,0)$, it remains to prove the  surjectivity of $\mathfrak{V}$. In fact, by surjectivity of  ${\mathcal L}_{\tilde A}\mid_{\mathcal C}$, for any $\phi \in \mbox{Lip}(\Omega)$ so that $\int \phi\, d\nu_{\tilde A}=0$, there exists
$H\in \mbox{Lip}(\Omega)$ and $z\, \in \mathbb R$ so that
$$
\phi={\mathcal L}_{\tilde A}(H- H\circ \sigma +z)={\mathcal L}_{\tilde A}(H- H\circ \sigma) + z.
$$
Integrating with respect to the $\sigma$-invariant probability measure $\tilde \nu_A$ one obtains
$$
0=\int \phi\, d\nu_{\tilde A} = \int (H- H\circ \sigma)\, d{\mathcal L}_{\tilde A}^*\tilde \nu_{\tilde A} + z,
$$
hence $z=0$.
This shows that $\mathfrak{V}(H- H\circ \sigma)=\phi$, and proves the surjectivity of $\mathfrak{V}$. Therefore, the solution of
~\eqref{atch43bfe}
is obtained as the unique solution of the cohomological equation
\begin{equation*} \label{atch43bfeh}
H- H\circ \sigma  = \mathfrak{V}^{-1}(\psi + h)
\end{equation*}
which satisfies $\int  H \,d \nu_{ \tilde{A}}  = b$.
This finishes the proof of the proposition.
\end{proof}

We are now in a position to finish the proof of Theorem~\ref{Mddeste0}.
In fact, in view of Proposition~\ref{le:derat0}, the implicit function theorem (see e.g. \cite[page 9]{Lang} or \cite{An,Ros})
ensures that there exists an
open neighborhood $\mathcal U\subset \mbox{Lip}(\Omega)$ of $\tilde A$ and a differentiable map
$$
\mathcal U \ni A \mapsto (\varphi_A,c_A) \in \mbox{Lip}(\Omega) \times \mathbb R
$$
such that
\begin{equation*}
\sum_{a=1}^d e^{A(a x) + \varphi_A (a x) - \varphi_A(x) - c_ A}=1\qquad \text{for every $x\in \Omega$}
\end{equation*}
 and $ \int \varphi_A d \nu_{\tilde{A}}=0$. This is to say that $\tilde A$ admits an open neighborhood formed by normalizable potentials.
Finally, the limit $\lim_{A \to 0} (\varphi_A,c_A) =(0,0)\in \mbox{Lip}(\Omega)\times \mathbb R$ follows as an immediate consequence of the
differentiability result. This finishes the proof of the theorem. \hfill $\square$

\begin{remark}
The previous arguments make use of some known results for the classical transfer operators. Nevertheless, the arguments do not make use of iterates
of the transfer operators, which is one of the main obstructions to develop a non-extensive thermodynamic formalism.
For instance, in our proof we used the claim of Theorem 3.3 item (4) in \cite{GKLM}, which was obtained from properties of the operator
$(I- \mathcal{L}_A)^{-1}$, where $A$ is a H\"older normalized potential. In a general case, if this is true, the argument works. We believe that following this reasoning it is possible to derive pathological examples  where one gets $C^2$ differentiability but not analyticity.
Finally, one should notice that the arguments concerning the proof of the surjectivity resemble similar constructions appearing for perturbation
theory of leading simple eigenvalues
(cf. \cite{KLOE,Ros}).
\end{remark}

\subsection{The space of extensive normalizable potentials}

In this subsection we proceed with the proof of Theorem~\ref{Mddeste} in the context of the non-extensive thermodynamic formalism.
To each normalizable Lipschitz continuous potential $\tilde A$ one can associate the normalized potential
\begin{equation}\label{eq-nomrlB}
B(\tilde A)= \tilde{A} + \varphi_{\tilde{A}} - \varphi_{\tilde{A}}\circ \sigma - c_{\tilde{A}},
\end{equation}
 where $P(\tilde A)=c_{\tilde A}$ and  $e^{c_{\tilde A}}$ is the simple leading eigenvalue of the transfer operator $\mathcal L_{\tilde A}$

and $h_{\tilde A} =e^{\varphi_{\tilde{A}}}$ is a leading eigenfunction which, up to a multiplicative constant,  we may assume to satisfy $\int h_{\tilde A}\, d\nu_{\tilde A}=1$.
Moreover, one has that
both maps
$$
h : \tilde A \to h_{\tilde A}\in C^0(\Omega,\mathbb R)
\quad\text{and}\quad
c: \tilde A \to c_{\tilde A}\in \mathbb R
$$
defined for potentials in $\mbox{Lip}(\Omega)$
are $C^1$-smooth, and that, denoting by $E^0$ the space of mean zero observables with respect to 
$\nu_{\tilde A}$,
\begin{equation}\label{eq:derhA}
D h(\tilde A) H
= h_{\tilde A} \cdot \int \big[(I - \mathcal{L}_{\tilde A \;|_{E_{0}}})^{-1}(1 - h_{\tilde A}) \big] \cdot H  \; d\nu_{\tilde A},
\end{equation}
and
\begin{equation}\label{eq:dercA}
D c(\tilde A) H
= c_{\tilde A} \cdot \int  h_{\tilde A} \cdot H  \; d\nu_{\tilde A} = c_{\tilde A} \cdot \int  H  \; d\mu_{\tilde A}.
\end{equation}
for every $H\in \mbox{Lip}(\Omega)$ (cf.  \cite[Arxiv version]{BCV} or \cite{LR1}). We now deduce that the normalized potential varies smoothly with the original Lipschitz continuous potential.

\begin{corollary}
The map $B: \mbox{Lip}(\Omega) \to \mbox{Lip}(\Omega)$ given by \eqref{eq-nomrlB} is $C^1$-smooth
and 
\begin{align*}
DB(\tilde A) H
	&= H + \int \big[(I - \mathcal{L}_{\tilde A \;|_{E_{0}}})^{-1}(1 - h_{\tilde A}) \big] \cdot (H-H\circ\sigma)  \; d\nu_{\tilde A}
	 - c_{\tilde A} \cdot \int  h_{\tilde A} \cdot H  \; d\nu_{\tilde A}
\end{align*}
for every $H\in \mbox{Lip}(\Omega)$.
\end{corollary}

\begin{proof}
It is immediate from ~\eqref{eq:derhA} and the chain rule that, for each $H\in \mbox{Lip}(\Omega)$,
\begin{equation}\label{eq:derhphiA}
D \varphi(\tilde A) H
	=D (\log \circ h)(\tilde A) H
	= \frac1{h(\tilde A)} D h(\tilde A) H
= \int \big[(I - \mathcal{L}_{\tilde A \;|_{E_{0}}})^{-1}(1 - h_{\tilde A}) \big] \cdot H  \; d\nu_{\tilde A}
\end{equation}
This, together with ~\eqref{eq:dercA}, implies that $B$ is $C^1$-smooth and yields the formula for the derivative of $B$.
\end{proof}

\medskip

 The next theorem offers a generalization of Theorem~\ref{Mddeste0}, by establishing a solution for Bowen's equation for general normalizable potentials. 

\begin{theorem} \label{t2} \,Fix $0<q<1$ and assume that $\tilde A$ is normalizable 
There exists an
open neighborhood $\mathcal U\subset \mbox{Lip}(\Omega)$ of $\tilde A$ 
such that for every Lipschitz continuous potential $A \in \mathcal U $ 
there exists a solution $(\varphi_A,\, c_A)$ such that
\begin{equation} \label{atch467} \sum_{a=1}^d e^{A(a x) + \varphi_A (a x) - \varphi_A(x) - c_A}=1, \qquad \text{for all $x\in \Omega$}
\end{equation}
and $ \int \varphi_A d \mu_{\tilde{A}}=0$. 
Moreover, $\mathcal U\ni A\mapsto (\varphi_A,c_A)$ is $C^1$-smooth.
\end{theorem}

\begin{proof}

If $\tilde A$ is normalized the result follows from Theorem~\ref{Mddeste0} and
there is nothing to prove. Hence we assume that $\tilde A$ is not normalized.
Let $\varphi_{\tilde A}$ and $c_{\tilde A}$ be such that the potential
$ 
B=\tilde A+\varphi_{\tilde A}-\varphi_{\tilde A}\circ\sigma-c_{\tilde A}
$ 
is normalized, hence 
$\sum_{a=1}^d e^{B(ax)}=1$ for all $x\in\Omega$.
Consider the analytic map
$ 
F:\mathrm{Lip}(\Omega)\times\mathrm{Lip}(\Omega)\times\mathbb R
\longrightarrow
\mathrm{Lip}(\Omega)\times\mathbb R
$ 
defined by
\[
F(A,\varphi,c)(x)
=
\left(
\sum_{a=1}^d
e^{A(ax)+[\varphi_{\tilde A}(ax)-\varphi_{\tilde A}(x)-c_{\tilde A}]
+\varphi(ax)-\varphi(x)-c},
\;
\int \varphi\, d\mu_{\tilde A}
\right)
\]
and observe that
$F(\tilde A,0,0)=(1,0).$

In order to prove Theorem \ref{t2} we proceed to verify the assumptions of the implicit function theorem.
In order to do so, first we compute the derivative of $F$ with respect to $(\varphi,c)$ at
$(\tilde A,0,0)$. Note that 
\begin{align*}
  \sum_{a=1}^d & e^{\tilde{A}(a x) +[  \varphi_{\tilde{A}} (a x) -  \varphi_{\tilde{A}} (x) -c_{\tilde{A}}] + v(a x) - v(x) - (c+\alpha)} -   \sum_{a=1}^d e^{\tilde{A}(a x)  +[  \varphi_{\tilde{A}} (a x) -  \varphi_{\tilde{A}} (x) -c_{\tilde{A}}] - c} \\
    & =\sum_{a=1}^d e^{(\tilde{A}(a x)+[  \varphi_{\tilde{A}} (a x) -  \varphi_{\tilde{A}} (x) -c_{\tilde{A}}]  - c)}  (e^{  v(a x) - v(x) -\alpha}-1  ) \\
    &=  \,\sum_{a=1}^d e^{B(a x)-c}  (e^{  v(a x) - v(x) -\alpha}-1)
\end{align*}
for every $v\in \text{Lip}(\Omega)$ and $\alpha\in \mathbb R$. In particular, taking the Taylor series for the exponential map one concludes that, for $(v,\alpha)\in\mathrm{Lip}(\Omega)\times\mathbb R$,
\[
\begin{aligned}
    D_{(\varphi,c)}F(\tilde A,0,0)(v,\alpha)(x)
& =
\left(
\sum_{a=1}^d
e^{B(ax)}\big(v(ax)-v(x)-\alpha\big),
\;
\int v\, d\mu_{\tilde A}
\right)\\
& = \left(\mathcal L_B v(x)-v(x)-\alpha, \int v\, d\mu_{\tilde A} \right),
\end{aligned}
\]
where $\mathcal L_B$ stands for the Ruelle operator associated with the normalized
potential $B$.
We claim that the latter is an isomorphism from
$\mathrm{Lip}(\Omega)\times\mathbb R$ onto
$\mathrm{Lip}(\Omega)\times\mathbb R$.
In fact, given $(f,\beta)\in\mathrm{Lip}(\Omega)\times\mathbb R$ we aim to prove that has a unique solution $(v,\alpha)$.
\begin{equation}
    \label{eq:auxili}
\mathcal L_B v-v-\alpha=f,
\qquad
\int v\,d\mu_{\tilde A}=\beta.
\end{equation}
Using that $\mathcal L_B^*\mu_{\tilde A}=\mu_{\tilde A}$ and integrating the first equation above with respect to $\mu_{\tilde A}$ one deduces that 
$\alpha=-\int f\, d\mu_{\tilde A}.$ Thus, the first equation in ~\eqref{eq:auxili} becomes
\[
(I-\mathcal L_B)v
=
-\Big(f-\int f\, d\mu_{\tilde A}\Big)
\]
is a $\mu_{\tilde A}$-mean zero observable.  
Since $B$ is normalized, the operator
$I-\mathcal L_B$ is invertible on the subspace of Lipschitz functions with 
$\mu_{\tilde A}$-mean zero, and so  there exists a unique solution $u\in \text{Lip}(\Omega)$ of the previous equation with
$\int u\, d\mu_{\tilde A}=0$.
Then, the pair $(v,\alpha)$ with 
$v=u+\beta$ and $\alpha=-\int f\, d\mu_{\tilde A}$ is the only solution of ~\eqref{eq:auxili}.
This proves that
$D_{(\varphi,c)}F(\tilde A,0,0)$ is bijective.

Hence, by the implicit function theorem,  there exist an open neighborhood $\mathcal U$ of $\tilde A$ and a
$C^1$ map
$A\in\mathcal U\mapsto (\varphi(A),c(A))$
such that
$F(A,\varphi(A),c(A))=(1,0)$.
Finally, defining
\[
\varphi_A=\varphi(A)+\varphi_{\tilde A},
\quad \text{and}\quad
c_A=c(A)+c_{\tilde A},
\]
we obtain a mean zero solution for ~\eqref{atch467} as desired. This concludes the proof.
\end{proof}

Denote by $\mathfrak{N}$ the set of normalizable Lipschitz continuous potentials.
Building over the previous results one can prove the following:

\begin{theorem} 
Every Lipschitz continuous potential $A: \Omega\to\mathbb R$ is normalizable.
\end{theorem}

\begin{proof}   The space 
$\mathfrak{N}$ is an open subset of $\text{Lip}(\Omega)$ as a consequence of Theorem~\ref{t2}. In order to prove the theorem it is enough to prove that $\mathfrak{N}$ is closed. 

Let $(A_n)_{n\ge 1}$ be a sequence in $\mathfrak{N}$ converging to $A\in \text{Lip}(\Omega)$ in the Lipschitz norm (recall ~\eqref{eq:Lipnorm}). 
We will show that $A\in \mathfrak{N}$, that is,
 that $\mathcal{L}_A$ has a leading eigenvalue and eigenfunction.
Denote by $W_n: \Omega \times \Omega  \to \mathbb{R}$, $n \ge 1$, an involution kernel
for $A_n$, with the normalization $W_n ( ...1,1,1|1,1,1...)=1$. 
(we refer the reader to \cite{BLT,LOS,LOT,LM1} for the existence of  involution kernels).
In particular, the dual potential $A_n^*:\Omega \to \mathbb R$, depending on negative coordinates of the two-sided shift, is defined by the equation
\begin{equation}
    \label{eq:invK}
    A_n^*(y) = A_n(x)+ W_n(\hat\sigma (y\mid x)) - W_n(y\mid x)
\end{equation}
where $(x|y) \in \Omega \times \Omega$ and $\hat \sigma$ stands for the two sided shift on $\Omega \times \Omega$.

For each $n\ge 1$, denote by 
$\nu_{A_n^*}$ the leading eigenmeasure for the potential $A_n^*$.
 It is known 
that the involution kernel   $W_n$ is Lipchitz in both variables  \cite[Proposition 7 in  Section 5]{LOT}, that 
the dual potential depends Lipchitz continuously on the potential
\cite[Section 3]{CLO} 
and that 
$$
 \varphi_n (x) = \int e^{W_n(x,y)} d \nu_{A_n^*}(y)
 $$
 is an eigenfunction for the Ruelle operator
 $\mathcal{L}_{A_n}$.
Up to extract some subsequence, we may assume without loss of generality that  $(\nu_{A_n^*})_{n\ge 1}$ is weak$^*$ convergent to $\nu$. 
Moreover, let $W$ denote the limit of the corresponding 
involution kernels $W_n$,
guaranteed by the Arz\`ela-Ascoli theorem. 
Then, one concludes that 
$$
\varphi = \lim_{n\to\infty} \varphi_n = \lim_{n\to\infty} \int e^{W_n(\cdot,y)} d \nu_n (y)
= 
\int e^{W(\cdot,y)} d \nu (y)
$$
is an eigenfunction for the Ruelle operator $\mathcal{L}_A$ associated to the leading eigenvalue $1$.
This proves that $\mathfrak{N}$ is a closed set and completes the proof of the theorem.
\end{proof}

\subsection{Solution of Bowen-type equations for non-extensive transfer operators} \label{p.thmB}

Fix $0<q<1$. 
Given a Lipschitz continuous potential $A:\Omega \to \mathbb{R}$ recall it is $q$-normalizable (resp. $q$-normalized) if there exist a Lipschitz continuous function $\varphi_A$ and $c_{A}$,
 such that, 
\begin{equation} \label{atch229} \sum_{a=1}^2 e_q^{A(a x) + \varphi_A (a x) - \varphi_A(x) - c_{A}}=1 \; \; 
\Big( \text{resp.} \;\;  \sum_{a=1}^2       (1 + (1 - q) {A}(a x) )^{1/(1 - q)}\,=\, \sum_{a=1}^2 e_q^{{A}(a x) }=1\Big)
\end{equation}
and all summands are strictly  positive for all $x\in \Omega$.
The potential $A\equiv 0$ is $q$-normalizable for all $0<q<1$ as  $\varphi_A=0 $ and   $c_{A}= \log_{2-q} (2) = \frac{-1 + 2^{-1 +q} }{q-1 }$ are solutions for the equation \eqref{atch229}, and so $A \equiv \log_{2-q} (2)$ is $q$-normalized.

\begin{theorem} \label{ddestemu}  Given $0<q<1$ and $\tilde A\in \mathfrak N_q$ there exists an open neighborhood $\mathcal U$ of $\tilde A$ in $\text{Lip}(\Omega)$ such that for  
every $A\in \mathcal U$ there exists 
a probability measure $\nu_{\tilde A}$, a Lipschitz continuous function 
$\varphi_A$ and $c_A \in \mathbb R$ such that 
\begin{equation*}  \sum_{a=1}^d e_q^{A(a x) + \varphi_A (a x) - \varphi_A(x) - c_ A}=1 \qquad \text{for all $x\in \Omega$,}
\end{equation*}
and $ \int \varphi_A d \mu_{\tilde{A}}=0$. Moreover, $\varphi_A,c_A$ depend in a differentiable way on $A$.
\end{theorem}

\begin{proof}
Let us prove the result in the case $A\in \mathfrak N_q$ is normalized  (the proof on the general case can be adapted in a natural way, by defining a similar map as in ~\eqref{atch467} with the exponential terms replaced by $q$-exponential terms).
Consider the analytic map $F:\text{Lip}(\Omega) \times \text{Lip}(\Omega) \times  \mathbb{R}\to \text{Lip}(\Omega)\times\mathbb{R} $, given by
\begin{equation} \label{atc21a}
F(A,\varphi,c) (x)= \Big(\sum_{a=1}^d e_q^{A(a x) + \varphi (a x) - \varphi(x) - c}, \int  (\varphi - \varphi \circ \sigma - c) d \mu_{ \tilde{A}}\Big),
\end{equation}
for a fixed probability measure $\mu_{\tilde{A}}$ to be determined later.
The assumptions ensure that $F(\tilde A, 0, 0)=(1,0)$.
We will obtain functions $A\mapsto\varphi(A),$ $A\mapsto c(A)$ such that $ F(A,\varphi(A),c(A)) =(1,0)$ as a consequence of the implicit function theorem. 

\smallskip
Observe that the second coordinate in the right-hand side of ~\eqref{atc21a} is linear in $(\varphi,c)$. Moreover,  
using the derivative of $q$-exp maps in \eqref{prop5},
for each $(H,h)\in \text{Lip}(\Omega) \times \mathbb R$,
\begin{align*}
\sum_{a=1}^d e_q^{A(a x) + H (a x) - H(x) - h}
-\sum_{a=1}^d e_q^{A(a x)}
& 
    =
     \sum_{a=1}^d \; \Big(e_q^{\tilde{A}(a x) + H(a x) - H(x) - {h}} -    e_q^{\tilde{A}(a x)} \, 
    \Big) \\
    & = 
    \sum_{a=1}^d 
    \big[ 1+ (1-q)\tilde A(ax) \big]^{\frac{q}{1-q}} \; (\, H(ax)-H(x)-h\,) \\
    & + \mathcal O (\|H\|^2 + |h|^2),
\end{align*}
where the first term in the right-hand side above is linear in $(H,h)$.
This proves that
\begin{equation}
\label{eq:derFAq}
DF_A(0,0) (H,h)= \Big( \sum_{a=1}^d 
    \big[ 1+ (1-q)\tilde A(a\cdot) \big]^{\frac{q}{1-q}} \; (\, H(a\cdot)-H(\cdot)-h\,) , \int  (H - H \circ \sigma- h) d \mu_{ \tilde{A}}\Big).
\end{equation}
In view of 
\eqref{eq:derFAq} it is natural to consider the extensive non normalized potential $q\, \tilde A_q$ given by
$$ x \mapsto q\, \tilde A_q (x) = \,\log\,(\,  [1+(1-q)\tilde{A}(x) ]^{\frac{q}{1 - q}}\,),$$
(recall ~\eqref{i1})
and the associated extensive Ruelle operator
$$ \mathcal{L}_{q\tilde{A}_q} (f) (x) = \sum_{a=1}^d [1+(1-q) \,\tilde{A}(a\,x)]^{\frac{q}{1 - q}}\, f(a x). $$
Denote by $\mu_{\tilde{A}}$ the extensive equilibrium state for the potential $q\tilde{A}_q.$ By $\sigma$-invariance we deduce from ~\eqref{eq:derFAq} that 
\begin{equation*}
DF_{\tilde A}(0,0) (H,h)= \Big( \mathcal L_{q\tilde A_q} (H-H\circ \sigma-h)  \;,\;  -h\Big).
\end{equation*}
The same argument as in the proof of  Proposition~\ref{le:derat0} ensures that the linear map $\mathfrak{V}: \mathcal C_0 \to \text{Lip}(\Omega)_0$ given by
$$
\mathfrak{V}(H)=\mathcal L_{q\tilde A_q} (H-H\circ \sigma)
$$
is an isomorphism. It is simple to check that the latter implies that $DF_{\tilde A}(0,0)$ is an isomorphism as well. Hence, the conclusion of the theorem is a direct consequence of the implicit function theorem. 
\end{proof}

\medskip


\section{Examples} \label{exe2}

\subsection{One-step potentials}

We will consider first the  case of locally constant potentials $A$ depending  just on the first coordinate.

\smallskip

\begin{example} \label{supex}  Consider the potential $A:\{1,2\}^\mathbb{N} \to \mathbb{R}$, such that, is constant equal to $a_1$ on the cylinder $\overline{1}$
and equal to $a_2$ on the cylinder $\overline{2}.$ Also consider another potential $B:\{1,2\}^\mathbb{R} \to \mathbb{R}$, such that, is constant equal to $b_1$ on the cylinder $\overline{1}$
and equal to $b_2$ on the cylinder $\overline{2}.$ For $q=3/2$, we want to consider the equation \eqref{atchew1} for the family of potentials $A+s B,$ where $s \in \mathbb{R}.$ More precisely, we will exhibit explicit  solutions, which  will be denoted by $c(s)$ and $\varphi^s: \{1,2\}^\mathbb{R} \to \mathbb{N}$, with a dependence  on the parameter $s$,
for the equation 
\begin{equation} \label{ato1} \sum_{a=1}^2 e_{2-\frac{3}{2}}^{(A+ s B )(ax) + \varphi^s (ax) - \varphi^s(x) - c(s)}=1,
\quad \text{for any $x=(x_1,x_2,...)\in \Omega$.}
\end{equation}
Given $s \in \mathbb{R}$, consider
$c(s)= 1/2 (4 + a_1 + a_2 + b_1 s + b_2 s)$. From Theorem~\ref{mthm1} we get that $c(s) =  P_{3/2}(A+ s\,B ).$
Consider the function $\varphi^s: \{1,2\}^\mathbb{N} \to \mathbb{R}$  such that  for each $s$, it  is constant equal to $\varphi_2(s) $ on the cylinder $\overline{1}$
and equal to $\varphi_1(s)=0$ on the cylinder $\overline{2},$ where
$$\varphi_2(s) = 1/2 (-a_1 + a_2 - b_1 s + b_2 s +$$
$$ \sqrt{16 - a_1^2 + 2 a_1 a_2 - a_2^2 - 2 a_1 b_1 s + 2 a_2 b_1 s + 2 a_1 b_2 s -
    2 a_2 b_2 s - b_1^2 s^2 + 2 b_1 b_2 s^2 - b_2^2 s^2]}).
    $$
We leave to the reader the task to verify that the equation \eqref{ato1}  holds. When $a_1=2$, $a_2=5.5$, $b_1=0=b_2$, we get $\varphi_2(0)= 2.71825$ and $c(0)=5.75 $. The explicit expressions were obtained using the software Mathematica. In this way we get an explicit solution for \eqref{ato1}, when $s=0$.
Note also that
$$
 c(0) = P_{3/2}(A) = \sup_p \Big\{ H_{3/2}(p)+\int A \,d p  \,\Big\}.
$$

 From Proposition \ref{esteaq1}, if the previous $A$ satisfies
 $A=-  \log_q (\frac{1}{ J})$, we get that $P_q(A)=0$ and the $q$-equilibrium state for $A$ is the classical one for   $\log J$. It is well known that  this one  is the independent probability $\mathfrak{p}$ on $\Omega$  with weights $p_1,p_2$, where
\begin{equation}  \label{altra1}p_k=  \frac{ e_{2-q}^{\, \alpha_k}}{\sum_{r=1}^2 e_{2-q}^{\, \alpha_r}  }, \qquad \text{for $k=1,2$.}
\end{equation}
In fact, given the real numbers $\alpha_1$ and $\alpha_2$, we get that
$$\frac{ e_{2-q}^{\, \alpha_1}}{\sum_{r=1}^2 e_{2-q}^{\, \alpha_r}  }+ \frac{ e_{2-q}^{\, \alpha_2}}{\sum_{r=1}^2 e_{2-q}^{\, \alpha_r}  }=1,$$
and the summands are different if $\alpha_1\neq \alpha_2$.
Then consider the Jacobian $J: \{1,2\}^\mathbb{R} \to \mathbb{R}$, such that, it  is constant equal to $ \frac{ e_{2-q}^{\, \alpha_1}}{\sum_{r=1}^2 e_{2-q}^{\, \alpha_r}  }$ on the cylinder $\overline{1}$
and equal to $\frac{ e_{2-q}^{\, \alpha_2}}{\sum_{r=1}^2 e_{2-q}^{\, \alpha_r}  }$ on the cylinder $\overline{2},$ and take $a_1= -  \log_q (\frac{1}{ \frac{ e_{2-q}^{\, \alpha_1}}{\sum_{r=1}^2 e_{2-q}^{\, \alpha_r}  }})$ and
 $a_2= -  \log_q (\frac{1}{ \frac{ e_{2-q}^{\, \alpha_1}}{\sum_{r=1}^2 e_{2-q}^{\, \alpha_r}  }}).$
In case $q=3/2$,
one obtains
$$p_1 = \frac{(2 + \alpha_2)^2}{8 + 4 \alpha_1 + \alpha_1^2 + 4 \alpha_2 + \alpha_2^2}
\quad\text{and}\quad p_2 = \frac{(2 + \alpha_1)^2}{8 + 4 \alpha_1 + \alpha_1^2 + 4 \alpha_2 + \alpha_2^2}. $$
Taking derivative at $s=0$, we get
\begin{equation} \label{atro1}
\frac{d}{d s} P_{3/2}(A+ s\,B )|_{s=0} = \frac{1}{2} b_1  +  \frac{1}{2} b_2  \neq b_1 \,p_1 + b_2\, p_2 =\int B \,d \mathfrak{p}.
\end{equation}
\end{example}

\begin{remark}
    Expression~\eqref{atro1} shows that the derivative of the non-extensive pressure differs from the one in the extensive setting (cf.  Proposition 4.10 in \cite{PP}).
In the Appendix Section \ref{dpreq}  we consider for the case $q=1/2$ the  derivative of pressure for a more general class of potentials (see \eqref{loo7}).
\end{remark}

\bigskip
\subsection{Locally constant two-step potentials}

Now, we will consider the next level of complexity considering the Ruelle operator equation for  potentials  $A:\{1,2\}^{\mathbb N} \to \mathbb R$  which depend just on the   first  two coordinates, that is, 
\begin{equation} \label{kilo}A(x_1,x_2,x_3,..,x_n,...)=A(x_1,x_2):=a_{x_1\,x_2}\end{equation} which is constant and equal to $a_{ij}$ in each cylinder $\overline{i\,j}$, $i,j=1,2$.
The information of the potential $A$ is described by the matrix
\begin{equation} \label{nnh}
A= \left(
\begin{array}{cc}
a_{1\,1}  & a_{1\,2}    \\
a_{2\,1}  & a_{2\,2}
\end{array}
\right).
\end{equation}
Given $0<{q}<1$ we want to  find $\varphi$ and $c$,
which are solutions of
\begin{equation} \label{atch3325} \sum_{a=1}^2 e_{q}^{A(a x) + \varphi (a x) - \varphi(x) - c}=1,
\end{equation}
with $\varphi (x_0)=\alpha$, for fixed $x_0$ and $\alpha$ (we will refer to this setting as the Markov case). 

\smallskip
We will assume, without loss of generality, that $A(1,1)=a_{11}$ is the minimum of the entries of the matrix ~\eqref{nnh}. As a potential  $A$ admits  $\varphi$ and $c$ as solutions of \eqref{atch3325} is equivalent to say that
	$\hat{A} =A - A(1,1) $ has $\varphi$ and $ \hat{c} = c +A(1,1)$ as solutions of the same equation,  in order to solve \eqref{atch3325} 
    we will assume throughout that $A(1,1)=a_{11}=0.$ In this case all other entries of the above matrix are nonegative. 

We will show examples of potentials of the form \eqref{kilo} so that the solutions $\varphi$  are  constant  equal to $\varphi_1$ in the cylinder $\overline{1}$ and equal to  $\varphi_2$ in the cylinder $\overline{2}$. In this case,  it follows from Proposition \ref{esteaq1} that the corresponding non-extensive equilibrium probability measure will be a (classical) stationary Markov probability. Observe also that 
if $\varphi$ is a solution of \eqref{atch3325}, then adding a constant to $\varphi$ we also get a solution of \eqref{atch3325}.
Therefore, it is enough to search for a solution $\varphi$ which is equal to $\varphi_1=0$ in the cylinder $\overline{1}$ and equal to  $\varphi_2$ in the cylinder $\overline{2}$.  In this way the function $\varphi$ depends on the first coordinate in $\Omega$ and is determined by the value $\varphi_2$ (it is equal to zero in the cylinder $\overline{1}$).

\begin{remark}
    The assumption that all summands in \eqref{atch3325} are strictly positive is necessary for the connection with the existence of
positive Jacobians which are of great importance in classical Thermodynamic Formalism. This issue is clearly related to the claims presented in  Lemma \ref{analy} which in some way connects the non-extensive setting with the extensive setting.
\end{remark}

In consonance with the previous discussion we will  assume that all entries in the matrix
\begin{equation} \label{atch51}
\left(
\begin{array}{cc}
e^{ a_{11} - c}_{q}  & e^{ a_{12} - \varphi_2 -  c}_{q}     \\
e^{ a_{21} + \varphi_2 - c}_{q}    & e^{ a_{22} - c}_{q}
\end{array}
\right)
\end{equation}
are  positive numbers and that $a_{11}=0$, and consider the functions
\begin{equation} \label{atch5} f_1 (A, c, \varphi_2) = e^{ - c}_{q} +e^{ a_{21} + \varphi_2 - c}_{q} 
\quad\text{and}\quad
f_2 (A, c, \varphi_2) = e^{ a_{12} - \varphi_2 -  c}_{q}  +e^{ a_{22} - c}_{q}.
\end{equation}
The next example shows that the solutions $\varphi$ of \eqref{atch3325} need not be unique.

\begin{example}[\emph{Non-uniqueness of solutions}] \label{jana} Set $a_{11}=0=a_{22},$ $q=1/2$ in \eqref{atch5}. Assume  $\varphi$  is equal to $\varphi_1=0$ in the cylinder $\overline{1}$ and equal to  $\varphi_2$ in the cylinder $\overline{2}$. Take
	
\begin{equation}  \label{tem0}	c = 1/2 (4 + \sqrt{16 - a_{12}^2 + 2 a_{12} a_{21} - a_{21}^2)}
\end{equation}
	
\begin{equation}  \label{tem1}	\varphi_2  = -2 - a_{21} + c + \sqrt{4 c - c^2},\end{equation}
when $4\,c - c^2\geq 0$ and $16 - a_{12}^2 + 2 a_{12} a_{21} - a_{21}^2\geq 0$.
	One can show that    
	\begin{equation} \label{olk} f_1 (A, c, \varphi_2)=1 \,\, \text{and}\,\,     f_2 (A, c, \varphi_2)=1.
	\end{equation} 
	In this case we get $\varphi$ and $c$ which are explicit solutions for the $\tilde q=1/2$ non-extensive Ruelle Theorem.
	
	When $a_{12}=2$ and $a_{21}=3.5$ we get that $\varphi_2=-0.89595$ and $c=3.85405 $ solves \eqref{olk}.  In this way we get explicitly  $\varphi$ and $c$ solving \eqref{atch3325}. Surprisingly, if we take above
	  \begin{equation}  \label{tem12}	\varphi_2  = -2 - a_{21} + c - \sqrt{4 c - c^2}\end{equation} instead of \eqref{tem1}, we get $	\varphi_2=-2.39595 $ and $c=3.85405 $ solving \eqref{atch3325}, for $a_{12}=2$ and $a_{21}=3.5$. In both cases, all values $e_{q}^{A(a x) + \varphi (a x) - \varphi(x) - c}$ in \eqref{atch3325} are positive.

	  Therefore, it is possible to get two different eigenfunctions associated to two different eigenvalues which satisfy the property that all summands in \eqref{atch3325} are positive. This shows that the non-extensive analogous of the Ruelle theorem displays issues of a much more complex nature.

\end{example}

The next example will show that there are examples of potentials $A$ which admit solutions for \eqref{atch3325}  where one of the summands is equal to zero. This is a further evidence that 
results analogous to the classical Ruelle's Theorem,
about the existence of eigenfunctions and eigenvalues in the non-extensive setting becomes a much more complex matter.

Now we present explicit solutions for the eigenfunction problem covering different possibilities.

\begin{example} \label{explimeq} 
Take $0<q<1$ and $a_{11}=0=a_{21}$. We will exhibit explicit formulas for  solutions $c$ and $\varphi$ for the system
$$
\begin{cases}
\begin{array}{l}
  1 = e^{  - c}_{q} +e^{   \varphi_2 - c}_{q}     \\
      1 = e^{ a_{12} - \varphi_2 -  c}_{q}  +e^{ a_{22} - c}_{q}.
\end{array}    
\end{cases}
$$
where all summands are
 positive. This will  provide solutions $\varphi$  for  \eqref{ato1}.
 Denote $q_1= e^{ - c}_{q} $ and $q_2= e^{ a_{22} - c}_{q}$, where we assume that the parameters are such that
$q_1,q_2 \in (0,1)$.
 In terms of the values $q_1$, $q_2$, we  get the solutions
   \begin{equation} \label{vo1} c = - \log_q (q_1),
   \end{equation}
   \begin{equation} \label{vo2}\varphi_2 =  \log_q (1 - q_1)  + c,
  \end{equation}
  \begin{equation} \label{vo3} a_{12} =  \log_q ( 1 - q_2 ) + \varphi_2 + c,
    \end{equation}
and
 \begin{equation} \label{vo4} a_{22} =  \log_q (q_2) + c.
 \end{equation}
The eigenfunction $\varphi$ is equal to zero in the cylinder $\overline{1}$ and is equal to $\varphi_2$ in the cylinder $\overline{2}.$
For example, when $q=2/3$, $a_{12}=0.857533$, $a_{22}=0.52199$, we get the $q$-eigenvalue $c=0.991701$ and the $q$-eigenfunction $\varphi$, is such that $\varphi_2=0.655413$. This case corresponds to $q_1=0.3$ and $q_2=0.6$.
When $q=4/5$, taking  $q_1=0.2$ and $q_2=0.3$, we get $a_{12}= 2.18972$, $a_{22}=0.306117$, $\varphi_2=1.15786$ and $c=1.3761$. In this way we get another example where we can find an explicit solution $\varphi$ and $c$.
Furthermore, 
 when $q=1/2$,  $a_{11}=0=a_{21},$   $a_{12}= - 5.67332$, $a_{22}= -3.09545$, one has that $c=0$, and  $\varphi_2 =-2$ provide solutions $\varphi$  for \eqref{ato1}, and one has that $e_{1/2}^{\varphi-c}=0.$

 \end{example}


\section{Appendix A: The derivative of the dynamical $q$-pressure} \label{dpreq}

Recall that for each $0<q<1$ and every Lipschitz continuous potential $A: \Omega \to \mathbb{R}$, the $q$-pressure of $A$ is defined by
\begin{equation*} \label{PPP2} P_q (A) =  \sup_{\mu \in \mathcal{G} }\,\Big[  H_q (\mu)+ \int A(x) \,d \mu (x) \Big] ,
\end{equation*}
where $H_q (\mu)$ is the $q$-entropy  of $\mu\in\mathcal{G}$ defined in \eqref{HQ}.

From Remark \ref{rmk.existencea}, for fixed $0<q<1$, we know that $P_q(A)$ is differentiable in $A$. We consider in this appendix infinitesimal variations of a normalized potential $A$ (for the differentiability of the pressure function on the extensive case see e.g. \cite{BCV,Lall2,Lop}).

\medskip
  Given the Lipschitz continuous function   $v: \Omega \to \mathbb{R}$ we want to compute
$  
\frac{d}{ds}P_q( A + s v)|_{s=0}
$ 
for some $q\neq 1$.
We will prove the following lemma.

\begin{lemma}
    \label{le:diff-q12}
Consider a Lipschitz continuous Jacobian  $J$ and the associated equilibrium state $\mu$ and let $A$ be the normalized potential $A=-\log_q (\frac{1}{J})$. 
Assume that $v$ is a Lipschitz continuous potential and assume that there exist $(\varphi_s, c_s)$ such that 
\begin{equation}
    \label{eq:cohoms}
\sum_{a=1}^d e_{2-q }^{( A(ax) + s v(ax) ) + [\varphi_s(a x) - ( \varphi_s (x)] - c_s }=1 
\end{equation}  
for every $x\in \Omega$ and every small $s$. If $q=\frac12$ then 
\begin{align}
   \frac{d}{ds}P_q( A + s v)|_{s=0}
   & = \frac{\int  J(x)^{1/2}  v( x) d \mu  +  \int J(x)^{1/2} (\frac{d}{ds} \varphi_s(x) - \frac{d}{ds}    \varphi_s (\sigma(x)))|_{s=0}  \; d\mu}{\int J(x)^{1/2}  d \mu(x)} .
   \label{loo7} 
\end{align}
\end{lemma}

\begin{proof}
Fix $q=1/2$ and let $A,v:\Omega \to \mathbb R$ be Lipschitz continuous as in the statement.
As $e_{2-q}^A=J$ (recall ~\eqref{prop16}) then
$ 
\sum_{a=1}^d e_{2-q }^{A(ax)} = \sum_{a=1}^d J(a x)=1.
$ 
In consequence, 
\begin{equation} \label{loo4}   \int \sum_{a=1}^d J (ax) f(a x) d \mu = \int \sum_{a=1}^d e_{2-q }^{A(ax)} f(a x) d \mu = \int f d \mu
\end{equation}
for every continuous function $f$ and 
the $q$-pressure of $A$ is $P_q (A)=0$
(cf. Proposition~\ref{esteaq1}).
Moreover, it is not hard to check that 
\begin{equation} \label{loo1}  e_{2-q}^{f(s) + g(s) - h(s)} = \Big[1 - \frac{1}{2} (f(s) +g(s)  -h(s))\Big]^{-2}
\end{equation}
and, using the formula for the derivative of the $(2-q)$-exp map in \eqref{prop17},
\begin{equation} \label{loo2} \frac{d}{ ds} e_{2-q}^{f(s) + g(s) - h(s)} = \Big[1 - \frac{1}{2} (f(s) +g(s) - h(s))\Big]^{-3} \cdot \Big[  f '(s) + g '(s) - h '(s)\Big].
\end{equation}

Now we note that if ~\eqref{eq:cohoms} holds then $c_s=P_q( A + s v).$
Therefore, differentiating both sides of \eqref{eq:cohoms} with respect to the parameter $s$ and using \eqref{loo1}  and \eqref{loo2}
\begin{align*}
   0 & = \frac{d}{ds} \sum_{a=1}^d e_{2-q }^{( A(ax) + s v(ax) ) + [\varphi_s(a x) -  \varphi_s (x)] - c_s }|_{s=0} \\
   & = 
   \sum_{a=1}^d  \Big( 1+ \frac{1}{2} \log_q \frac{1}{ J(a x)}\Big)^{-3}\, \Big[v(a x)  +
(\frac{d}{ds} \varphi_s(a x) - \frac{d}{ds}    \varphi_s (x))|_{s=0} -
\frac{d}{ds}c_s|_{s=0} \,\Big]
\end{align*}
As $(1+ \frac{1}{2} \log_q (1/y)\, )^{-1}= y^{1/2}$ for any $y>0$ one can write
$$  \Big( 1+ \frac{1}{2} \log_q \frac{1}{ J(a x)}\Big)^{-3} = J(a x)^{3/2}$$
and, consequently, 
$$ \sum_{a=1}^d J(a x) 
\; \Big[ J(a x)^{1/2}  v(a x) +  J(a x)^{1/2} (\frac{d}{ds} \varphi_s(a x) - \frac{d}{ds}    \varphi_s (x))|_{s=0} -  J(a x)^{1/2}
\frac{d}{ds}c_s|_{s=0} \Big]=0
$$
for every $x\in \Omega.$ Furthermore, using that 
$e^{A}_{2-q}=J$,
\begin{equation} \label{loo3}
 \sum_{a=1}^d e^{A(ax)}_{2-q} [ J(a x)^{1/2}  v(a x) +  J(a x)^{1/2} (\frac{d}{ds} \varphi_s(a x) - \frac{d}{ds}    \varphi_s (x))|_{s=0}  -  J(a x)^{1/2} \frac{d}{ds} c_s|_{s=0} ]=0.
 \end{equation}
 Integrating with respect to $\mu$  together with \eqref{loo4} and \eqref{loo3} one deduces that
 \begin{align*}
  0 & =\int  \sum_{a=1}^d e^{A(ax)}_{2-q} [ J(a x)^{1/2}  v(a x) +  J(a x)^{1/2} (\frac{d}{ds} \varphi_s(a x) - \frac{d}{ds}    \varphi_s (x))|_{s=0}  -  J(a x)^{1/2} \frac{d}{ds} c_s|_{s=0} ] d \mu(x)    \\
  & 
  = \int  J(x)^{1/2}  v( x) +  J(x)^{1/2} (\frac{d}{ds} \varphi_s(x) - \frac{d}{ds}    \varphi_s (\sigma(x)))|_{s=0}  -  J(x)^{1/2} \frac{d}{ds} c_s|_{s=0}  \; d \mu(x) \\
  & \int  J(x)^{1/2}  v( x) d \mu  +  \int J(x)^{1/2} (\frac{d}{ds} \varphi_s(x) - \frac{d}{ds}    \varphi_s (\sigma(x)))|_{s=0}  d\mu - \frac{d}{ds}  c_s|_{s=0}\int J(x)^{1/2}  d \mu(x).
 \end{align*}
This finishes the proof of the lemma.
\end{proof}

 \section{Appendix B: Shannon approach to the non-extensive measure theoretic entropy} \label{sec:shannon}
 
 In this subsection we shall introduce an alternative notion of non-extensive entropy for all probability measures in $\mathcal M(\sigma)$, which is
 a natural extension of the concept introduced by M\'eson and Vericat \cite{MeVe} for Bernoulli probability measures.
 Fix $0<q<1$.
 Given a finite partition $\cP$ of a compact metric space $X$, define
 \begin{equation}\label{Kolmogorov-Tsallis}
 \mathfrak{H}_q (\mu,\cP) = \frac1{1-q} \Big( \sum_{P\in \cP} \mu(P)^q -1 \Big)
 = \frac1{1-q} \Big( e^{\log \big(\sum_{P\in \cP} \mu(P)^q\big)} -1 \Big),
 \end{equation}
 which, using ~\eqref{T11} one can write alternatively 
 $$
 \mathfrak{H}_q (\mu,\cP) = \sum_{P\in \cP} \mu(P) \log_q \Big( \frac{1}{\mu(P)} \Big)
 $$
 Lemma~\ref{tos} ensures that
 $
 \mathfrak{H}_q (\mu,\cP) \leqslant \log_q(d)
 $
 and that its maximal value is attained in the case of the equidistributed Bernoulli probability measure.
 Moreover, if $\cP^{(n)}=\bigvee_{j=0}^{n-1} \sigma^{-j}(\cP)$ denotes the dynamically defined partition, 
 $
 \mathfrak{H}_q (\mu,\cP^{(n)})
 $
 need not be sub-additive for $q\neq 1$, due to the properties of $\log_q(\cdot)$.
 Using ~\eqref{Kolmogorov-Tsallis} we obtain
 \begin{equation}\label{Kolmogorov-Tsallis-resc}
 \log [\,1+ (1-q) \, \mathfrak{H}_q (\mu,\cP)\,] = \log \big(\sum_{P\in \cP} \mu(P)^q\big).
 \end{equation}
 Inspired by \eqref{Kolmogorov-Tsallis-resc}, M\'eson and Vericat define the \emph{$\mathfrak{h}_q$-entropy of $\mu$}
 by
 $$
 \mathfrak{h}_q (\mu)= \sup_{\cP}\; \mathfrak{h}_q (\mu,\cP)
 $$
 where the supremum is taken over all finite partitions of $X$ and
 \begin{equation}\label{Kolmogorov-Tsallis-resc2}
 \mathfrak{h}_q (\mu,\cP) =\limsup_{n\to\infty} \frac1n \log [\,1+ (1-q) \, \mathfrak{H}_q (\mu,\cP^{(n)})\,].
 \end{equation}

 We will need some auxiliary results.
 
 \begin{lemma}\label{auxle}
 	The function $(p_1, p_2, \dots, p_n ) \mapsto \log \big(\sum_{i=1}^d p_i^q\big)$ defined on the space of probability vectors
 	$p=(p_1, p_2, \dots, p_n )$, and attains the maximum value $\frac1{1-q}\log (n)$ at the equidistributed probability vector.
 \end{lemma}
 
 \begin{proof} In order to make use Lagrange multipliers consider the function
 	$$
 	p=(p_1, p_2, \dots, p_n ) \to \log \,\Big(\sum_{j=1}^d p_j^{q}\Big) - \alpha \sum_{j=1}^d  \, p_j
 	$$
 	where $\alpha $ is a constant.
 	Taking derivative in $p_i$, $i=1,2,...,n$, we look for the condition
 	\begin{equation}  \label{louc1a1}
 	\frac{q\, p_i^{q-1}}{\sum_{j=1}^d p_j^{q}}  - \alpha=0,
 	\end{equation}
 	which implies that
 	$$
 	{\, p_i} = \Big[\alpha \frac1q \sum_{j=1}^d p_j^{q}\Big]^{\frac1{q-1}}, \quad \forall  i=1,2,...,n.
 	$$
 	This, together with the assumption that $\sum_{i=1}^d p_i=1$, this implies that $p_i=\frac1n$, that $p$ is the equidistributed
 	probability vector and that the maximum value is
 	$\log \big(\sum_{i=1}^d n^{-q}\big)$. This proves the lemma.
 \end{proof}
 
 One could ask whether the usual notion of topological entropy admits a natural counterpart in the non-extensive framework.
 Given a finite open covering $\mathcal U$ of $X$ we write
 $N_q(\cU)=\log_q(\# \cU)$. A simple modification of the arguments in \cite[Section 7.1]{Walters} and using
 property ~\eqref{goodeq}  of  $q$-log functions we conclude that
 $$
 N_q(\cU \vee \mathcal V) = \log_q(\cU \vee \mathcal V)
 = N_q(\cU) + N_q( \mathcal V) + (1-q)\,N_q(\cU) N_q( \mathcal V)
 $$
 for every finite open coverings $\cU, \mathcal V$ of $X$. In consequence, denoting by $\cU^{(n)}=\bigvee_{j=0}^{n-1} \sigma^{-j}(\cU)$ the
 dynamically defined open covering of $\cU$, one concludes that
 \begin{align*}
 N_q(\cU^{n+m}) & = \log_q\Big(\bigvee_{j=0}^{n-1} \sigma^{-j}(\cU) \vee \bigvee_{j=0}^{m-1} \sigma^{-j}(\sigma^{-n}(\cU))\Big) \\
 & \leqslant N_q(\cU^{(n)}) + N_q(\cU^{(m)})+ (1-q)\,N_q(\cU^{(n)}) N_q( \mathcal U^{(m)})
 \end{align*}
 and so the sequence $(N_q(\cU^{n}))_{n\geqslant 1}$ is not sub-additive for $q\neq 1$.
 In particular a notion of $q$-topological entropy should be defined differently.
 
 Inspired by the previous discussion and Lemma~\ref{auxle}, we define the $q$-\emph{topological entropy of $\sigma$} by 
 $$
 \mathfrak{h}_{top,q} (\sigma) =\sup_{\mathcal U} \frac1{1-q} \limsup_{n\to\infty} \frac1n N(\mathcal U^{(n)}),
 $$
 where $N(\mathcal U^{(n)})$ denotes the smallest cardinality of an open subcover of $\cU^{(n)}$.

 \smallskip

 \section{Appendix C - Renyi entropy} \label{renyi}

 \begin{definition}
 	\label{def:q-Renyi}
 	The \emph{$q$-Renyi entropy} for $p=(p_1,p_2,...,p_d)$ is defined as
 	\begin{equation} \label{R1} H^R_q (p) = \frac{\log (\sum_{j=1}^d \,p_j^q)}{1-q} .
 	\end{equation}    
 \end{definition}

 Given $0<q<1$, one can show that $H^R_q (p)=F(H_q (p))$ for every probability vector $p=(p_1,p_2,...,p_n)$, where $F$ stands for the monotone increasing map $F: \mathbb R_+ \to \mathbb R_+$ given by
 \begin{equation} \label{F1} x \mapsto F(x)= \frac{\log (1 + (1-q)\, x) }{1-q}.
 \end{equation}

 A version of the $q$-entropy, when considering probability measures defined on sets that are not finite  is presented in (1.28) in \cite{UmaTsa} and also in  (14) in \cite{CS}, but our dynamical definition of $q$-entropy has different features when compared with the ones in  these two references.

 In this way, the understanding of properties of $H_q$  can help in the understanding of the Renyi entropy $H^R_q.$
 
 \smallskip

 For the MaxEnt method, Tsallis et al  \cite{Tsa3}  considered two forms of internal energy constraints. In \cite{Abe} a third choice was considered. Our main focus is on the pressure problem and not on the MaxEnt method.
 
 It is also true that
 \begin{equation} \label{rt15}
 \sum_j q_j \log_q (\frac{1}{p_j}) - \sum_j q_j \log_q (\frac{1}{q_j})\geqslant 0.
 \end{equation}
 and that the maximal value of the left-hand side of \eqref{rt15} is equal to $d$.

\section{Appendix D: Basic properties of $q$-exp and $q$-log functions} \label{apen}

\subsection*{$q$-log functions}
Under the  non-extensive point of view it is natural to consider the  $q$-log function (see \cite{Nau,Tsa0,Tsa}), defined by
\begin{equation} \label{T1} u \to \log_q (u) =\frac{1}{1-q} ( u^{1-q}-1),
\end{equation}
where $q\neq 1$ and $u>0$.
We observe that $\log_q(1)=0$, that $\log(x)\geq \log_q(x)$, when $q\geq 1$, that $\log(x)<\log_q(x)$, when $q<1$, and that one recovers the classical function $\log$ taking the limit of  $\log_q $ as $q \to 1$.
See Figure \ref{fig6} for the graph of $\log_q$. 

\begin{figure}[h!]
	\centering
	\includegraphics[scale=0.31]{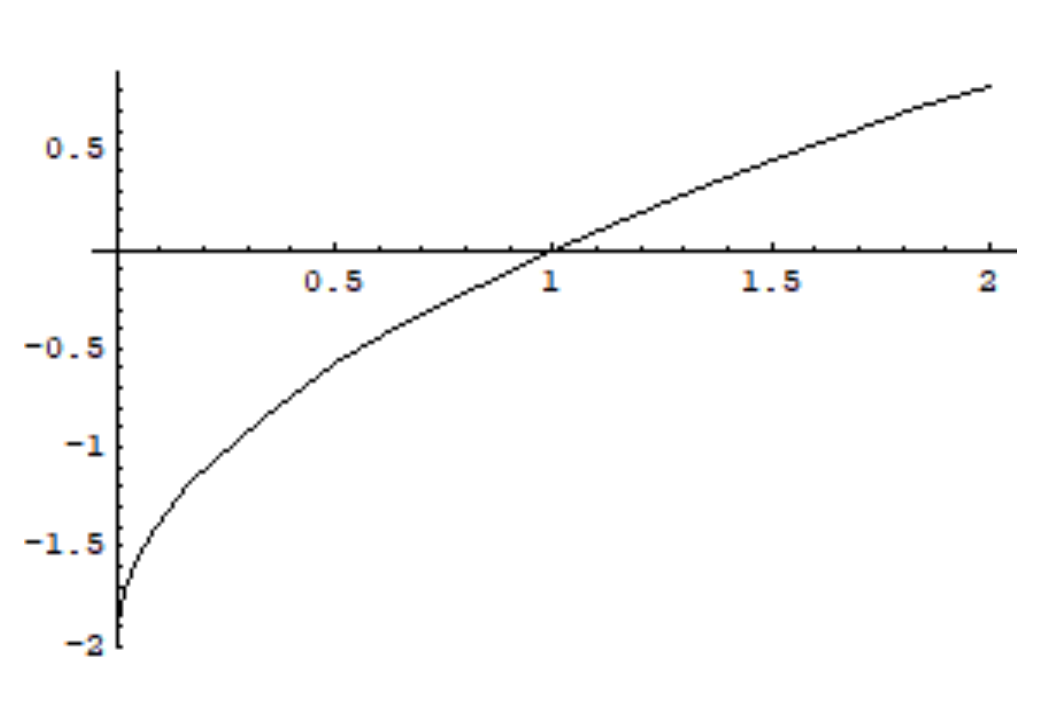}
	\qquad
	\includegraphics[scale=0.31]{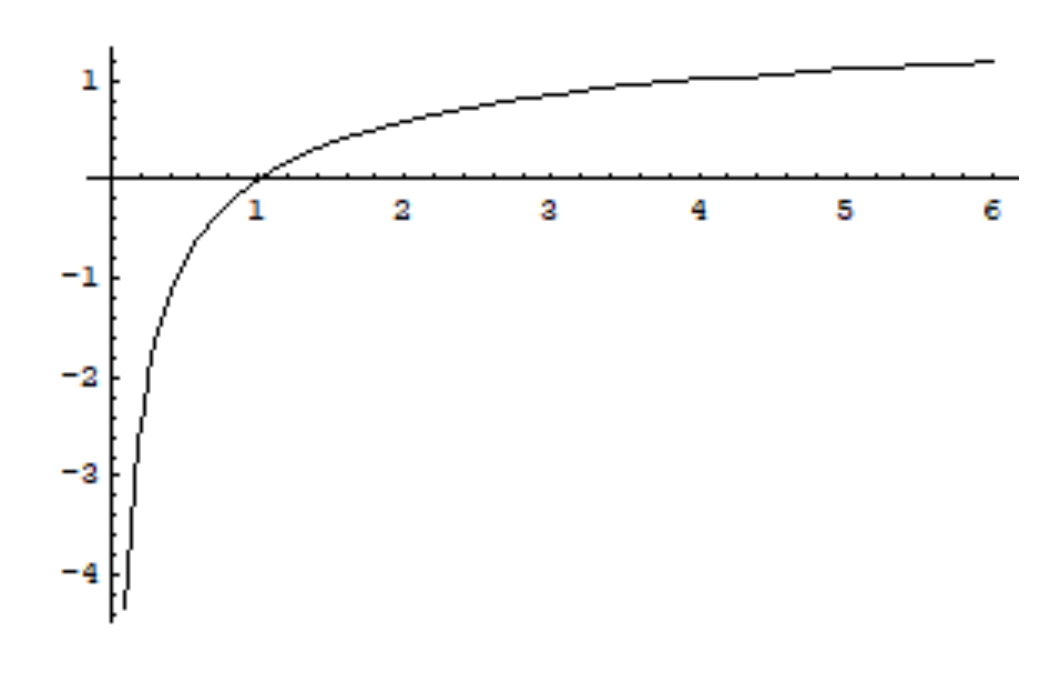}
	\caption{Graph of the function $\log_q$ in case $q=0.5$ (left) and $q=-0.5$ (right). The domain is $(0,\infty)$.}
	\label{fig6}
\end{figure}
Moreover, for each $0<q<1$ the function $u \to \log_q (u)$ is concave, $\log_q (u)<0$ for every $0<u<1$, and it satisfies
\begin{equation} \label{goodeq}   \log_q (a b) = \log_q (a) + \log_q (b) + (1-q)\, \log_q (a)\,\,\log_q (b)
\end{equation}
and
\begin{equation} \label{goodeq1}   \log_q \Big(\frac{1}{p}\Big)  = - p^{q -1} \log_q (p).
\end{equation}

For $0<q<1$, $0<x<1$, the function $- \log x > \frac{1}{1- q}(x^{q-1} -1)$; therefore the equilibrium-Boltzman (nondynamical) entropy $ h(p)=- \sum_j p_j \log p_j $ satisfies
\begin{equation} \label{hH21} h(p) \geqslant H_q(p).
\end{equation}
In the case $q>1$, given $0<x<1$ one has  $- \log x < \frac{1}{1- q}(x^{q-1} -1)$, hence 
\begin{equation} \label{hH212} h(p) \leqslant H_q(p).
\end{equation}
When $q>0$, the largest value of the $q$-entropy map is attained on the uniformly distributed probability vector $p_0=(1/n,1/n,..,1/n)$ and it is equal to
$$H_q (p_0)= \frac{1}{1-q} (n^{1-q} -1) = \log_q (n).$$

\begin{example}[$q$-entropy of Markov measures] \label{maio}
	A two by two line stochastic matrix $P$ with all positive entries  $P= (P_{i,j})_{i,j=1,2}$  determines a unique stationary Markov measure $\mu$ on
	$\{1,2\}^\mathbb{N}$. As $P_{11} + P_{21}=1= P_{12} + P_{22}$, these probability measures $\mu$ are indexed by $P_{12},P_{2,1}\in (0,1) \times (0,1)$. The vector of probability $\pi=\binom{\pi_1}{\pi_2}$ is the one such that $\pi P =\pi$. The probability of the cylinder  set $\overline{ij}$ is $\pi_i P_{i,j}$.    For a stationary Markov measure $\mu$, the Jacobian $J$ in the cylinder $\overline{ij}$ is equal to $Q_{ij}=P_{ij} \frac{\pi_i}{\pi_j}$. The classical entropy of $\mu$ is given by
	$h(\mu)=-\sum_{i,j=1}^2 \pi_i P_{ij} \log P_{ij}.$ In the non-extensive case, for $q>0$, we get the concave function
	$$ H_q(\mu) =   \sum_{i,j=1}^2 \pi_i P_{ij}   \log_q (\frac{1}{Q_{ij}}).$$
	Figure \ref{figqent} below shows the
	graph of the values of the concave  function $H_q(\mu) $, as a function of $P_{12},P_{2,1}\in (0,1) \times (0,1)$, when $q=0.9$. We point out that for values $q>1$ the function $H_q(\mu) $ is also concave.
	\begin{figure}[h!]
		\centering
		\includegraphics[scale=0.45]{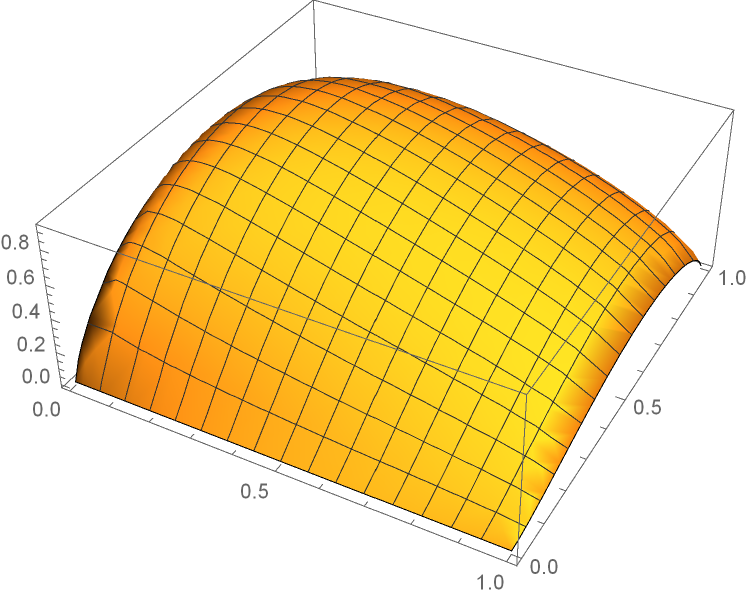}
		\caption{A Markov stationary probability is determined by the values $P_{12}, P_{21}$ of a line stochastic matrix $P$. Above the graph of the $q$-entropy  $H_q(\mu)$, when $q=0.9$, as a function of $(P_{12},P_{21})$. The domain of the concave function is $(0,1)\times (0,1)$}
		\label{figqent}
	\end{figure}
\end{example}

\begin{remark}
	Given two probability measures $p=(p_1,p_2,...,p_n),q=(q_1,q_2,...,n_n)$, by Jensen inequality
	\begin{equation*} \label{rt} \sum_j q_j \log_q (\frac{p_j}{q_j}) \leqslant  \log_q ( \sum_j q_j \frac{p_j}{q_j}  )= \log_q (1)=0.
	\end{equation*}
	In particular, the probability vectors $p_i=(0, \dots 0, 1, 0, \dots, 0)$ have zero $p$-entropy.  
\end{remark}

\subsection*{$q$-exp functions}

For $q>0$, with $q\neq 1$, the inverse of $\log_q$ is the $q$-exp function. This concept is necessary for the definition of the $q$-Ruelle operator.
Recall that the $q$-exp function is defined by
\begin{equation*} \label{T1} u \to  e_q^u=\exp_q (u) =  (1  + (1-q) u) ^{\frac{1}{ 1-q}},
\end{equation*}
for every $q\neq 1$ and $u>0$, 
that it is convex and $\exp_q (0)=1$ and the image of the function $\exp_q $ is $(0,\infty)$.  For the graph of  $\exp_q $ see  Figure \ref{fig11}.

\begin{figure}[h!]
	\centering
	\includegraphics[scale=0.25]{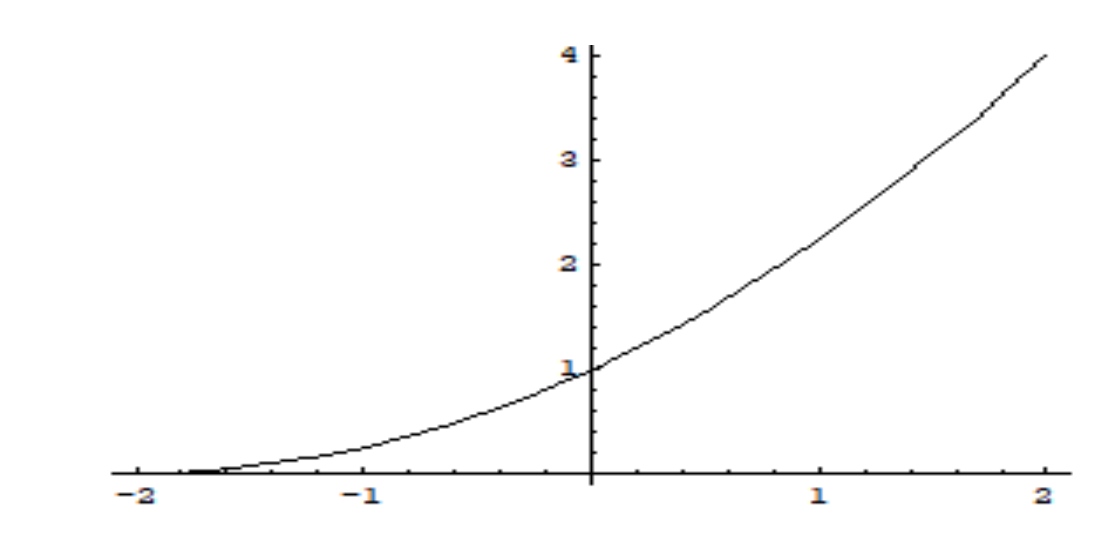}
	\qquad
	\includegraphics[scale=0.25]{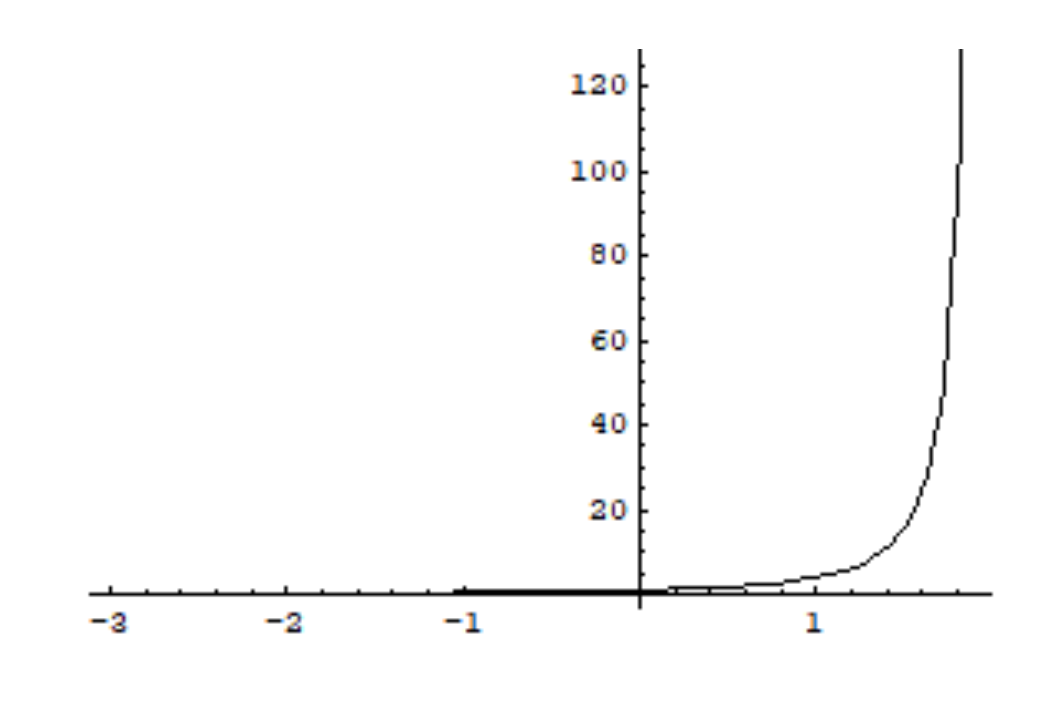}
	\caption{Graph of $\exp_q$. On the left $q=0.5$ and its domain is $(-2,\infty)$. On the right $q=-0.5$ and its domain is $(-\infty,2)$.}
	\label{fig11}
\end{figure}
Moreover, function $e_q^x$ is the  solution of the ordinary differential equation $\frac{d y(x)}{d x} = y(x)^q $ with initial condition $y(0)=1$, and 
\begin{equation} \label{T111}
\exp_q ( - \log_q (1/x))\leqslant x
\end{equation} 
and it is not the identity map.

\begin{remark} \label{kw} The following holds:
	\begin{enumerate}
		\item If $q>1$: $\exp_q (u)>0$ for every $u<\frac{1}{q-1}$;
		\item If $0<q<1$: $\exp_q (u)>0$ for every
		$u<\frac{1}{q-1}$
	\end{enumerate}
	
	When $q<1$ and $q$ close to $1$, we get that $\exp_q (u)>0$ for very negative  values of $u$.
\end{remark}

\subsection{Some properties}
In what follows we list some of the relevant properties of the $q$-exponential functions, some of which used in the text,
which illustrate the origin of the non-additivity of the non-extensive entropy function.
\begin{align}
\label{prop1} 
& a) \,\,\, e^z_q\,\, e^{-z}_{2-q} =
(e^z_q)^q  \,\,e^{-q\,z}_{\frac{1}{q}}=1  \\
& 
\label{prop2} b) \,\,\, e^{x + y + (1-q) x\,y}_{q} =
e^x_q \, \,e^y_q\,\, \\
& \label{prop2a} c) \,\,\, e_q^{a+b} - e_q^a =-\,[1 + a (1 - q)]^{\frac{1}{1-q}} \,+\,[1 + (a + b) (1 - q)]^{\frac{1}{1 - q}} \\
&
\label{prop2b} d) \,\,\, \frac{e_q^{a+b}}{e_q^a}=[1 + a (1 - q)]^{-\frac{1}{1-q}} \,[1 + (a + b) (1 - q)]^{\frac{1}{1 - q}} \\
& 
\label{prop3} e) \,\,\,
(e^x_q)^{-1} \,-  \,e^{-x}_q \geqslant 0, \,\,\text{if}\,\, x>1,\,\,\text{and}\,\,(e^x_q)^{-1} \,-  \,e^{-x}_q \leqslant 0, \,\,\text{if}\,\, x<1 \\
& 
\label{prop4} f)\,\,\, \frac{d}{d x} \log_q(x) = \frac{1}{x^q} \\
& 
\label{prop5} g)\,\,\, \frac{d}{d x} e_q^x =(e_q^x)^q  \hspace{11cm} {\color{white} .}\\
&
\label{prop6} h)\,\,\,
P_j= \frac{p_j^q}{\sum_{k=1}^d p_k^q} \Leftrightarrow p_i^q= \frac{P_i}{(\sum_{j=1}^d P_j^{1/q})^q} \\
& 
\label{prop7} i)\,\,\,
P_j= \frac{p_j}{\sum_{k=1}^d p_k} \Leftrightarrow p_i= \frac{P_i}{(\sum_{j=1}^d P_j^{1/q})^q}
\\
& \label{prop8} j)\,\,\,
(e_q^x)^a = e^{ a\, x}_{1- \frac{1-q}{a}}, \;\;\text{for any $a$ and $q$}\\
& \label{prop9} k)\,\,\,
\log_q (x) = \frac{(1 + (1-m) \log_m (x))^{\frac{1- q }{1-m   } }}{   1- q },\;\; \text{for any $m$ and $q$} \\
& \label{prop10} l)\,\,\, \text{(Taylor expansion)}\,\,\,\,
e_q^x = 1 +\sum_{n=1}^\infty  \frac{1}{n\, !} \,Q_{n-1} (q) \, x^n
\;\; \text{for any $q>0$} \\ 
& \nonumber \text{  where $Q_{n-1} (q)= q\, (2 q -1) (3 q -2) ...[n\,q - (n-1)]$}  \hspace{5.6cm}{\color{white} .}\\
&
\label{prop10} m)\,\,\, \text{(Taylor expansion)}\,\,\,\,
\log_q (1 +x)= x  +\sum_{n=2}^\infty   (-1)^{n+1  } \frac{1}{n\, !}\, \Pi_{j=0} ^{n-2}\,(q + j)  \,x^n \\
& \nonumber \text{for any $q>0$} \\
&
\label{prop11} n)\,\,\,\, \text{(First derivative)}\,\,\,\,
\frac{ d\,(e_q^{\alpha+ \beta}- e_q^\alpha) }{d \beta} =[1+(\alpha+\beta)(1-q) ]^{-1+\frac{1}{1-q}}, \; \; q>0 \\
& 
\label{prop12} o)\,\,\,\,\, \text{(Second derivative)}\,\,\,\,
\frac{ d^2\,(e_q^{\alpha+ \beta}- e_q^\alpha) }{d^2 \beta} = (-1+\frac{1}{1-q})\,(1-q)\,[1+(\alpha+\beta)(1-q)]^{-2+\frac{1}{1-q}} \\
& \nonumber \text{for any $q>0$} \\
& 
\label{prop13} p)\,\,\,\,\,  \,\,\,\,
e_q^a \, e_{2-q}^b = e_{2-q}^{  -2 (-1 + \frac{|b-2|}{|2+a|} ) } \\
& 
\label{prop14} q)\,\,\,\,\,  \,\,\,\,
- \log_q (\frac{1}{e_{2-q}^y})= \frac{-1 + (1 + y\, (q-1))^{-1} }{q-1 }, \\
& 
\label{prop15} r)\,\,\,\,\,  \,\,\,\,
e^{x+y}_{2-q} = e_q^{\frac{-1 + e^{y (q-1) } (1 + (q-1) (x + y))^{-1}}{1 - q}}\, e^y, \\
& 
\label{prop16} s)\,\,\,\,\,  \,\,\,\,
e_{2-q}^{ - \log_q (1/y)}=y \\
& 
\label{prop17} t)\,\,\,\,\,  \,\,\,\,
\frac{d}{ ds} e_{2-q}^{ f(s)   } = (1 + (q-1) (f(s)  ))^{\frac{ 2-q}{q-1} } ( f^\prime (s)    )
\end{align}

\subsection*{Acknowledgements}
We would like to thank L. Cioletti  for providing us with references \cite{Tsa1,Tsa0}, and encouraging us to study this topic.
AOL was partially supported by CNPq-Brasil.
 PV was partially supported by CIDMA, through FCT -- Funda\c c\~ao para a Ci\^encia e a Tecnologia, I.P., under the projects with references 
 UID/04106/2025 (https://doi.org/\-10.54499/UID/04106/2025)
 and UID/PRR/04106/2025
(https://doi.org/\-10.54499/\-UID/\-PRR/04106/2025).

\medskip


\medskip

\smallskip

\noindent Artur O. Lopes\\ Instituto de Matem\'atica e Estat\'istica - UFRGS \\
91509-900 Porto Alegre, Brazil.\\
E-mail: arturoscar.lopes@gmail.com

 \medskip

\noindent 
Paulo Varandas \\ Departamento de Matem\'atica
da Universidade de Aveiro \\
Campus Universitário de Santiago \\ 3810-193 Aveiro, Portugal  \\
E-mail: paulo.varandas@ua.pt


\begin{thebibliography}{99}
	
\bibitem{Abe}	S. Abe, \emph{Remark on the escort distribution representation of non-extensive
statistical mechanics}, Physics Letters A 275 (2000)  250--253.

\bibitem{Abe2}	S. Abe,  \emph{Axioms and uniqueness theorem for Tsallis entropy}, Physics Letters A 271 (2000) 74–-79.

\bibitem{Abe3}	S. Abe,
\emph{Temperature of non-extensive systems: Tsallis entropy as Clausius entropy},
Phys. A 368:2 (2006) 430--434.

\bibitem{ABH}
J. M. Amig\'o,  S. G. Balogh and S. Hernandez, \emph{A Brief Review of Generalized Entropies}, Entropy, 20:11 (2018) 813.


\bibitem{An} J. An and K-H Neeb,
\emph{An implicit function theorem for Banach
spaces and some applications}, Math. Z. 262 (2009) 627.


\bibitem{Bala} V. Baladi, \emph{Positive Transfer Operators and decay of correlations}, World Scientific (2000)

\bibitem{BalaS} V. Baladi and  D. Smania,
\emph{Linear response formula for piecewise expanding unimodal maps}, Nonlinearity 21:4 (2007)  677.

\bibitem{BalaT} V. Baladi and M. Todd,
\emph{Linear response for intermittent maps},
Comm. Math. Phys. 347:3 (2016) 857--874.

\bibitem{BHVZ}
T. Bomfim, R. Huo, P. Varandas and Y. Zhao,
\emph{Typical properties of ergodic optimization for asymptotically additive potentials,} {Stoch. Dynam.}  23: 1 (2023) 2250024.

\bibitem{Bow}
R. Bowen,
\emph{Gibbs States and the Ergodic Theory of Anosov Diffeomorphisms}, Lecture notes in Math., volume 470, Springer-Verlag, (1975)

\bibitem{Barr0} L. Barreira,
\emph{Thermodynamic Formalism
and Applications
to Dimension Theory}, Ed. Birkhauser (2011)

\bibitem{Barre} L. Barreira,
\emph{Ergodic Theory, Hyperbolic Dynamics and Dimension Theory}, Springer Verlag (2012)


 \bibitem{BLT} A. Baraviera, A. O. Lopes and Ph. Thieullen,
\emph{A Large Deviation Principle for equilibrium states of H\"older potentials: the zero temperature case.} { Stoch. and  Dyn.}\, (6), 77-96, (2006).

\bibitem{BecS} C. Beck and F. Schogl,
\emph{Thermodynamics of Chaotic Systems: An Introduction}, Cambridge Univ. Press (1993)

\bibitem{BCMV}  A. B\'is, M. Carvalho, M. Mendes and P. Varandas, \emph{A convex analysis approach to entropy functions, variational principles and equilibrium states}, {Comm. Math. Phys.} 394 (2022) 215--256.
Correction: Commun. Math. Phys. 401  (2023)
3335--3342.

\bibitem{BCV}T. Bomfim, A. Castro and P. Varandas, \emph{Differentiability of thermodynamic quantities in
non-uniformly expanding dynamics}, Adv. Math. 292 (2016) 478--528.

\bibitem{CFH}
Y. Cao, D. Feng and W. Huang. \emph{The thermodynamic formalism for sub-additive potentials.} Discrete Cont.
Dyn. Syst. 20 (2008) 639--657.



\bibitem{Cat} A. Caticha,
\emph{Entropic Physics: Lectures on Probability, Entropy and Statistical Physics}, Preprint version (2021)
	

\bibitem{CDLS} L. Cioletti, M. Denker, A. O. Lopes and M. Stadlbauer,
Spectral Properties of the Ruelle Operator for Product Type Potentials on Shift Spaces,
J. London Math. Soc., 95:2 (2017) 684--704.

\bibitem{CLS}  L. Cioletti, A. O. Lopes and  M. Stadlbauert, \emph{Ruelle Operator for Continuous Potentials and DLR-Gibbs Measures},
Discrete Cont. Dyn. Sys. 40:8  (2020)
4625--4652.

\bibitem{CL1}  L. Cioletti and A. O. Lopes,
\emph{Phase Transitions in One-dimensional Translation Invariant Systems: a Ruelle Operator Approach,}
J. Stat. Phys., 159:6 (2015) 1424--1455.


\bibitem{CL3} L. Cioletti and A. O. Lopes,  \emph{Correlation Inequalities and Monotonicity Properties of the Ruelle Operator,} Stoch. Dyn., 19:6 (2019) 1950048.


\bibitem{CS}  L. Cioletti,  and E. Silva,  \emph{Spectral properties of the Ruelle operator on the Walters class over compact spaces}. Nonlinearity 29:8  (2016)
2253--2278.

\bibitem{CS} A. Creaco and N. Kalogeropoulos,	\emph{Power-law entropies for continuous systems and
generalized operations},
Modern Physics Letters B, 32:14 (2018) 1850338.
	
\bibitem{CLO} G. Contreras, A.O. Lopes and E. Oliveira, 	\emph{Ergodic Transport Theory, periodic maximizing probability measures and the twist condition,
Modeling, Optimization}, Dynamics and Bioeconomy I, Springer Proceedings in Mathematics and Statistics, Volume 73, Edit. David Zilberman and Alberto Pinto, 183-219 (2014)
	
\bibitem{Curado} 	 E. M. F Curado and C. Tsallis, \emph{Generalized statistical mechanics: connection with
thermodynamics},  J. Phys. A
24 L69 (1991); Corrigenda: 24, 3187 (1991) and 25, 1019 (1992)

\bibitem{Curado1} E. M. F Curado  and A. Plastino,
\emph{Information theory link between MaxEnt and a key
thermodynamic relation}, Physica A 386 155–-166 (2007)


\bibitem{DWY} M. F. Demers, P. Wright and L.S. Young,
\emph{Entropy, Lyapunov exponents and escape rates in open systems,}
Ergod. Th. Dynam. Sys., 32:4 (2012)  1270--1301.

\bibitem{DF} N. Dunford and T. Schwartz, Linear Operators  Part I. Interscience (1958)


\bibitem{Falconer}	K. J. Falconer
\emph{A sub-additive thermodynamic formalism for mixing
repellers,}
J. Phys. A  21 (1988) L737.
	
\bibitem{Fiorenza} R. Fiorenza,
\emph{Lipschitz and locally H\"older Continuous Functions, and Open Sets of Class $C^k$, $C^{k,\lambda}$,} Frontiers in Mathematics, Springer Verlag, 2016.

\bibitem{FV} R. Freire and V. Vargas. \emph{Equilibrium states and zero temperature limit
on topologically transitive countable Markov shifts.} Trans. Amer. Math.
Soc., 370:12 (2018) 8451--8465.

\bibitem{FL} A. Fisher and A. O. Lopes,
\emph{Exact bounds for the polynomial decay of
correlation, $1/f$ noise and the central limit theorem for a
non-H\"older Potential},
Nonlinearity, 14:5 (2001) 1071--1104.

\bibitem{HV}  N. Haydn and S. Vaienti,
\emph{The Renyi entropy function and the large deviation of short return times.} Ergodic Th. Dynam. Sys. 30:1 (2010) 159--179.

\bibitem{HaCha} J. Havrda and F. Charvat, Quantification Method of Classification Processes, Kibernetica Cislo Rocknic 3, 30-34 (1967)
Processes

\bibitem{HYH} J. Huang, W. Yong and  L. Hong,
\emph{Generalization of the Kullback–-Leibler divergence in the Tsallis
statistics,} J. Math. Analysis App., 436:1 (2016)  501--512.


\bibitem{Keller}
G. Keller,
\emph{Equilibrium states in ergodic theory,}
Cambridge University Press, 1998.


\bibitem{Kif}
Y. Kifer, Large Deviations in Dynamical Systems and Stochastic processes, TAMS, Vol 321, N.2, 505-–524 (1990)

\bibitem{KLOE}
B. R. Kloeckner,
\emph{Effective perturbation theory
for linear operators, }
J. Operator Theory 81:1 (2019) 175--194.

\bibitem{KW}
T. Kucherenko and C. Wolf,
\emph{Localized pressure and equilibrium states,} J. Stat. Phys.,
160:6  (2015) 1529--1544.


\bibitem{GKLM} P. Giulietti, B. Kloeckner,  A. Lopes and D. Marcon,
 \emph{The calculus of thermodynamic formalism}, J. Eur. Math. Soc., 20:10 (2018) 2357--2412.


\bibitem{GLM} D. Gomes, A. O. Lopes and J. Mohr,
\emph{The Mather measure and a Large Deviation Principle for the Entropy Penalized Method}, Comm. Contemporary Math., 13:2 (2011) 235--268.

\bibitem{GL} S. Gouezel and C. Liverani,
\emph{Banach spaces adapted to Anosov systems}, Ergodic Th. Dynam. Sys. 26, (2006) 189--217.

\bibitem{Lall2} S. Lalley,
\emph{Ruelle's Perron-Frobenius Theorem and the Central Limit Theorem for Additive Functionals of One-Dimensional equilibrium States,}
 Lecture Notes-Monograph Series, Vol. 8, Adaptive Statistical Procedures and Related Topics, pp. 428--446 (1986).

 \bibitem{Lang} S. Lang,  \emph{Fundamentals of Differential Geometry}, Springer Verlag, 1999.




   \bibitem{Lep1}
R. Leplaideur,  \emph{Chaos: butterflies also generate phase transitions.} J. Stat. Phys. 161:1 (2015) 151--170.


 \bibitem{RW}
 R. Leplaideur and F. Watbled,
 \emph{Curie-Weiss Type Models for General Spin Spaces and
Quadratic Pressure in Ergodic Theory}, J. Stat. Phys., 181 (2020) 263-–292.



\bibitem{Lo1} A. O. Lopes,
\emph{The Zeta Function, Non-Differentiability
of Pressure and The Critical Exponent of Transition}, Adv. Math., 101
(1993) 133--167.

\bibitem{LPT} A. O. Lopes,
\emph{Dimension Spectra and a Mathematical Model for Phase Transitions},
Adv. Appl. Math. 11 (1990)  475--502.

\bibitem{L3} A. O. Lopes,
\emph{Entropy and Large Deviation,} Nonlinearity, 3:2
(1990) 527--546.


\bibitem{LopR} A. O. Lopes
\emph{A general renormalization procedure on the one-dimensional lattice and decay of correlations}, Stoch. Dyn., 19:1  (2019)  1950003.




 \bibitem{LOS} A. O. Lopes,  E. R. Oliveira and D. Smania,
 \emph{Ergodic Transport Theory and Piecewise Analytic Subactions for Analytic Dynamics,}
Bull. Braz. Math Soc., 43:3 (2012) 467--512.


\bibitem{LOT}
A. O. Lopes, E. Oliveira and Ph. Thieullen,
\emph{The Dual Potential, the involution kernel and Transport in Ergodic Optimization},
on ``Dynamics, Games and Science'' - International Conference and Advanced School Planet Earth DGS II, Edit. J-P Bourguignon, R. Jelstch, A. Pinto and M. Viana, Springer Verlag, pp 357-398 (2015)


\bibitem{LR} A. O. Lopes and R. Ruggiero,
\emph{Nonequilibrium in Thermodynamic Formalism: the Second Law, gases and Information Geometry}, Qual. Th. Dynam. Sys. 21 (2022) 1--44.


\bibitem{LR1} A. O. Lopes and R. Ruggiero, \emph{The sectional curvature of the infinite dimensional manifold of H\"older equilibrium states},  Proc. of the Edinburg Matematical Soc. Volume 68 - Issue 2 - 348-402 (2025)

\bibitem{LR2} A. O. Lopes and R. Ruggiero, \emph{
Geodesics and dynamical information projections on the manifold of H\"older equilibrium probability measures}  Stoch. and Dynamics, Vol. 25, No. 03n04, 2550022 (2025)




\bibitem{LMMS} A. O. Lopes, J. K. Mengue, J. Mohr and  R. R. Souza,
\emph{Entropy and Variational Principle for  one-dimensional Lattice Systems with a general a-priori probability: positive and zero temperature},  Ergodic Th. Dynam. Sys. 35:6 (2015) 1925--1961.


\bibitem{LM1} A. O. Lopes and J.  Mengue,
\emph{On information gain, Kullback-Leibler divergence, entropy production and the involution kernel}, Discrete Cont. Dynam. Sys 42:7 (2022) 3593--3627.

\bibitem{Lop} A. O. Lopes,
\emph{Thermodynamic Formalism, Maximizing probability measures and Large Deviations}, Lecture notes (online).

\bibitem{MeVe} A. M\'eson and F. Vericat,
\emph{On the Kolmogorov-like generalization of
Tsallis entropy, correlation entropies and
multifractal analysis}, J. Math. Phys. 43, 904 (2002)


\bibitem{Nau}  J. Naudts, \emph{Generalised Thermostatistics}, Springer Verlag, 2011.


\bibitem{PP}
W. Parry and M.  Pollicott.
\emph{Zeta functions and the periodic
orbit structure of hyperbolic dynamics,} {Ast\'erisque}
Vol {187-188} (1990).

\bibitem{Poll}
M.  Pollicott,
\emph{A complex Ruelle-Perron-Frobenius theorem and two counterexamples},
Ergodic Th. Dynam. Sys., 4:1 (1984) 135--146.



\bibitem{Ros}
 P. Rosenbloom, \emph{Perturbation of Linear operators in Banach Spaces}, Arch.
Math. 6 (1955) 89--101.

\bibitem{Rue} D. Ruelle, \emph{Thermodynamic Formalism}, Cambridge University Press (2010)

	\bibitem{Sa}  T. Sagawa,
\emph{Entropy, Divergence, and Majorization in Classical and Quantum Thermodynamics},
Volume 16 of SpringerBriefs in Mathematical Physics, Springer Nature, 2022.

\bibitem{Sar} O. M. Sarig,	\emph{Thermodynamic formalism for countable Markov shifts}, Lecture Notes.

\bibitem{Steele}
 M. J. Steele, \emph{Probability theory and combinatorial optimization}, SIAM, 1997.

\bibitem{UmaTsa} S. Umarov and C. Tsallis,
\emph{Mathematical Foundations of non-extensive Statistical Mechanics}, World Scientific, 2022.



	
\bibitem{Tsa1} C. Tsallis,	
\emph{Possible Generalization of
Boltzmann--equilibrium Statistics}, J. Stat. Phys., 52:1/2 (1988) 479--487.
	
\bibitem{Tsa0} C. Tsallis,
\emph{Beyond Boltzmann-equilibrium-Shannon in Physics
and Elsewhere}, Entropy, 21, 696 (2019)

\bibitem{Tsa} C. Tsallis,
\emph{Introduction to non-extensive Statistical Mechanics}, Springer, New York, (2009)


\bibitem{Oka} C. Tsallis,
\emph{Non-extensive Statistical Mechanics and Thermodynamics:
Historical Background and Present Status}, in Non-extensive
Statistical Mechanics and Its Applications, Lect. Notes in Physics, editors
S. Abe  and Y. Okamoto, Springer Verlag  (2001)

\bibitem{Tsa3} C.Tsallis,  R. S. Mendes and A.R. Plastino,
\emph{The role of constraints within generalized
non-extensive statistics}, Physica A 261,  (1998) 534--554.



\bibitem{TaVe}
F. Takens and E. Verbitski,
\emph{Generalized entropies: Renyi and correlation integral
approach}, Nonlinearity 11, 771–782 (1998)


\bibitem{VZ}
P. Varandas and Y. Zhao,
\emph{Weak Gibbs measures: speed of convergence to entropy, topological and geometrical
aspects}, Ergodic Th. Dynam. Sys. 37, no. 7, 2313-–2336 (2017)

\bibitem{Walters} P. Walters,
\emph{An introduction to Ergodic Theory,} Springer Verlag (1982)





\bibitem{Yam}
T. Yamano, \emph{Some properties of q-logarithm and q-exponential functions in Tsallis statistics}. Phys. A 305, no. 3--4, 486–-496  (2002)

\end{thebibliography}
\end{document}